\documentclass{amsart}
\usepackage[dvipdfm]{graphicx}
\usepackage{verbatim}
\usepackage{amssymb}
\usepackage{amsbsy}
\usepackage{amscd}
\usepackage{amsmath}
\usepackage{amsthm}
\usepackage{amsxtra}
\usepackage{latexsym}
\usepackage[mathscr]{eucal}
\usepackage{tikz}
\usepackage[colorlinks]{hyperref}

\usetikzlibrary{arrows}

\newcommand{\fink}{\mathfrak{k}}

\newcommand{\Hnew}{H_{\new}}
\newcommand{\x}{x_0}

\newcommand{\Sl}{\mc{S}}
\newcommand{\Ass}[1]{X_{#1}}
\newcommand{\Ring}[1]{{R_{#1}}}

\newcommand{\coker}{\on{coker}}
\newcommand{\Nil}{\mc{N}}

\newcommand{\BRS}[2]{H^{\semiinf+#1}_{f}(#2)}

\newcommand{\Weyl}[1]{\mathbf{V}_{#1}}
\newcommand{\Fneu}{\mathscr{F}^{{\affchi}}}

\newcommand{\affchi}{\wh\chi}

\newcommand{\wh}{\widehat}

\newcommand{\mc}{\mathcal}
\newcommand{\mf}{\mathfrak}
\newcommand{\mb}{\mathbb}

\newcommand{\on}{\operatorname}
\newcommand{\affP}{\widehat{P}}

\newcommand{\KL}{\mathsf{KL}}

\newcommand{\Vg}[1]{V^{#1}(\fing)}

\newcommand{\A}{\mathscr{A}}

\newcommand{\Prp}{Adm_+}

\newcommand{\finb}{\mathfrak{b}}
\newcommand{\finn}{\mathfrak{n}}

\newcommand{\isomap}{{\;\stackrel{_\sim}{\to}\;}}

\newcommand{\W}{\mathscr{W}}

\newcommand{\Cat}{\mathcal{C}}

\newcommand{\nc}{\newcommand}
\nc{\Hp}[1]{H^{#1}}
\newcommand{\V}{V}
\newcommand{\prin}{\mathrm{prin}}

\newcommand{\affh}{\widehat{\mathfrak{h}}}

\newcommand{\affg}{\widehat{\mathfrak{g}}}
\newcommand{\bigaffg}{\widetilde{\mathfrak{g}}}
\newcommand{\fing}{\mathfrak{g}}
\newcommand{\finh}{\mathfrak{h}}
\newcommand{\finm}{\mathfrak{m}}

\newcommand{\Wg}[1]{\W^{#1}(\fing, f)}

\newcommand{\bh}{\widehat{\mathfrak{h}}}

\newcommand{\Lamsemi}[1]{\bigwedge\nolimits^{\frac{\infty}{2}+#1}}

\newcommand{\Irr}[1]{\mathbf{L}_{#1}}

\newcommand{\BGG}{{\mathcal O}}

\newcommand{\N}{\mathbb{N}}
\newcommand{\Q}{\mathbb{Q}}
\newcommand{\1}{{\mathbf{1}}}
\newcommand{\teigi}{:=}

\newcommand{\dual}[1]{{#1}^*}
\newcommand{\bra}{{\langle}}
\newcommand{\ket}{{\rangle}}
\newcommand{\iso}{\overset{\sim}{\rightarrow}}
\newcommand{\roots}{\Delta}
\newcommand{\nno}{\nonumber}
\newcommand{\Lam}{\Lambda}

\newcommand{\lam}{\lambda}
\newcommand{\ra}{\rightarrow}
\newcommand{\+}{\mathop{\oplus}}
\newcommand{\Z}{\mathbb{Z}}
\newcommand{\Mod}{\text{-}\mathrm{Mod}}

\newcommand{\cprime}{$'$}
\newcommand{\inv}{^{-1}}

\renewcommand{\*}{{\otimes}}
\newcommand{\C}{\mathbb{C}}
\newcommand{\che}{^{\vee}}

\newcommand{\finp}{{\mathfrak{p}}}
\newcommand{\finr}{{\mathfrak{r}}}

\theoremstyle{plain}
\newtheorem{Th}{Theorem}[subsection]

\newtheorem{Pro}[Th]{Proposition}

\newtheorem{Lem}[Th]{Lemma}

\newtheorem{Co}[Th]{Corollary}

\theoremstyle{definition}

\theoremstyle{remark}
\newtheorem{Def}[Th]{Definition}
\newtheorem{Rem}[Th]{Remark}

\newcommand{\affW}{\widehat{W}}

\newcommand{\semiinf}{\frac{\infty}{2}}

\DeclareMathOperator{\im}{Im}

\DeclareMathOperator{\rank}{rk}
\DeclareMathOperator{\ch}{ch}

\DeclareMathOperator{\id}{id}

\DeclareMathOperator{\End}{End}
\DeclareMathOperator{\gr}{gr}
\DeclareMathOperator{\Hom}{Hom}

\DeclareMathOperator{\new}{new}

\DeclareMathOperator{\ad}{ad}
\DeclareMathOperator{\Ad}{Ad}
\DeclareMathOperator{\haru}{span}

\DeclareMathOperator{\Spec}{Spec}

\title[Associated varieties
and  $W$-algebras]
{Associated varieties of
modules over   Kac-Moody algebras 
and  $C_2$-cofiniteness of $W$-algebras}
\author{Tomoyuki Arakawa}

\address{
Research Institute for Mathematical Sciences, Kyoto University, Kyoto
606-8502 JAPAN}

\email{arakawa@kurims.kyoto-u.ac.jp}

\thanks{This work was partially  supported 
by the JSPS Grant-in-Aid  for Scientific Research (B)
No.\ 20340007}

\begin{document}
\maketitle

\begin{abstract}
First, we establish the relation 
between  the associated varieties of modules over Kac-Moody algebras
$\affg$
and those  over affine $W$-algebras.
Second, 
we prove the Feigin-Frenkel conjecture  on
 the singular supports
of $G$-integrable admissible representations.
In fact
we show that
the associated varieties 
of
  $G$-integrable
 admissible
 representations
are irreducible $\Ad G$-invariant  subvarieties
of the nullcone of $\fing$,
by
determining them
explicitly.
Third, 
we prove
the $C_2$-cofiniteness
of a large number of simple $W$-algebras,
including all  the minimal series principal $W$-algebras \cite{FKW92}
and
the 
exceptional $W$-algebras 
recently discovered by Kac-Wakimoto \cite{KacWak08}.
\end{abstract}

\section{Introduction}
This article addresses 
some basic  problems 
concerning the representation theory
of Kac-Moody algebras, that of (affine)
$W$-algebras
and the interrelation  between them.
It has three  aims.

Let $\fing$ be a complex simple Lie algebra,
$\affg$ the non-twisted affine Kac-Moody algebra
associated with $\fing$.
For a nilpotent element $f$ of $\fing$
and $k\in \C$,
both the    $W$-algebra  $\W^k(\fing,f)$
at level
$k$ 
and  representations
of $\Wg{k}$
are constructed 
by means of the 
 BRST cohomology 
functor $\BRS{0}{?}$ 
associated with the generalized Drinfeld-Sokolov reduction \cite{FF90,FKW92,KacRoaWak03}.

The first aim of this article   is to
establish the relation 
between
the {\em associated varieties} 
(\cite{Ara12}, see \eqref{eq:def-of-ass.v})
of 
modules over  $\affg$ and those over $\Wg{k}$.
More precisely,
we show that
\begin{align}
X_{\BRS{0}{M}}\cong X_M\cap \Sl
\label{eq:Main}
\end{align}
for
a finitely generated  graded  Harish-Chandra
($\affg,G[[t]]$)-module
$M$ of level $k$,
where
$X_M$
is the  associated variety of $M$
and $\Sl$  is the  Slodowy slice 
at $f$ to $\Ad G.f$
(Theorem
\ref{Th:BRST-reduction-of-varieties}).
From (\ref{eq:Main})
it follows 
that
the $\Wg{k}$-module $\BRS{0}{M}$ is {\em $C_2$-cofinite}\footnote{%
$C_2$-cofinite representations of vertex algebras
may be regarded as analogue of
finite-dimensional representations (\cite{Ara12})
.} 
\cite{Zhu96}
if 
$X_M$ equals 
 the closure 
of the orbit $\Ad G.f $.
This result  may  be viewed  as a chiralization of 
a theorem   \cite[Theorem 3.1]{Pre07}
of Premet
on finite $W$-algebras.

The second is 
to determine the associated varieties  
of {\em (Kac-Wakimoto) admissible representations}\footnote{Admissible
representations
of vertex operators algebras have nothing to do with
(Kac-Wakimoto) admissible representations of $\affg$.} \cite{KacWak89}
of $\affg$
at rational levels.
We 
prove  the    Feigin-Frenkel  conjecture
which states
that
 the singular supports of 
$G$-integrable admissible representations 
are contained in the jet scheme of the nullcone $\mc{N}$ of $\fing$,
or equivalently,
the associated varieties of 
$G$-integrable admissible representations 
are contained $\mc{N}$
(Theorem \ref{Th:Conj:Feigin-Frenkel}).
Furthermore we 
show that
 those associated varieties
 are $\Ad G$-invariant {\em irreducible} subvarieties   of $\Nil$,
that is,
closures of some nilpotent orbits in $\fing$
(Theorem
\ref{Th:main-admissible} and Theorem \ref{Th:Main-orbit1}).
These nilpotent orbits  depend only on  
the level $k$
of the representations.
Thus
for each  admissible 
rational number $k$
there exists a unique nilpotent orbit 
$\mb{O}[k]$ in $\fing$ such that
\begin{align*}
 X_{
\Irr{\lam}}=\overline{\mb{O}[k]}
\end{align*}
for any $G$-integrable irreducible admissible representation 
$\Irr{\lam}$ of level $k$.
(Actually the orbit $\mb{O}[k]$
depends only on the denominator of $k\in \Q$.)
The orbits $\mb{O}[k]$ are 
explicitly determined and listed in Tables
\ref{table:classical-principal}--\ref{table:E8}
(cf.\ \eqref{eq:O[k]}).

The third  
is to apply the above results 
to the $C_2$-cofiniteness problem 
of $W$-algebras. 
The $C_2$-cofiniteness condition \cite{Zhu96} is a certain finiteness condition
on a vertex operator algebra
which ensures the coherency of 
the associated conformal blocks
on any Riemann surface.
It also coincides with the
lisse condition  of Beilinson, Feigin and Mazur
\cite{BeiFeiMaz}
(see \cite{Ara12}).
 The minimal series 
$W$-algebras associated principal nilpotent elements
have been expected to be rational and $C_2$-cofinite
since its discovery \cite{FKW92}.
Furthermore,
a remarkable family of $W$-algebras,
called the {\em exceptional $W$-algebras},
which generalizes the minimal series 
principal 
$W$-algebras,
has been  recently discovered by Kac and Wakimoto \cite{KacWak08}. 
They conjectured\footnote{To be precise they conjectured that
the conformal filed theories associated with exceptional 
$W$-algebras are rational.} that exceptional $W$-algebras are  rational and
$C_2$-cofinite.
We prove the $C_2$-cofiniteness 
of a large number of $W$-algebras
including
 {\em all} the exceptional
$W$-algebras 
(Theorem \ref{Th:exceptionals-are-C2}).
More precisely,  we prove 
 that,
for each  admissible number $k$, 
the simple quotient\footnote{Conjecturally \cite{FKW92,KacRoaWak03},
$\W_k(\fing,f)=\BRS{0}{\Irr{k\Lam_0}}$.
We have proved this
in \cite{Ara05} for a minimal nilpotent element $f$
and in \cite{Ara07}
for a principal nilpotent
element $f$ under some regularity condition on $k$.
In type $A$ 
one can show this
for any nilpotent element
using the results of \cite{Ara08-a}
under some regularity condition on $k$.
The detail will appear elsewhere.
} 
 $\W_k(\fing,f)$ of $\Wg{k}$ with $f\in \mb{O}[k]$ is $C_2$-cofinite
(Theorem \ref{Th:C_2-cofiniteness-of-modules-over-W-algebras}),
and that each 
 exceptional $W$-algebra is isomorphic to
$\W_k(\fing,f)$ for some  admissible number $k$ and $f\in
\mb{O}[k]$.
We note that
there are also
  a considerable number of 
$C_2$-cofinite $W$-algebras 
which  are {\em not}
exceptional, see Tables \ref{table:classical-principal},
\ref{Oq-for-G_2},~\ref{Oq-for-F4},
\ref{Oq-for-E6},
\ref{Oq-for-E7}, \ref{table:E8}.
We conjecture that 
all the 
$C_2$-cofinite $W$-algebras 
appearing in this way
are rational.

\smallskip
Our strategy to prove (\ref{eq:Main})
is based on   Ginzburg's  reproof \cite{Gin08}
 of 
Premet's conjecture
\cite[Conjecture 3.2]{Pre07} (proved by Losev \cite{Los07})
on finite $W$-algebras.
A ``chiralization'' of the argument of \cite{Gin08}
proves the vanishing of the BRST cohomology of 
the associated graded spaces (Theorem \ref{Th:BRST-vanishing-for-assicaited-graded}).
The difficult part is the proof of the convergency
of the corresponding spectral sequence
because  our algebras are not 
Noetherian.
We overcome this problem by using the 
right exactness of the functor $\BRS{0}{?}$ (Theorem
      \ref{Th:right-exactness}, cf.\ \cite{AraMal}),
see \S \ref{subsection:strong-vanishing} for the detail.
As a result we obtain the strong vanishing assertion (Theorem
      \ref{Th:strong-vanishing}) of the BRST cohomology,
which gives 
 (\ref{eq:Main}) as desired.

Note that the vanishing of the BRST cohomology 
proves the exactness of the functor 
\begin{align*}
\BRS{0}{?}:\KL_k\ra \Wg{k}\Mod
\end{align*}
as well
(Theorem \ref{Th:vanishing-new}),
where
 $\KL_k$ is the 
category of graded Harish-Chandra ($\affg$, $G[[t]]$)-modules
of level $k$,
or equivalently,
 the
full subcategory of the category $\BGG$ of $\affg$
consisting  $G$-integrable representations
of level $k$.
This generalizes some of the exactness results
obtained  in \cite{Ara04,Ara05,Ara07,FreGai07}.

\smallskip
For a  finitely generated object $M$ of $\KL_k$,
the associated variety  $X_M$ is an $\Ad G$-invariant,  conic,
Poisson
  subvariety of 
$\fing^*$,
while the variety
 $X_{\BRS{0}{M}}$
 of the $\Wg{k}$-module $\BRS{0}{M}$
 is a $\C^*$-invariant,
Poisson subvariety of the
Slodowy slice $\Sl$.
Because $f$ is the unique fixed point of the $\C^*$-action of $\Sl$,
from (\ref{eq:Main}) it follows that
\begin{align}
 \BRS{0}{M}\ne 0\iff \overline{\Ad G.f}\subset X_M
\label{eq:how-to-determine-the-variety}
\end{align}
(compare \cite[Theorem 2]{Mat87},
\cite[Corollary 4.1.6]{Gin08}).

Now on the contrary to the cases of
associated varieties
of primitive ideals of $U(\fing)$,
the variety $X_M$ need not be contained in
the nullcone $\Nil$.
Nevertheless
it was conjectured by Feigin and Frenkel that
 this is the case for 
$G$-integrable {\em admissible} representation
of $\affg$.
We prove the conjecture 
of Feigin and Frenkel
by reducing to the $\mf{sl}_2$-cases
(which is known by Feigin and Malikov \cite{FeiMal97})
using the fact  \cite{A-BGG}  that
admissibility is preserved 
under the {\em semi-infinite restriction} 
(Theorem \ref{Th:Conj:Feigin-Frenkel}).

The character  
of an admissible representation is given by the Weyl-Kac type character
formula \cite{KacWak88},
and therefore
 the  character of $\BRS{0}{\Irr{\lam}}$  is computable 
\cite{FKW92,KacRoaWak03}
by  the vanishing of the BRST cohomology
and the Euler-Poincar\'{e} principle.
Hence from the Feigin-Frenkel conjecture
and (\ref{eq:how-to-determine-the-variety})
we see that the associated variety 
of an $G$-integrable admissible representation $\Irr{\lam}$
can be  determined by
knowing for which nilpotent $f$
the cohomology
$\BRS{0}{\Irr{\lam}}$ is nonzero.
As a result 
we have obtained the following description of 
$X_{\Irr{\lam}}$:
Let $k$ be an admissible rational number with denominator $u$.
Then 
for a $G$-integrable irreducible
admissible 
representation $\Irr{\lam}$
we have
\begin{align*}
X_{\Irr{\lam}}\cong
\begin{cases}
\{x\in \fing; (\ad x)^{2u}=0\}&\text{if
      }(r\che,u)=1,\\
\{x\in \fing; \pi_{\theta_s}(x)^{2u/r\che}=0\}&\text{if }(r\che,u)=r\che,
 \end{cases}
\end{align*}
where
$r\che$ is the lacing number of $\fing$,
$\theta_s$ is the highest short root of $\fing$
and $\pi_{\theta_s}$ is the irreducible finite-dimensional
representation of $\fing$
with highest weight $\theta_s$
(Theorem \ref{Th:main-admissible}).
The irreducibility
of the above variety
is known
in most cases \cite{Geo04}
and in
the other cases
it is checked by
 the classification of
nilpotent elements
and their closure relations (Theorem \ref{Th:Main-orbit1}).

An exceptional  $W$-algebra is  by definition
the  simple $W$-algebra
$\W_k(\fing,f)$
at a  principal admissible level\footnote{That is,
$\Irr{k\Lam_0}$ is  principal admissible.} $k$
such that
  $(u,f)$ is an {\em exceptional pair}
 \cite{KacWak08},
where $u$ is the denominator of $k$.
The exceptional pairs were 
 classified in \cite{KacWak08,ElaKacVin08}.
From the classification 
it follows  that 
$(u,f)$ is 
an exceptional pair if and only of
(i) $\overline{\Ad G.f}=X_{\Irr{k\Lam_0}}$
and
(ii)
$f$ is 
of principal type (Theorem \ref{Th:Kac-Wakimoto}).
This fact together the above explained results 
proves
the $C_2$-cofiniteness of the 
exceptional $W$-algebras.

\subsection*{Acknowledgments}
I would like to thank Fyodor Malikov for explaining 
me the results of \cite{FeiMal97}.
Results in this article were presented in part in
``Algebraic Lie Structures with Origins in Physics'',
Isaac Newton Institute for Mathematical Sciences, Cambridge, March
2009,
``Workshop and Summer School on Lie Theory and Representation Theory
II'',
East China Normal University,
July 2009,
``The 9th Workshop on Nilpotent Orbits and Representation Theory'',
Hokkaido, February 2010.
I would also like to thank 
the organizers of these conferences.

\subsection*{Notation}
Throughout this paper the ground field is the complex number 
$\C$
and
 tensor products and dimensions are always meant to be as vector spaces
over $\C$
 if not otherwise stated.
For basis notions on the theory of vertex algebras we refer the reader
to \cite{Kac98}
and \cite{FreBen04}.

\section{Associated graded vertex Poisson algebras
and their vertex Poisson modules}
In \S \ref{subsection:Function-on-jet-schemes} --
 \S \ref{subsection:affine vertex algebras}
we recall some basic facts on vertex algebras
 and also some results from \cite{Ara12}.
In \S \ref{subsection:Kac-Moody} we recall some fundamental facts on
Kac-Moody algebras.
\subsection{Functions on jet schemes
of affine Poisson varieties as vertex Poisson algebras}
\label{subsection:Function-on-jet-schemes}
For a scheme $X$ of finite type,
let
 $X_m$  the $m$-th jet scheme 
of $X$,
$X_{\infty}$
the infinite jet scheme
of $X$
(or the arc space of $X$).

Let us recall the definition of jet schemes.
For general theory of jet schemes see e.g., \cite{EinMus,Fre07}.
The scheme $X_m$ is  determined by its functor of points:
for every commutative  $\C$-algebra $A$,
there is  a bijection
\begin{align*}
 \Hom(\Spec A, X_m)\cong \Hom(\Spec A[t]/(t^{m+1}),X).
\end{align*}
If $m>n$, we have projections
$ X_m\ra X_n$.
This yields  a projective system
$\{X_m\}$ of schemes,
and the infinite jet scheme $X_{\infty}$
is the projective limit
$\lim\limits_{\underset{m}{\leftarrow}}
X_{m}$
in the category of schemes.

For an affine scheme $X=\Spec R$,
the jet scheme $X_m$ is explicitly described.
Choose a presentation  
$R=\C[x^1,\dots, x^r]/\bra f_1,\dots,f_s\ket$.
Define a new variables $x^j_{(-i)}$ for $i=1,\dots,m+1$,
and a derivation $T$ of the
 ring
\begin{align*}
 \C[x^j_{(-i)};i=1,2,\dots,m+1,\ 
 j=1,\dots,r
]
\end{align*}
 by setting
\begin{align*}
 T x^j_{(-i)}=\begin{cases}
	    i x^j_{(-i-1)}&\text{for $i\leq m$},\\
0&\text{for $i=m+1$}.
	   \end{cases}
\end{align*}

Let us  consider $f_i$ as an element of 
$\C[x^j_{(-i)}; j=1,\dots,r,i=0,1,\dots,m]$
by 
 identifying  $x^j$ with $x^j_{(-1)}$,
and 
set
\begin{align*}
 R_m:=\C[x^j_{(-i)}; i=1,\dots,m+1,\ j=1,\dots,r
]
/
\bra T^j f_i;
i=1,\dots, s,\
j=0,\dots,m+1\ket,
\end{align*}
where
$f_i$  is considered
as 
an element of 
$ \C[x^j_{(-i)};i=1,2,\dots,m+1,\ 
 j=1,\dots,r
]$
Then we have
\begin{align*}
X_m = \Spec R_m.
\end{align*}

Let $R_{\infty}$ denotes the differential algebra
obtained from $R_m$ by taking the limit $m\ra {\infty}$:
\begin{align}
 R_{\infty}=\C[x^j_{(-i)}; i=1,2,\dots,\
j=1,\dots,
]/
\bra T^j f_i;i=1,\dots, s,
j=0,\dots, \ket.
\label{eq:R-infty}
\end{align}
We have
\begin{align*}
X_{\infty}= \Spec R_{\infty}.
\end{align*}
Note that
$R_{\infty}$ is a {\em differential algebra} with derivation $T$.
Here a differential algebra is a unital commutative algebra
equipped with a derivation.
Obviously
$R_{\infty}$ is generated by $R$ as a differential algebra.

A differential algebra
$V$ is called a 
{\em vertex Poisson algebra}
\cite{FreBen04}
if it is equipped with
a linear map
\begin{align*}
 V\ra (\on{Der} V)[[z\inv]]z\inv,\quad a\mapsto Y_-(a,z)=\sum_{n\geq 0}a_{(n)}z^{-n-1,}
\end{align*}
such that
\begin{align}
&a_{(n)}b=0\quad n\gg 0,
\nonumber\\
&(Ta)_{(n)}=-n a_{(n-1)}\nonumber,\\
&a_{(n)}b=\sum_{j\geq 0}(-1)^{n+j+1}\frac{1}{j!}T^j(b_{(n+j)}a),
\label{eq:skew-symmetry} 
\\
&[a_{(m)},b_{(n)}]=\sum_{j\geq 0}\begin{pmatrix}
				  m\\ j
				 \end{pmatrix}(a_{(j)}b)_{(m+n-j)}
\label{eq:VPA-commutator-}
\end{align}
in $\End V$
for all $m,n\geq 0$,
$a,b\in V$,
where $\on{Der}V$ denotes the
space of derivation on $V$.

Note that the following formula follows from \eqref{eq:skew-symmetry}:
\begin{align*}
(ab)_{(n)}=\sum_{i=0}^{\infty}(a_{(-i-1)}b_{(n+i)}+b_{(-i-1)}a_{(n+i)})\quad\text{for
 }a,b\in V,\ n\geq 0,
\label{eq:associativity-VPA-V}
\end{align*}
where
we have set
\begin{align*}
a_{(-n)}=\frac{1}{(n-1)!}T^{(n-1)}(a)\quad\text{for 
$n\geq 1$, $a\in V$}.
\end{align*}

 \begin{Lem}[{\cite[Proposition 2.3.1]{Ara12}}]\label{Lem:level-0-VPA}
For 
 a Poisson algebra $R$,
there is a unique vertex Poisson algebra
structure on
$R_{\infty}=\C[(\Spec R)_{\infty}]$
such that
\begin{align}
a_{(n)}b=\begin{cases}
	  \{a,b\}&\text{if }n=0,\\
0&\text{if }n>0,
	 \end{cases}\quad\text{for }a,b\in R\subset R_{\infty}.
\end{align}
 \end{Lem}
The vertex Poisson algebra structure described in Lemma
 \ref{Lem:level-0-VPA}
 is called
the {\em level $0$ vertex Poisson algebra structure} of $R_{\infty}$.

Let $\fing$
be a complex simple Lie algebra as in Introduction,
 $\{x^i; i\in I\}$ a basis of $\fing$.
Then
$\C[\fing^*]\cong \C[x^i; i\in I]$,
and thus,
\begin{align*}
\C[\fing^*_{\infty}]\cong \C[x^i_{(-n)};i\in I, n\geq 1].
\end{align*}
Throughout this paper  we regard
$\C[\fing^*_{\infty}]$
as a vertex Poisson  algebra
with the level zero vertex Poisson algebra structure induced from
the Kirillov-Kostant Poisson algebra structure on $\C[\fing^*]$.

Let $G$ be  the adjoint group of $\fing$.
The $m$-the jet scheme $G_m$ of $G$ is an algebraic 
group, which is isomorphic to a semi-direct product of 
$G$ and a unipotent group.
The infinite jet scheme $G_{\infty}$
of $G$ is  a proalgebraic group $G[[t]]$,
which is isomorphic to a semi-direct product of $G$
and a prounipotent group.
We have  
\begin{align*}
\on{Lie}(G_m)=\fing_m=\fing[t]/(t^{m+1}),\quad
\on{Lie}(G_{\infty})=\fing_{\infty}=\fing[[t]].
\end{align*}
The group $G_{\infty}$ acts on $\fing^*_{\infty}$
by adjoint.

Note that
by \eqref{eq:VPA-commutator-}
\begin{align*}
[x_{(m)},y_{(n)}]=[x,y]_{(m+n)}
\quad\text{for }
 m,n\in \Z_{\geq 0},\ x\in \fing\subset \C[\fing^*]\subset \C[\fing^*_{\infty}].
\end{align*}
Hence
the assignment 
\begin{align}
xt^n\mapsto x^M_{(n)}\quad
 x\in \fing,\ 
n\geq 0,
\end{align}
defines a $\fing[[t]]$-module structure on $\C[\fing^*_{\infty}]$.
This $\fing[[t]]$-action coincides with
the one obtained by 
differentiating
the adjoint action of $G[[t]]$.

\subsection{Modules over vertex Poisson algebras}
Recall that a {\em Poisson module}
$M$ over a Poisson algebra 
$R$ is
a $R$-module  $M$ in the usual associative sense 
equipped with
a bilinear map
\begin{align*}
R\times M\ra M,\quad (r,m)\mapsto \ad r(m)=\{r,m\},
\end{align*}
which makes $M$ a Lie algebra module over $R$
satisfying 
\begin{align*}
 \{r_1,r_2 m\}=\{r_1,r_2\}m+r_2\{r_1,m\},\quad
\{r_1 r_2,m\}=r_1\{r_2,m\}+r_2\{r_1,m\}
\end{align*}
for $r_1,r_2\in R$, $m\in M$.

For instance
a Poisson module over $\C[\fing^*]$
is the same as
a $\C[\fing^*]$-module $M$ as a ring
equipped with
an action of $\fing$ such that
\begin{align*}
x \cdot fm=\{x,f\}m+f(x\cdot m)\quad\text{ for }x\in \fing,\ f\in
 \C[\fing^*],\ m\in M.
\end{align*}

 \begin{Def}
A {\em vertex Poisson module} over 
a vertex Poisson algebra $V$ is a 
$V$-module $M$ in the usual associative sense 
 equipped with 
 a linear map
\begin{align*}
 V\mapsto (\End M)[[z\inv]]z\inv,\quad a\mapsto Y_-^M(a,z)=\sum_{n\geq 0}a^M_{(n)}z^{-n-1},
\end{align*}
satisfying
\begin{align}
& a_{(n)}^Mm=0\quad\text{for }n\gg 0, 
\\
&(T a)_{(n)}^M=-na^M_{(n-1)},\\
 &a_{(n)}^M(b v)=(a_{(n)}^Mb)  v+b (a_{(n)}^Mv),\\
&[a^M_{(m)},b^M_{(n)}]=\sum_{i\geq 0}\begin{pmatrix}
				  m\\ i
				 \end{pmatrix}(a_{(i)}b)^M_{(m+n-i)},\\
&(ab)^M_{(n)}=\sum_{i=0}^{\infty}(a_{(-i-1)}b_{(n+i)}^M+b_{(-i-1)}a_{(n+i)}^M)
\label{eq:associativity-VPA}
\end{align}
for all $a, b\in V$, $m, n\geq 0$, $v\in M$.

\end{Def}

A  vertex Poisson algebra $R$ is naturally a vertex Poisson module over
itself.

Let  $M$ be a vertex Poisson module over 
 $\C[\fing^*_{\infty}]$.
Then
the assignment 
\begin{align*}
xt^n\mapsto x^M_{(n)}\quad x\in \fing\subset \C[\fing^*]\subset \C[\fing^*_{\infty}],\
n\geq 0,
\end{align*}
defines a $\fing[t]$-module structure on $M$,
and in fact, a vertex Poisson module over $\C[\fing^*_{\infty}]$
is the same as a 
$\C[\fing^*_{\infty}]$-module $M$
in the usual associative sense
equipped with an action of the Lie algebra $\fing[t]$
such that
$(xt^n) m=0$ for $n\gg 0$,
$x\in \fing$,
$m\in M$,
and 
\begin{align*}
(xt^n)\cdot (a m)=(x_{(n)}a) m+a (x t^n)\cdot m
\end{align*}
for $x\in \fing$, $n\geq 0$,
$a\in \C[\fing^*_{\infty}]$,
$m\in M$.

Below we often write $a_{(n)}$ for $a^{M}_{(n)}$.

The proof of 
the following assertions are straightforward.
 \begin{Lem}\label{Lem:VPA-induced}
Let $R$ be a Poisson algebra,
$E$ a Poisson module over $R$.
Then
there is a unique
 vertex Poisson $R_{\infty}$-module structure 
on $
R_{\infty}\*_R E
$
such that
\begin{align*}
 a_{(n)}(b\* m)=(a_{(n)}b)\*
 m+\delta_{n,0}b\* \{a,m\}
\end{align*}
for
$n\geq 0$, 
$a\in R\subset R_{\infty}$, $b\in R_{\infty}$.
$m\in E$.
 \end{Lem}

 \begin{Lem}\label{Lem:universality-of-induced-VPA-module}
Let $R$ be an Poisson algebra,
$M$  a vertex Poisson module over $R_{\infty}$.
Suppose that
there exists a
 $R$-submodule $E$ of $M$
(in the usual commutative sense)
such that
$a_{(n)}E=0$  for $n>0$, $a\in R$,
and $M$ is generated by $E$
(in the usual commutative sense).
Then there exists a surjective homomorphism
\begin{align*}
 R_{\infty}\*_{R}E\twoheadrightarrow M
\end{align*}
of vertex Poisson modules.
 \end{Lem}

\subsection{$C_2$-cofiniteness and associated 
varieties of  vertex algebras}
\label{subsection:singulra-supports-and-assocaited-varites}
Recall that a  {\em vertex algebra}
is a vector space  $V$
equipped with
an element
$\1\in V$
called the {\em vacuum},
$T\in \End(V)$ called the {\em translation operator},
and 
 a linear map
 \begin{align*}
Y(?,z):V\ra (\End V)[[z,z\inv]], 
\quad a\mapsto Y(a,z)=a(z)=\sum_{n\in \Z}
a_{(n)}z^{-n-1}
 \end{align*}
called the {\em state-field correspondence}, 
such that
\begin{align*}
 &\1(z)=\id_V,\\
&\text{$a_{(n)}b=0$ for $n\gg 0$,
}\\
&\text{$a_{(n)}\1=0$ for $n\geq 0$
and $a_{(-1)}\1=a$},\\
&\text{$(Ta)(z)=[T, a(z)]=\frac{d}{dz}a(z)$},\\
&\text{$(z-w)^n [a(z),b(w)]=0$ in $\End (V)$
for $n\gg0$,}
\end{align*}
for all $a,b\in V$.

As a consequence of the definition 
we have
the {\em Borcherds identity} which may \cite{MatNag99}
be described as
\begin{align*}
 &[a_{(m)},b_{(n)}]=\sum_{i\geq 0}
\begin{pmatrix}
 m\\ i
\end{pmatrix}(a_{(i)}b)_{(m+n-i)},\\
& (a_{(m)}b)_{(n)}=\sum_{i\geq 0}(-1)^m
\begin{pmatrix}
 m\\i
\end{pmatrix}
(a_{(m-i)}b_{(n+i)}-(-1)^mb_{(m+n-i)}a_{(i)})
\end{align*}
in $\End V$
for all $a,b\in V$,
$m,n\in\Z$.

Let $F^{\bullet}V$ be the
Li filtration\footnote{The Li filtration
is defined  independent of the grading of $V$.
Hence it unifies the notion of the standard filtration
and the Kazhdan filtration in our case (compare \cite{Kos78,Lyn79,GanGin02}).
}
of 
a vertex algebra $V$ (\cite{Li05}).
By definition,
$F^0 V=V$,
\begin{align*}
F^p V=\haru_{\C}\{a_{(-i-1)}b; a\in V, b\in F^{p-i}V, i\geq 1\}.
\end{align*}
We have
\begin{align*}
& V=F^0V\supset F^1 V\supset F^2V\supset\dots,
\\
&
 a_{(n)}F^q V\subset F^{p+q-n-1}V\quad \text{for }
a\in F^p V,\  n\in \Z,
\nno \\
&
 a_{(n)}F^qV\subset  F^{p+q-n}V\quad \text{for }
 a\in
F^p V,
\ n\geq 0,\\
&T F^p V\subset F^{p+1}V.
\end{align*}
Also we have
$\bigcap F^p V=0$
if $V$ is positively graded (see below).

The associated graded space
\begin{align*}
\gr^F V=\bigoplus_{p\geq 0} F^p V/F^{p+1}V
\end{align*}
is a differential algebra
by
\begin{align*}
 \sigma_p(a)\sigma_q(b)=\sigma_{p+q}(a_{(-1)}b),\quad
T\sigma_p(a)=\sigma_{p+1}(Ta),
\end{align*}
where
$\sigma_p: F^pV\ra F^p V/F^{p+1}V$ is the projection
and $T$ is the translation operator of $V$.
Its vertex Poisson algebra
structure is given by
\begin{align*}
 \sigma_p(a)_{(n)}\sigma_q(b)=\sigma_{p+q-n}(a_{(n)}b),\quad n\geq 0,
\end{align*}
where we understand $\sigma_{n}(a)=0$ for $n<0$.

The subspace
$F^1 V$
 is the linear span of the vectors $a_{(-2)}b$
with $a,b\in \V$,
which is usually denoted by 
$C_2(V)$ in the vertex operator algebra
theory.
Set
\begin{align*}
    \Ring{V}\teigi V/F^1V(=V/C_2(V)).
   \end{align*}
Then 
$R_V$ is a subring
of $\gr V$ in the usual associative sense,
and moreover,
it has the structure of  a Poisson algebra,
where the Poisson bracket is given by
\begin{align*}
 \{a,b\}=a_{(0)}b\quad\text{for }a,b\in R_V\subset \gr^F V.
\end{align*}
The $R_V$ is called
{\em Zhu's $C_2$-algebra}
of $V$ (\cite{Zhu96}).

The vertex Poisson algebra
$\gr^F V$
is generated by $R_V$ as a differential algebra.
In fact
the  embedding $\Ring{V}\hookrightarrow \gr^F V$
induces 
the
surjective homomorphism 
\begin{align}
(R_V)_{\infty}\twoheadrightarrow \gr^F V
\label{eq:VPA-surj}
\end{align}
of vertex Poisson algebras (\cite{Li05,Ara12}).

Let $\{ a^i; i\in I\}$ be elements 
of $V$ such that their images generate  $R_V$
in usual commutative sense.
The surjectivity of \eqref{eq:VPA-surj}
implies that
\begin{align*}
&F^p V\\
&=\haru_{\C}\{a_{(-n_1-1)}^{i_1}
\dots a_{(-n_r-1)}^{i_r}\1;n_i\geq 0,
\ n_1+\dots +n_r\geq p,\ i_1,\dots,i_r\in I\}.
\end{align*}

In the rest of paper 
we assume that 
$V$ is {\em finitely strongly generated},
or equivalently,
 $R_V$ 
is
finitely  generated as a ring,
unless otherwise stated.

Define
\begin{align*}
 X_V=\Spec R_V.
\end{align*}
The $X_V$
is called the
{\em associated variety} of $V$.

A vertex algebra 
$V$ is called
{\em $C_2$-cofinite}
if $R_V$ is finite-dimensional.
Clearly,
we have
\begin{align*}
\text{$V$ is $C_2$-cofinite}
\iff   \dim X_V=0.
\end{align*}

A $C_2$-cofinite vertex algebra
is an analogue of finite-dimensional algebra.
To see this,
consider the {\em singular support}
\begin{align*}
 SS(V):=\Spec (\gr^F V)\subset (X_V)_{\infty}
\end{align*}
of the vertex algebra $V$.
We have \cite{Ara12}
\begin{align*}
 \dim SS(V)=0\iff \dim X_V=0\ (\iff
\text{$V$ is $C_2$-cofinite}).
\end{align*}
This
follows from the fact that
\begin{align*}
X_V=\pi_{\infty,0}(SS(V)),
\end{align*}
where 
$\pi_{\infty,0}: (X_V)_{\infty}\ra X_V$
is the natural projection,
and the general fact 
$X_{\infty}$
is homeomorphic to
$(X_{\on{red}})_{\infty}$,
where 
$X_{\on{red}}$ denotes the reduced scheme of $X$.

\smallskip

A vertex algebra $V$ is called  {\em graded}
if  there exists a semisimple operator
$H\in \End V$,
called a {\em Hamiltonian},
such that
\begin{align}
 [H,a_{(n)}]=-(n+1)a_{(n)}+(Ha)_{(n)}.
\label{eq:Hamiltonian}
\end{align}
For an eigenvector $a$ of $H$,
it eigenvalue is called the {\em conformal weight} of $a$
and denoted by $\Delta_a$.
The vertex algebra $V$
is called  {\em positively graded} if 
\begin{align*}
V=\bigoplus_{\Delta\in \Q_{\geq 0}}V_{\Delta},
\quad V_0=\C \mathbf{1},
\end{align*}
where
$V_{\Delta}=\{a\in V; Ha=\Delta a\}$.
If this is the case 
$X_V$ is conic, and therefore,
\begin{align*}
 V\text{ is $C_2$-cofinite }&\iff X_V=\{pt\}.
\end{align*}

\subsection{Associated varieties of modules over vertex algebras}
\label{subsection:Associated varieties of modules over vertex algebras}
Let 
$V$ be a finitely strongly generated,
graded vertex algebra.
A {\em module}  over  a vertex algebra
$V$  is a vector space  $M$
equipped with a linear map
 \begin{align*}
Y^M(?,z):V \ra (\End M)[[z,z\inv]], 
\quad a\mapsto a(z)=\sum_{n\in \Z}
a_{(n)}^M z^{-n-1},
 \end{align*} 
such that
\begin{align*}
&Y^M(\1,z)=\id_M,\\
&\text{$a_{(n)}^M m=0$ for $n\gg 0$,
$a\in V$, $m\in M$}
\\
 &[a_{(m)}^M,b_{(n)}^M]=\sum_{i\geq 0}
\begin{pmatrix}
 m\\ i
\end{pmatrix}(a_{(i)}b)^M_{(m+n-i)},\\
& (a_{(m)}b)^M_{(n)}=\sum_{i\geq 0}(-1)^m
\begin{pmatrix}
 m\\i
\end{pmatrix}
(a_{(m-i)}^Mb_{(n+i)}^M-(-1)^mb_{(m+n-i)}^Ma_{(i)}^M)
\end{align*}

In particular $V$ itself if a module over $V$
called the {\em adjoint module}.

Below we often write $a_{(n)}$ for $a^M_{(n)}$.

A $V$-module $M$ is called {\em graded}
if there is  a semisimple action
of the Hamiltonian $H$
of $V$
on $M$
satisfying
\eqref{eq:Hamiltonian} in $\End V$.
For a
graded module $M$
set 
\begin{align*}
M_d=\{m\in M;
Hm =d m\}
\end{align*}A {\em positively graded} $V$-module
is a graded $V$-module $M=\bigoplus_{d\in \C}M_d$
such that 
$M_d=0$ unless $d\in \bigcup_{i=1}^r (d_i +\Q_{\geq 0})$
for some $d_1,\dots,d_r\subset \C$.

Let $\V\Mod$ denote
the abelian category of graded $V$-modules.

A {\em compatible filtration} 
of a $V$-module $M$ is a decreasing filtration
\begin{align*}
M=\Gamma^0M\supset \Gamma^1M\supset \cdots
\end{align*}
such that
\begin{align*}
&
 a_{(n)}\Gamma^q M\subset \Gamma^{p+q-n-1}M\quad \text{for }
a\in F^p V,\ \forall n\in \Z,
\nno \\
&
 a_{(n)}\Gamma^qM\subset  \Gamma^{p+q-n}M\quad \text{for }
 a\in
F^p V,
\ n\geq 0,
\\
& H. \Gamma^p M\subset \Gamma^p M\quad \text{for all }p\geq 0,\\
&\bigcap_p \Gamma^p M=0.
\end{align*}
For a compatible filtration
$\Gamma^{\bullet}M$
the associated graded space
\begin{align*}
\gr^\Gamma M=\bigoplus_{p\geq 0} \Gamma^p M/\Gamma^{p+1}M
\end{align*}
is
naturally a graded vertex Poison module over the graded vertex Poisson algebra
$\gr^F V$,
and hence,
it is a 
graded vertex Poison module over $(R_V)_{\infty}$
by \eqref{eq:VPA-surj}.

The vertex Poisson
$(R_V)_{\infty}$-module
structure 
of $\gr^\Gamma M$
restricts to 
the Poisson $R_V$-modules structure
of $M/\Gamma^1 M=\Gamma^0M /\Gamma^1M$,
and
$a_{(n)}(M/\Gamma^1 M)=0$
 for $a\in R_V\subset (R_V)_{\infty}$,
$n>0$.
It follows that
there is a  homomorphism
\begin{align*}
 (R_V)_{\infty}\*_{R_V}(M/\Gamma^1 M)\ra \gr^{\Gamma}M,
\quad a\* \bar m\mapsto a\bar m,
\end{align*}
of vertex Poisson modules
by Lemma \ref{Lem:universality-of-induced-VPA-module}.

Suppose that
$V$ is positively graded and so
is the $V$-module $M$.
We denote by $F^\bullet M$  
the Li filtration \cite{Li05}  
of $M$,
which 
is 
 defined 
inductively by
$F^0 M=M$,
\begin{align*}
F^p M=\haru_{\C}\{a_{(-i-1)}b; a\in V, b\in F^{p-i}M, i\geq 1\}.
\end{align*}
It is a compatible filtration of $M$,
and in fact,
 it 
 the finest compatible filtration of $M$,
that is,
$F^p M\subset \Gamma^p M$ for all $p$
for any compatible filtration 
$\Gamma^\bullet M$ of $M$.
The subspace $F^1 M$ is spanned by the
vectors $a_{(-2)}m$ with $a\in V$, $m\in M$,
which is usually denote by
$C_2(M)$ in the vertex operator algebra theory.
Set
\begin{align}
 \bar M=M/F^1 M(=M/C_2(M)),
\end{align}
which 
 is a 
Poisson module over $R_V=\bar V$.
By \cite[Proposition 4.12]{Li05}
the vertex Poisson module homomorphism
\begin{align*}
 (R_V)_{\infty}\*_{R_V}\bar M\ra \gr^F M
\end{align*}
is surjective.

Let $\{ a^i; i\in I\}$ be elements 
of $V$ such that their images generates  $R_V$
in usual commutative sense,
and let $U$ be a subspace of
$M$ such that $M=U+F^1 M$.
The surjectivity of the above map is equivalent to that
\begin{align}
 &F^p M\label{eq:span-FpM}
\\
=\haru_{\C}&\{a_{(-n_1-1)}^{i_1}
\dots a_{(-n_r-1)}^{i_r}m;m\in U,\ n_i\geq 0,
n_1+\dots +n_r\geq p,\
 i_1,\dots,i_r\in I\}.
\nonumber
\end{align}

A $V$-module $M$ is called 
 {\em finitely strongly generated} 
if 
 $\bar M$ is finitely generated over $R_V$
in the usual associative sense.
It is called 
{\em $C_2$-cofinite} if $\bar M$ 
is finite-dimensional.
The {\em associated variety}
of $M$
is by definition
\begin{align}
 X_M=\on{supp}_{\Ring{V}}(\bar M)
\label{eq:def-of-ass.v}
\end{align}
which is a Poisson subscheme of $X_V$.
Clearly,
\begin{align}
\text{ $M$ is $C_2$-cofinite }\iff
\dim X_M=0
\label{eq:C_2-cofintie-criteri}
\end{align}
for a finitely strongly generated $V$-module $M$.

{\em Throughout the paper 
$\{F^p M\}$ denotes the Li filtration 
and a general compatible filtration will be denoted by $\{\Gamma^p
M\}$.}

\subsection{Affine vertex algebras}
\label{subsection:affine vertex algebras}
Let $\fing$ be a finite-dimensional simple Lie algebra
as above,
and let
 $\bigaffg$ be the non-twisted affine Kac-Moody algebra
associated with  $\fing$:
\begin{align*}
 \bigaffg=\fing[t,t\inv]\+\C K\+ \C D,
\end{align*} 
whose commutation relations  are given by
\begin{align*}
 &[x_{(m)},y_{(n)}]=[x,y]_{(m+n)}+m(x|y)\delta_{m+n,0}
K,\\
&[D, x_{(m)}]=mx_{(m)},\quad
[K,\bigaffg]=0
\end{align*}
with $x,y\in \fing$ and $m,n\in \Z$,
where
 $x_{(m)}=x\* t^m$ and $(~|~)$ is a normalized invariant bilinear
form on $\fing$.
We identify $\fing$ with the subalgebra
$\fing\* \C\subset \bigaffg$.
Set
\begin{align*}
 \affg\teigi [\bigaffg,\bigaffg]=\fing[t,t\inv]\+ \C K.
\end{align*}

Define
\begin{align*}
 \Vg{k}=U(\bigaffg)\*_{U(\fing[t]\+ \C K\+ \C D)} \C_k,
\end{align*}
where
 $\C_k$ is the one-dimensional
representation of 
$\fing[t]\+ \C K\+ \C D\subset \bigaffg$ on which 
$\fing[t]\+ \C D$ acts trivially and $K$ acts as a  multiplication by $k$.
There is a unique vertex algebra
structure on $\Vg{k}$ such that
$\1=1\* 1$
 is the vacuum and
\begin{align*}
 Y(xt\inv\1,z)=x(z):=\sum_{n\in \Z}xt^n z^{-n-1}\quad\text{for }x\in \fing.
\end{align*}
The vertex algebra
$\Vg{k}$ is called
the  {\em universal affine vertex algebra}
associated with $\fing$ at level $k$.
The standard grading of $\Vg{k}$ is given  by  the
Hamiltonian
\begin{align*}
 H=-D.
\end{align*}

A $\Vg{k}$-module
is the same as
a  $\affg$-modules $M$
of level $k$ such that $x(z)$ is a field on $M$ for any $x\in \fing$,
that is,
$x_{(n)}m=0$ with  $n\gg 0$  for any $m\in M$.
Such a representation is called a {\em smooth} module of $\affg$
of level $k$.

Because the action of $H$
on $\Vg{k}$
stabilizes the Li filtration,
the vertex Poisson algebra $\C[\fing^*_{\infty}]$ 
is graded by the Hamiltonian $H$.
If $M$ is a graded $\Vg{k}$-module
and $\{\Gamma^p M\}$ is a compatible filtration
then 
$\gr^\Gamma M$ is a graded vertex Poisson
$\C[\fing^*_{\infty}]$-module.

By \eqref{eq:span-FpM} we have
\begin{align}
 F^p M=F^p U(\fing[t\inv]t\inv) M
\end{align}
for a $\Vg{k}$-module $M$,
where
\begin{align}
F^{p}&U(\fing[t\inv]t\inv)\\
&
=\haru_{\C}\{(x_1)_{(-n_1-1)}
\dots (x_r)_{(-n_r-1)};x_i\in \fing,\ n_i\geq 0,
n_1+\dots +n_r\geq  p
\}.
\nonumber
\end{align}
In particular,
\begin{align*}
F^1 M(=C_2(M))=\fing[t^{-1}]t^{-2} M.
\end{align*}
Since
we have
$\Vg{k}\cong U(\fing[t\inv]t\inv)$ 
as  vectors spaces
by the PBW theorem,
 we have the 
isomorphism of 
Poisson algebras 
\begin{align*}
 \C[\fing^*]=S(\fing)\isomap  R_{\Vg{k}}=\Vg{k}/F^1 \Vg{k},
\quad x\mapsto \overline{xt\inv \1},\quad x\in \fing.
\end{align*}
This induces the isomorphism of  vertex Poisson algebras
\begin{align*}
 \C[\fing^*_{\infty}]\isomap  \gr^F \Vg{k}.
\end{align*}
Thus,
\begin{align*}
 X_{\Vg{k}}=\fing^*,\quad 
SS(\Vg{k})=\fing^*_{\infty}.
\end{align*}

We denote by  $V_k(\fing)$ be the unique simple graded quotient of
$\Vg{k}$.
The surjection
$\Vg{k}\ra V_k(\fing)$
induces
the surjective homomorphism 
\begin{align*}
 \C[\fing^*]= R_{\Vg{k}} 
\twoheadrightarrow 
 R_{V_k(\fing)}=V_k(\fing)/F^1\Vg{k}
=V_k(\fing)/\fing[t\inv]t^{-2}V_k(\fing)
\end{align*}
of Poisson algebras.
Therefore
$X_{\Vg{k}}$ is a conic,
$G$-invariant,
 closed Poisson subscheme of 
$\fing^*$.

\subsection{Kac-Moody algebras and the category 
$\KL_k$}\label{subsection:Kac-Moody}
Let $\finb$ be a Borel subalgebra
of $\fing$,
$\finh\subset \finb$
the Cartan subalgebra,
$\Delta_+$ the corresponding set of positive roots,
$\Delta=\Delta_+\sqcup -\Delta_+$.
Let
$\fing_{\alpha}$ the root space of root $\alpha\in \Delta$,
$\theta$ the highest root,
$\theta_s$ the highest short root,
$\rho=\sum_{\alpha\in \Delta_+}\alpha/2$, 
$\rho\che=\sum_{\alpha\in \Delta_+}\alpha\che/2$,
where
$\alpha\che=2\alpha/(\alpha|\alpha)$.
Let $\Pi=\{\alpha_1,\dots,\alpha_{l}\}\subset \Delta_+$
be the set of simple roots of $\Delta$,
where $l$ is the rank of $\fing$.

Let
$P^+=\{\bar \lam\in \finh^*;
\bar \lam(\alpha\che)\in \Z_{\geq 0}\
\text{for }\alpha\in \roots_+\}$,
the set of the integral dominant weights of $\fing$.
Denote by $E_{\lam}$  the irreducible finite-dimensional
representation of $\fing$ with highest weight $\lam\in P^+$.

Let 
$
\affh=\finh\+ \C K \+ \C D
$ be 
the  Cartan subalgebra of $\bigaffg$,
$\dual{\affh}=\dual{\finh}\+ \C \delta\+ \C\Lam_0$
the dual of $\affh$.
Here
$\delta(D)=\Lam_0(K)=1$,
$\delta(\finh)=\Lam_0(\finh)=\delta(K)=\Lam_0(D)=0$.
For a 
$\affh$-module $M$
and $\lam\in \dual{\affh}$,
let
\begin{align*}
M^{\lam}=\{m\in M;
hm=\lam(h)m\text{ for }h\in \affh\}.
\end{align*}

Set
$\affg_+=[\finb,\finb]+ \fing[t]t\subset \bigaffg$,
$\affg_-=[\finb_-,\finb_-]+ \fing[t\inv]t\inv\subset \bigaffg$,
where
$\finb_-$ be  the opposite Borel subalgebra
of $\fing$.
Then
\begin{align*}
\bigaffg=\affg_- \+ \affh\+ \affg_+
\end{align*}
gives the standard  triangular decomposition of $\bigaffg$.
Denote by $\wh{\roots}_+$ 
the
corresponding set of positive roots of $\bigaffg$,
by $\wh{\roots}_+^{\on{re}}$ the 
set of  positive real roots of $\bigaffg$,
$\wh{Q}_+=\sum_{\alpha\in \wh{\roots}_+}\Z_{\geq 0}\alpha
\subset \dual{\affh}$.
Let $\affW$ be the subgroup of $GL(\affh^*)$
generated by 
the reflection
$s_{\alpha}$ with $\alpha\in \wh{\roots}_+^{\on{re}}$,
where
$s_{\alpha}(\lam)=\lam-\lam(\alpha\che)\alpha$.
The dot action of $\affW$ on $\dual{\affh}$
is given by $w\circ \lam=w(\lam+\wh \rho)-\wh \rho$,
where $\wh \rho=\rho+h\che \Lam_0$ and $h\che$ is the 
dual Coxeter number of $\affg$.

For $\lam\in \dual{\affh}$,
let $\Irr{\lam}$ be the 
irreducible highest weight representation of $\bigaffg$
with highest weight $\lam$.
Note that 
\begin{align*}
 V_k(\fing)\cong \Irr{k\Lam_0}
\end{align*}as $\affg$-modules.

For $k\in \C$,
let $\KL_k$ be as the
 full subcategory of the category of 
$\bigaffg$-modules consisting of objects $M$ 
satisfying
\begin{itemize}
\item $M$ is of level  $k$,
that is,  $K$ acts on $M$
as the multiplication by  $k$,
\item $M=\bigoplus_{d\in \C}M_d$
and  $\dim M_d<\infty$ for all $d\in \C$,
where \begin{align*}
M_d=\{m\in M; D m=-d m\},
      \end{align*}
\item 
there exists a finite subset   $\{d_1,\dots, d_r\}$ 
of $\C$
such that
$M_d=0$ unless $d\in \bigcup_{i=1}^r (d_i+\Z_{\geq 0})$.
\end{itemize}
From the definition it follows that  the action of 
$\fing[t]\subset \bigaffg$
on $M\in \KL_k$ integrates to the action of $G_{\infty}=G[[t]]$.

Since an object of $\KL_k$ is obviously smooth,
 $\KL_k$ may be thought as a full subcategory 
of $\Vg{k}\Mod$.

The irreducible representation
 $\Irr{\lam}$ is an object of 
$\KL_k$
if and only if
$\lam\in \affP^+_k$,
where
\begin{align}
 \affP_{k}^+\teigi \{\lam\in \dual{\bh}_k;
\bar \lam\in P^+
\}.
\label{eq:p+}
\end{align}
Here $\affh_k^*\teigi \{\lam\in \dual{\affh};
\lam(K)=k\}$
and  $\dual{\affh}\ra \dual{\finh}$,
$\lam\mapsto \bar \lam$,
is the restriction map.

For $\lam\in \affP^+_k$,
let
$\Weyl{\lam}\in \KL_k$ be the {\em Weyl module} with highest weight $\lam$:
\begin{align*}
 \Weyl{\lam}=U(\bigaffg)\*_{U(\fing[t]\+ \C K\+ \C D)}E_{\lam},
\end{align*}
where $E_{\lam}$ is the $\fing$-module
$E_{\bar \lam}$ considered as a $\fing[t]\+ \C K \+ \C D$-module
on which $\fing[t]t$ acts trivially,
$K=k \id_{E_{\lam}}$
and $D=\lam(D)\id_{E_{\lam}}$.

Note that
$M\in \KL_k$ is finitely generated as a $\bigaffg$-module
if and only if
it is finitely strongly generated as a $\Vg{k}$-module.

Let $\KL_k^{\Delta}$
be the full subcategory of $\KL_k$ consisting
of objects $M$ that admits a {\em Weyl flag},
that is a finite filtration
$M=M_0\supset M_1\supset \dots \supset M_r=0$
such that 
each successive quotient $M_i/M_{i+1}$ is isomorphic to
$\Weyl{\lam_i}$ for some $\lam_i$.
Any object of $\KL_k^{\Delta}$ is finitely generated and
 any finitely generated object of $\KL_k$ is a quotient of some
object
of $\KL_k^{\Delta}$.

Since $M\in \KL_k^{\Delta}$ is free over
$U(\fing[t\inv]t\inv)$,
we have the following assertion.
\begin{Lem}\label{Lem:surjection-as-gr-delta}
For  an object $M$ of $\KL_k^{\Delta}$,
the surjective homomorphism
\begin{align*}
\C[\fing^*_{\infty}]\*_{\C[\fing^*]}\bar M
\ra\gr^F M
\end{align*}
is an isomorphism
of vertex Poisson modules over $\C[\fing^*_{\infty}]$.
\end{Lem}

\section{Jet schemes of Slodowy slices and their vertex Poisson modules}
In this section we collect some fundamental facts about jet schemes
of Slodowy slices and prove two (co)homology vanishing 
assertions Proposition \ref{Pro:vanishing-of-homology}
and Proposition
\ref{Pro:cohomology-vanishing-gr}.
\subsection{The category $\Cat$}
\label{subsection:category-C}
Let $\Cat$ be the full subcategory of 
the category of 
 vertex Poisson modules over 
$\C[\fing^*_{\infty}]$
consisting of modules $M$
such that  
$\fing[t]t\subset \fing[t]$ acts locally nilpotently on $M$
and
 $M$ is a direct sum of finite-dimensional representations
as a module over $\fing\subset \fing[t]$.

If $M\in \KL_k$
and
$\Gamma^\bullet M$ is
a compatible filtration,
then $\gr^{\Gamma}M$ is an object of $\Cat$.

For
a finite-dimensional
representation $E$ of $\fing$,
define
\begin{align*}
\Delta(E)
=\C[\fing^*_{\infty}]\*_{\C[\fing^*]}(\C[\fing^*]\*_{\C} E)\in \Cat,
\end{align*}
where
$\C[\fing^*]\*_{\C}  E$ is considered as a Poisson $\C[\fing^*]$-module 
by the action
\begin{align*}
& r (r'\* m)=(r r')\* m, \quad r,r'\in \C[\fing^*], \ m\in M,\\
&\{x,r\* m\}=\{x,r\}\* m+r\* xm,\quad x\in \fing,\ r\in \C[\fing^*], \
 m\in M.
\end{align*}
Note that 
\begin{align*}
\gr^F \Weyl{\lam}\cong 
\Delta(E_{\bar \lam}).
\end{align*}
 
\begin{Lem}\label{Lem:highest-weight-filtration}
 Let $M$ be an object of $\Cat$.
Then there exists a filtration
$0=M_0\subset M_1\subset \dots $
such that (1) $M=\bigcup_i M_i$,
(2) each $M_i/M_{i-1}$ is a quotient of $\Delta(E_i)$
for some finite dimensional simple $\fing$-module $E_i$,
and (3) $d_i\not < d_{j}$ for $i<j$.
Moreover if $M$ is finitely generated then there exists a finite
filtration 
$0=M_0\subset M_1\subset \dots \subset M_r=M$ with this property.
 \end{Lem}

\subsection{Jet schemes of Slodowy slices}
Let $f$ be a nilpotent element of $\fing$.
By the Jacobson-Morozov theorem,
 $f$ can be embedded
into 
 an   $\mathfrak{sl}_2$-subalgebra 
\begin{align*}
\mf{s}=\haru_{\C}\{e,h,f\}
\end{align*}
satisfying
\begin{align*}
& [h,e]=2e, \quad [h,f]=-2f,\quad [e,f]=h.
\end{align*}
The form $(\ |\ )$
 gives an isomorphism
\begin{align*}
 \nu:\fing\isomap \fing^*
\end{align*}
of $G$-modules.
The induced isomorphisms $\fing_m\isomap \fing_m^*$
and $\fing_{\infty}\isomap (\fing_{\infty})^*$
are  denoted by $\nu_{\infty}$ and
$\nu_m$,
respectively.

Let $\chi_{\infty}$,
$\chi_m$ and
$\chi$
be the images of $f\in \fing_{\infty}$,
$f+\fing[t]t^{m+1}\in \fing_m$
and $f\in \fing$ by $\nu_{\infty}$,
$\nu_m$ and $\nu$,
respectively.

Let  $\Nil\subset \fing^*$  be  the image of the set  of nilpotent
elements of $\fing$,
and let
$\Sl\subset \fing^*$ be the 
image of  $f+\fing^e$,
where $\fing^e$ is the centralizer of $e$ in $\fing$.
The affine space $\Sl$ is transversal at $\chi$ to $\Ad G.\chi$,
and is often called the
{\em Slodowy slice}.

We have
\begin{align*}
 \Sl_m=\chi_m+\nu_m(\fing^e_{m})\subset
 \fing_{m}^*,
\quad
 \Sl_{\infty}=\chi_{\infty}+\nu_{\infty}(\fing^e_{\infty})\subset \fing_{\infty}^*.
\end{align*}
Because
$ \fing_m=[\fing_m,f]\+ \fing_m^e$,
one has the following assertion.
\begin{Lem}\label{Lem:transversality}
The affine space $\Sl_m$
is transverse to the orbit
 $\Ad G_m .\chi_m$
at $\chi_m$
for any $m\geq 0$.
\end{Lem}

It is known \cite[\S 3]{GanGin02} that
the Poisson structure of $\fing^*$
restricts to $\Sl$.
Hence $\C[\Sl_{\infty}]$ is equipped with
the level zero vertex Poisson algebra
induced from the Poisson structure of $\C[\Sl]$.

  \subsection{Good gradings}\label{subsection:good-grading}
Let
\begin{align}
 \fing_j=\bigoplus_{j\in \frac{1}{2}\Z} \fing_j
\label{eq:good-grading}
\end{align}
be a {\em good grading} \cite{KacRoaWak03} for $\mf{s}$,
that is,
a $\frac{1}{2}\Z$-grading of $\fing$
such that
\begin{align}
 &e\in \fing_1,\quad h\in \fing_0,\quad f\in \fing_{-1},\\
&\ad f: \fing_j\ra \fing_{j-1}\text{ is injective for }j\geq
 \frac{1}{2},\\
&\ad f: \fing_j\ra \fing_{j-1}\text{ is surjective for }j\leq \frac{1}{2}.
\end{align}
Let $\x$
be the element 
in $\fing_0$
which 
defines  the grading,
i.e.,
 \begin{align*}
     \fing_j=\{x\in \fing;[\x,x]=j x\}.
    \end{align*}
Good gradings were classified in \cite{ElaKac05}.
The grading defined by 
\begin{align}
\x=\frac{h}{2}
\label{eq:dynkin-h}
\end{align}
is obviously good and 
is called the {\em
Dynkin
grading}.

{\em 
 We assume that the grading \eqref{eq:good-grading}
is compatible with the
triangular decomposition of $\fing$},
that is,
$\finb_+\subset \bigoplus_{j\geq 0}\fing_j$,
 $\finh\subset \fing_0$ and 
$h,\x\in \finh$.

Put
\begin{align}
 \Delta_j=\{\alpha\in \Delta; \fing_{\alpha}\subset \fing_j\},
\quad \Delta_{0,+}=\Delta_+\cap \Delta_0,
\label{eq:dec-of-roots}
\end{align}
so that $\Delta=\bigsqcup_{j\in \frac{1}{2}\Z}\Delta_j$,
$\Delta_+=\Delta_{0,+}\sqcup \bigsqcup_{j>0}\Delta_j$.

\subsection{The fundamental isomorphisms}
Set
\begin{align}
 &\finm=\bigoplus_{j\geq  1}\fing_j,\quad
\finn=\fing_{\frac{1}{2}}\+\bigoplus_{j\geq  1}\fing_j.
\label{eq:nilpotent-subagebras}
\end{align}
Then $\finm\subset \finn$, and they are both nilpotent
subalgebras
of $\fing$.

Denote by $N$ the unipotent subgroup of $G$
with $\on{Lie}(N)=\finn$.
It is known by  \cite{Kos78,GanGin02}
that
 the coadjoint action map
gives the isomorphism of affine varieties 
\begin{align}
N\times \Sl\isomap \chi+\finm^{\bot},
\end{align} 
where 
$\finm^{\bot}\subset \fing^*$ is the annihilator of  $\finm$.
Since $(X\times Y)_{m}\cong X_m\times Y_m$,
we immediately get the following  assertion.
\begin{Pro}\label{Pro:fundamental}
 The coadjoint action maps give the following isomorphisms:
 \begin{align*}
&  N_m\times \Sl_m\isomap \chi_m+(\finm^{\bot})_m
\quad \text{for all }m\geq 0,\\
 & N_\infty\times \Sl_\infty\isomap \chi_{\infty}+(\finm^{\bot})_\infty.
 \end{align*}
\end{Pro}

For a Lie algebra $\mf{a}$ and a $\mf{a}$-module $M$,
we write
 $H_{\on{Lie}}^{\bullet}(\mf{a},M)$ 
(respectively $H^{\on{Lie}}_{\bullet}(\mf{a},M)$) for
the Lie algebra $\mf{a}$-cohomology 
(respectively $\mf{a}$-homology)
with the coefficient  
in $M$.

\begin{Co}\label{Co:vanishing-for-functions}
We have the following.
\begin{align*}
& H^i_{\on{Lie}}(\finn_m,
\C[\chi_r+(\finm^{\bot})_m])
\cong \begin{cases}
       \C[\Sl_{m}]&\text{for }i=0,\\
0&\text{for }i>0,
      \end{cases}\\
& H^i_{\on{Lie}}(\finn_{\infty},\C[\chi_{\infty}+(\finm^{\bot})_{\infty}])
\cong \begin{cases}
       \C[\Sl_{\infty}]&\text{for }i=0,\\
0&\text{for }i>0.
      \end{cases}
\end{align*}
\end{Co}
 \subsection{The $\C^*$-action}
Let
$I_{\infty}$ be the ideal of $\C[\fing^*_{\infty}]$
generated by 
$x-\chi_{\infty}(x)$
with 
$x\in \finm[t\inv]t\inv$,
$I_m=I_{\infty}\cap \C[\fing^*_m]$.
Then
\begin{align*}
 \C[\chi_{\infty}+(\finm^{\bot})_{\infty}]=
\C[\fing^*_{\infty}]/I_{\infty},\quad
 \C[\chi_{m}+(\finm^{\bot})_{m}]=
\C[\fing^*_{m}]/I_{m}.
\end{align*}
(For $m=0$ we have
$\C[\chi+\finm^{\bot}]=\C[\fing^*]/I$,
where $I=I_0$.)

Hence
Corollary \ref{Co:vanishing-for-functions}
gives the isomorphism
\begin{align}
 \C[\Sl_{\infty}]\cong
 (\C[\fing^*_{\infty}]/I_{\infty})^{\finn[t]}
,\quad
 \C[\Sl_{m}]\cong
 (\C[\fing^*_{m}]/I_{m})^{\finn_m[t]}\quad\text{for all }m\geq 0.
\label{eq:slodowy-as-subquotient}
\end{align}
(For $m=0$ we have
$\C[\Sl]\cong (\C[\fing^*]/I)^{\finn}$.)
Set
\begin{align*}
 H_{\new}=H- (\x)_{(0)}\in \End \C[\fing^*_{\infty}].
\end{align*}
This defines 
a new $\frac{1}{2}\Z$-grading on 
the vertex Poisson algebra
$\C[\fing^*_{\infty}]$:
\begin{align*}
\C[\fing^*_{\infty}]
=\bigoplus_{\Delta\in \frac{1}{2}\Z}
\C[\fing^*_{\infty}]
_{\Delta,\new},\quad
\C[\fing^*_{\infty}]_{\Delta,\new}=\{a\in \C[\fing^*_{\infty}];
H_{\new}a=\Delta a\}.
\end{align*}

As easily seen $H_{\new}$ stabilizes
$I_{\infty}$ and $I_{m}$.
Therefore
it
acts on 
$\C[\chi_{\infty}+(\finm^{\bot})_{\infty}]$
and $\C[\chi_m+(\finm^{\bot})_{m}]$.
It gives $\frac{1}{2}\Z_{\geq 0}$-gradings
\begin{align}
 &\C[\chi_{\infty}+(\finm^{\bot})_{\infty}
]
=\bigoplus_{\Delta
\in \frac{1}{2}\Z_{\geq
 0}}\C[\chi_{\infty}+(\finm^{\bot})_{\infty}]_{\Delta},
\quad \dim \C[\chi_{\infty}+(\finm^{\bot})_{\infty}]_0=1,\\
 &\C[\chi_m+(\finm^{\bot})_{\infty}]
=\bigoplus_{\Delta
\in \frac{1}{2}\Z_{\geq
 0}}\C[\chi_m+(\finm^{\bot})_{m}
]_{\Delta},
\quad \dim \C[\chi_m+(\finm^{\bot})_{m}]_0=1,
\end{align}
where 
$\C[\chi_{\infty}+(\finm^{\bot})_{\infty}
]_{\Delta}$
and $\C[\chi_m+(\finm^{\bot})_{m}]_{\Delta}$
are eigenspaces of $H_{\new}$ with eigenvalue $\Delta$.
This gives
$\C^*$-actions
on 
$\chi_{\infty}+(\finm^{\bot})_{\infty}$
and 
$\chi_{m}+(\finm^{\bot})_{\infty}$,
which contract to the unique fixed points $\chi_{\infty}$ and $\chi_m$,
respectively.
Therefore
by
Lemma \ref{Lem:transversality}
and \cite[1.4]{Gin08} 
we have 
the following assertion.
\begin{Pro}\label{Pro:strong-transversality}
For each $m\geq 0$,
 any $G_m$ orbit in
$\fing_m$ meets the affine space
$\chi+(\finm^{\bot})_m$
transversely.
\end{Pro}

As $\finn[t]$ is stable under the adjoint action of
$H_{\new}$,
it follows from \eqref{eq:slodowy-as-subquotient}
that
there are induced  contracting  $\C^*$-actions
on
$\Sl_{\infty}$
and $\Sl_m$
as well.

\subsection{The flatness}
Let $\C[\finm_{\infty}^*]_{\chi_{\infty}}$ be the localization of 
the polynomial algebra 
$\C[\finm_{\infty}^*]$
at the point $\chi_{\infty}$.

\begin{Pro}\label{Pro:flatness}
 Let $M$ be a finitely generated 
object of 
$\Cat$.
Then, for any $x\in \chi_{\infty}+(\finm^{\bot_{\fing}})_{\infty}$,
the localization $M_x$  of $M$ at $x$ is flat over
$\C[\finm^*_{\infty}]_{\chi}$.
\end{Pro}
\begin{proof}
By Lemma \ref{Lem:highest-weight-filtration},
it suffices to 
prove the assertion in the case that $M$
is a quotient of  $\Delta(E)$
for some finite-dimensional simple $\fing$-module $E$.

Let $U$ be the image of $\C\*E\subset \Delta(E)$
under the quotient map $\Delta(E)\ra M$.
Set
$M_r=\C[\fing^*_r]U$.
Then 
$M=\bigcup_r M_r$.
Note that  
the  $G_{\infty}$-module
structure of $M$
induces the $G_r$-module structure on
  $M_r$.

We have $\C[\finm^*_r]_{\chi_r}\subset \C[\finm^*_{r+1}]_{\chi_{r+1}}$
and 
$\C[\finm^*_{\infty}]_{\chi_{\infty}}=\bigcup_r \C[\finm^*_r]_{\chi_r}$.
Let  $x_r=\pi_{\infty,r}(x)$,
where $\pi_{\infty,r}: \fing^*_{\infty}\ra \fing^*_m$ is the projection.
Then $M_{x}=\bigcup_r (M_r)_{x_r}$. 
By \cite[Corollary 6.5]{Eis95}
it is sufficient to show that
each $(M_r)_{x_r}$ is flat over $\C[\finm^*_r]_{\chi_r}$.

Let
$\tilde{M}_r$ be the sheaf on 
the affine space $\fing^*_r$ corresponding to the
$\C[\fing^*_r]$-module $M_r$.
Because  $\tilde{M}_r$ is $G_r$-equivalent and coherent,
 we can apply by Proposition \ref{Pro:strong-transversality}
the argument of
\cite[Corollary 1.3.8]{Gin08} to get  that
 $(M_r)_{x_r}$ is flat over $\C[\finm^*_{\chi_r}]$.
This completes the proof.
\end{proof}

\subsection{A homology vanishing}\label{subsection Homology vanishing}
 Let $\{y_i\}$ be a basis of $\finm[t\inv]t\inv$.
Set
\begin{align*}
t_i=y_i-\chi(y_i),
\end{align*}
so that
\begin{align}
I_{\infty}=\sum_i t_i \C[\fing^*_{\infty}].
\label{eq:functions-of-f+m}
\end{align}

For a 
module $M$
over $\C[\finm^*_{\infty}]$ as a ring,
let 
$H_{\bullet}^{\on{Kos}}(\finm[t\inv]t\inv,
M\* \C_{-\chi}
)$
denote
 the homology of the Koszul complex
associated with  $M$
and the sequence
$t_1,t_2,\dots $.
By definition 
\begin{align*}
 H^{\on{Kos}}_{0}(\finm[t\inv]t\inv, M\* \C_{-\chi})
=M/I_{\infty} M.
\end{align*}
\begin{Pro}\label{Pro:vanishing-of-homology}
We have
 $H_i^{\on{Kos}}(\finm[t\inv]t\inv, M\* \C_{-\chi})=0$ for
 $i>0$
and
 any abject $M$ of $\Cat$.
\end{Pro}
Let
\begin{align*}
 M_{\on{tor}}=\{m\in M; m_x=0\text{ for all }x\in 
\chi_{\infty}+(\finm^{\bot})_{\infty} \}.
\end{align*}
Then $M_{\on{tor}}$ is a 
submodule of $M$
over $\C[\finm_{\infty}^*]$ as a ring.

The following assertion follows from  \cite[Proposition 17.14b]{Eis95}.
\begin{Lem}\label{Lem:homology-of-torsion-module-is-zero}
$H^{\on{Kos}}_{\bullet}(\finm[t\inv]t\inv,M_{\on{tor}}\* \C_{-\chi})= 0$.
\end{Lem}
\begin{proof}[Proof of Proposition \ref{Pro:vanishing-of-homology}]
We may assume that $M$ is finitely generated
as the homology functor commutes with the inductive limits.
 Set
$\overset{\circ}{M}=M/M_{\on{tor}}$.
By Lemma \ref{Lem:homology-of-torsion-module-is-zero}
it suffices to show that
\begin{align*}
H^{\on{Kos}}_i(\finm[t\inv]t\inv,
\overset{\circ}{M}\* \C_{-\chi})=0\text{ for }i>0.
\end{align*}
Clearly, we may assume that
$\overset{\circ}{M}\ne 0$.
It  is enough to show that
$\{t_1,t_2,\dots \}$ is  a 
regular sequence on  $\overset{\circ}{M}$,
that is,
$t_r$ is not a zero divisor of $\overset{\circ}{M}/\sum_{i=1}^{r-1}t_i
\overset{\circ}{M}$ for all $r\geq 0$.

Let $a\in \overset{\circ}{M}$ such that
$t_r a\in \sum_{i=1}^{r-1}t_i \overset{\circ}{M}$.
By localizing it at $x\in \chi_{\infty}+(\finm^{\bot})_{\infty}$,
we get that
$t_r a_x\in \sum_{i=1}^{r-1}t_i \bar{M}_x$.
Because $\overset{\circ}{M}_x=M_x$
for $x\in \chi+(\finm^{\bot})_{\infty}$,
$\overset{\circ}{M}_x$ is flat over $\C[\finm_{\infty}^*]_{\chi_{\infty}}$
by Proposition \ref{Pro:flatness}.
In particular
$\{t_1,t_2\dots,\}$ is  a regular sequence on
$\overset{\circ}{M}_x$ (\cite[Exercise 18.18]{Eis95}).
This forces  $a_x=0$ for any $x\in \chi_{\infty}+(\finm^{\bot})_{\infty}$.
Hence we get that  $a=0$.
This completes the proof.
\end{proof}

\subsection{A cohomology vanishing}
\label{subsection;A cohomology vanishing}
Let $M$ be an object of $\Cat$.
Because 
 the action of $\finn[t]\subset \fing[t]$ on 
$M$ stabilizes $I_{\infty} M$,
$M/I_{\infty}M$ is a $\finn[t]$-module.
Thus
$\left(M/I_{\infty} M\right)^{\finn[t]}$ 
is a module over $\C[\Sl_{\infty}]
=(\C[\fing^*_{\infty}]/I_{\infty})^{\finn[t]}$ as a ring\footnote{%
We will see in (the proof of) 
Theorem \ref{Th:BRST-vanishing-for-assicaited-graded}
that $(M/I_{\infty}M)^{\finn[t]}$
is actually a $\C[\Sl_{\infty}]$-module as a vertex Poisson algebra.}.

The proof of the following assertion is based on \cite[6.2]{GanGin02}.
\begin{Pro}\label{Pro:cohomology-vanishing-gr}
Let $M\in \Cat$.
Then we have 
 $H^i_{\on{Lie}}(\finn[t], M/I_{\infty} M)=0$
for $i>0$.
\end{Pro}
\begin{proof}
 We show that
the multiplication  map
\begin{align}
 \varphi:\C[\chi_{\infty}+(\finm^{\bot})_{\infty}]
\*_{\C[\Sl_{\infty}]}
(M/I_{\infty} M)^{\finn[t]}
\ra M/I_{\infty} M
\label{eq:isonap-gr-inthe-proof}
\end{align}
is an isomorphism of 
 $\finn[t]$-modules.
This proves the assertion
since
\begin{align}
& H^i_{\on{Lie}}\left(\finn[t], 
\C[\chi_{\infty}+(\finm^{\bot})_{\infty}]
\*_{\C[\Sl_{\infty}]}
(M/I_{\infty} M)^{\finn[t]}\right)
\nno \\
&=H^i_{\on{Lie}}\left(\finn[t],\C[N_{\infty}]
\*
(M/I_{\infty} M)^{\finn[t]}\right)
\quad\text{(by Proposition \ref{Pro:fundamental})}\nno\\
&
=H^i_{\on{Lie}}\left(\finn[t],\C[N_{\infty}]\right)
\*
(M/I_{\infty} M)^{\finn[t]}
\nno \\
&=
\begin{cases}
(M/I_{\infty} M)^{\finn[t]}&\text{for $i=0$},\\
0&\text{for $i>0$}.
  \end{cases}
\label{eq:2009:08:21:15:19}
\end{align}

To prove $\varphi$ is an isomorphism
we have only to show that
$(\ker \varphi)^{\finn[t]}=0$
and $
(\coker \varphi)^{\finn[t]}=0$
because $\finn[t]$ acts locally nilpotently
on $\ker\varphi$
and $\coker \varphi$.

Applying the left exact functor 
$H^0_{\on{Lie}}(\finn[t],?)$ 
to the exact sequence
$0\ra \ker \varphi \ra 
\C[\chi_{\infty}+(\finm^{\bot})_{\infty}]
\*_{\C[\Sl_{\infty}]}
(M/I_{\infty} M)^{\finn[t]}
\ra M/I_{\infty} M$,
 we obtain the exact sequence
$0\ra (\ker \varphi)^{\finn[t]}
\ra (M/I_{\infty} M)^{\finn[t]}
\ra (M/I_{\infty} M)^{\finn[t]}$.
This gives that  $(\ker \varphi)^{\finn[t]}=0$.
Similarly, the exact sequence 
$0\ra \C[\chi_{\infty}+(\finm^{\bot})_{\infty}]
\*_{\C[\Sl_{\infty}]}
(M/I_{\infty} M)^{\finn[t]}
\ra M/I_{\infty} M\ra \coker \varphi \ra 0$
gives
the exact sequence
\begin{align*}
 0\ra (M/I_{\infty} M)^{\finn[t]}\ra 
(M/I_{\infty} M)^{\finn[t]}\ra 
(\coker \varphi)^{\finn[t]}\ra 0
\end{align*}
by the vanishing of $H^1$
in (\ref{eq:2009:08:21:15:19}).
This gives that
$(\coker \varphi)^{\finn[t]}=0$,
completing the proof.
\end{proof}

The following assertion follows immediately 
from
 Proposition \ref{Pro:vanishing-of-homology}
and Proposition \ref{Pro:cohomology-vanishing-gr}.
 \begin{Th}\label{Th:exactness-of-BRST-reduction-of-vertex-Poisson}
The assignment
\begin{align*}
M\mapsto (M/I_{\infty}M)^{\finn[t]}
\end{align*}
yields an exact functor
from $\mc{C}$ to the category of $\C[\Sl_{\infty}]$-modules.
 \end{Th}

Now let $M$ be an object  of $\KL_k$.
Recall that
$\gr^F M\in \Cat$,
and  the subspace
$\bar M=M/F^1 M\subset \gr^F M$ is a 
Poisson $\C[\fing^*]$-module.
Since
$I_{\infty}\bar M\cap \bar M=I \bar M$
(recall $I=I_0$),
we have the inclusion
$\bar M/I\bar M\hookrightarrow \gr^F M/(I_{\infty}\gr^F M)$.
As the action
of $\fing[t]t$ is trivial on $\bar M$,
This induces the embedding
\begin{align*}
 (\bar M/I\bar M)^{\finn}\hookrightarrow (\gr^F M/I_{\infty}\gr^F M)^{\finn[t]}.
\end{align*}
On the other hand,
the isomorphism 
\eqref{eq:isonap-gr-inthe-proof}
restricts to the isomorphism
\begin{align}
\bar M/I \bar M\cong \C[\chi+\finm^{\bot}]\*_{\C[\Sl]}(\bar M/I
  \bar M)^{\finn}.
\label{eq:iso-for-fd-cases}
\end{align}

 \begin{Pro}\label{Pro:generating-subspace-gr}
\begin{enumerate}
\item Let $M$ be an object of 
$\KL_k^{\Delta}$.
Then
\begin{align*}
(\gr^F M/I_{\infty}\gr^F M)^{\finn[t]}\cong  
\C[\Sl_{\infty}]\*_{\C[\Sl]}(\bar M /I \bar M)^{\finn}
\end{align*}as $\C[\Sl_{\infty}]$-modules.
\item 
Let $M$ be a finitely generated object of $\KL_k$.
Then $(\gr^F M/I_{\infty}\gr^F M)^{\finn[t]}$ is
generated by the subspace
$(\bar M/I\bar M)^{\finn}$
over $\C[\Sl_{\infty}]$.
\end{enumerate}
 \end{Pro}
 \begin{proof}
(i)
By Lemma \ref{Lem:surjection-as-gr-delta},
we have
$
 \gr^F M\cong \C[\fing^*_{\infty}]\*_{\C[\fing^*]}\bar M
$.
It follows from
\eqref{eq:iso-for-fd-cases}.
that
\begin{align*}
&\gr^F M/(I_{\infty} \gr^F M)
\cong (\C[\fing^*_{\infty}]/I_{\infty}
)\*_{\C[\fing^*]/I}(\bar M/I\bar M)\\
&\cong \C[\chi_{\infty}+(\finm^{\bot})_{\infty}]\*_{\C[\chi+\finm^{\bot}]}
(\bar M/I \bar M)
\cong
\C[\chi_{\infty}+(\finm^{\bot})_{\infty}]\*_{\C[\Sl]}
(\bar M/I \bar M)^{\finn}.
\end{align*}
Hence
$
( \gr^F M/I_{\infty} \gr^F M)^{\finn[t]}
\cong \C[\Sl_{\infty}]\*_{\C[\Sl]}(\bar M/I \bar M)^{\finn}
$.
(ii)
Since $M$ is finitely generated,
there is an exact sequence
$P\ra M\ra 0$ with  $P\in \KL^{\Delta}_k$.
By Theorem \ref{Th:exactness-of-BRST-reduction-of-vertex-Poisson},
this induces the exact sequence
$(\gr^F P/I_{\infty}\gr^F P)^{\finn[t]}
\ra (\gr^F M/I_{\infty}\gr^F M)^{\finn[t]}
$,
which restricts to the surjection
$(\bar P/I\bar P)^{\finn}
\ra (\bar M/I \bar M)^{\finn}
\ra 0$.
Therefore
the assertion follows from (i).
 \end{proof}

\section{Associated varieties 
of module over affine vertex algebras and affine $W$-algebras}
\label{section:associated-varieties}
The main purpose of 
 this section is 
to establish  the relation between the  associated varieties
of modules over $\affg$ and those over $\Wg{k}$
(Theorem \ref{Th:BRST-reduction-of-varieties}).
Throughout  this section the level
$k$ is an arbitrary complex number unless otherwise stated.
\subsection{The BRST complex of the generalized quantized Drinfeld-Sokolov
reduction}
For the detail of this subsection we refer the reader to
\cite{KacRoaWak03,KacWak04} and  also to \cite{Ara05}\footnote{In
 \cite{Ara05} $\Fneu$ and $ \Lamsemi{\bullet}$ were denoted by
 $\mathscr{F}^{\on{ne}}(\chi)$ and $\mathscr{F}(L\fing_{>0})$, respectively.}.

Let
$\{u_i; i=1,\dots, \dim \finn\}$ be a homogeneous basis of $\finn$ 
(defined in \eqref{eq:nilpotent-subagebras})
with respect to the grading \eqref{eq:good-grading}
such that $\{u_1,\dots,u_{\dim \fing_{1/2}}\}$ gives a basis of 
$\fing_{1/2}$.
Let $c_{ij}^k$ be the structure constant:
$[u_i,u_j]=\sum_{k}c_{ij}^k u_k$.

Denote by $\Fneu$ 
 the vertex algebra of {\em neutral free superfermions}
associated with $\fing_{1/2}$,
or the $\beta\gamma$-system of rank $\frac{1}{2}
\dim \fing_{1/2}$,
which 
is freely generated by the fields
$\phi_i(z)$ with $i=1,\dots, \dim \fing_{1/2}$ (corresponding to 
the basis $\{u_1,\dots, u_{\dim\fing_{1/2}}\}$)
satisfying the OPE
\begin{align*}
 \phi_i(z)\phi_j(w)\sim \frac{\chi([u_i,u_j])}{z-w}.
\end{align*}

Let
$\Lamsemi{\bullet}=\bigoplus_{i\in \Z}\Lamsemi{i}$
be  the vertex (super)algebra of the {\em semi-infinite forms}
associated with $\finn[t,t\inv]$,
which 
is a vertex superalgebra freely generated by the odd fields
$\psi_1(z),\dots, \psi_{\dim \finn}(z)$
(corresponding to the
basis $\{u_i\}$ of $\finn$)
and 
$\psi_1^*(z),\dots, \psi_{\dim \finn}^*(z)$
(corresponding to the
dual basis $\{u_i^*\}$ of $\{u_i\}$)
satisfying the OPE
\begin{align*}
 \psi_i(z)\psi_j^*(w)\sim \frac{\delta_{ij}}{z-w},\quad
 \psi_i(z)\psi_j(w)\sim 
 \psi_i^*(z)\psi_j^*(w)\sim 0.
\end{align*}
The space $\Lamsemi{\bullet}$
is graded by $\deg (\psi_i)_{(n)}=-1$,
$\deg (\psi_i^*)_{(n)}=1$
and $\deg \1=0$, where $\1$ is the vacuum vector.
We have 
\begin{align*}
  \Lamsemi{\bullet}=\bigoplus_{i\in \Z}\Lamsemi{i}
\cong \bigwedge\nolimits_{op}^{\bullet}(\finn[t\inv]t\inv)
\* \bigwedge\nolimits^{\bullet}(\finn^*[t\inv]),\\
\Lamsemi{i}=\bigoplus_{t-s=i}\bigwedge\nolimits^{s}(\finn[t\inv]t\inv)
\* \bigwedge\nolimits^{t}(\finn^*[t\inv]).
\end{align*}
Here
$(\psi_{i})_{(-n)}$ is identified with $u_it^{-n}\in\finn[t\inv]t\inv$
and $(\psi_{i}^*)_{(-n)}$ is identified with $u_i^*t^{-n}\in
\finn^*[t\inv]t\inv$,
and for a vector space $V$, $\bigwedge^{\bullet}_{op} V$
denotes the Grassmann  algebra 
$\bigwedge^{\bullet} V$ with opposite grading,
that is,
$\bigwedge^i_{op}V=\bigwedge^{-i}V$.

For any  $\Vg{k}$-module $M$,
set \begin{align*}
 &C(M)=M\* \Fneu\* \Lamsemi{\bullet}
=\bigoplus_{i\in \Z}C^i(M),\quad
C^i(M)=M\* \Fneu\* 
\Lamsemi{i}.
\end{align*}
Note that
$C(\Vg{k})$ is a vertex (super)algebra
since it is  a tensor product of vertex (super)algebras,
and $C(M)$ is naturally a module over $C(\Vg{k})$.

Let $Q(z)$ be the filed on $C(\Vg{k})$  defined by
\begin{align*}
 &Q(z)=\sum_{n\in \Z}Q_{(n)}z^{-n-1}\\
&:=\sum_{i=1}^{\dim \finn}
(u_i(z)+\phi_i(z))\psi_i(z)-\frac{1}{2}\sum_{1\leq i,j,k\leq \dim
 \finn}
c_{ij}^k \psi_i^*(z)\psi_j^*(z)\psi_k(z).
\end{align*}
Here 
we have omitted the tensor product symbol
and have put
$\phi_i(z)=\chi(u_i)$
for $i>\dim \fing_{1/2}$.
We have $Q(z)Q(w)\sim 0$,
and therefore, 
\begin{align*}
Q_{(0)}^2=0.
\end{align*}
It follows that
 $(C(M),Q_{(0)})$ is a cochain complex for any $\Vg{k}$-module $M$.

Let
\begin{align*}
& \BRS{\bullet}{M}\teigi H^{\bullet}(C(M))=H^{\bullet}(C(M),Q_{(0)}),\\
& \Wg{k}\teigi \BRS{0}{\Vg{k}}.
\end{align*}
Then
$\Wg{k}$
inherits the vertex algebra structure
from
$C(\Vg{k})$.
The vertex algebra $\Wg{k}$
is called the {\em (universal) $W$-algebra
associated with $(\fing,f)$ at level $k$}.
It is a
 conformal vertex algebra  
with central charge 
\begin{align}
 \dim \fing_0-\frac{1}{2}\dim \fing_{\frac{1}{2}}
-\frac{12}{k+h\che}|\rho-(k+h\che)\x|^2
\label{eq:cc}
\end{align}
provided that $k\ne-h\che$,
where $x_0$ is defined in 
\eqref{eq:dynkin-h}.

For a $\Vg{k}$-module
$M$
(or equivalently 
a smooth $\affg$-module of level $k$),
the $C(\Vg{k})$-module structure on $C(M)$
induces
the $\Wg{k}$-module 
structure on  $\BRS{\bullet}{M}$.

\subsection{Weights of $\BRS{\bullet}{M}$}
By abuse of notation we set
\begin{align*}
 H_{\new}:=-D-\x\in \affh.
\end{align*}
The operator $H_{\new}$
gives a $\frac{1}{2}\Z$-grading on $C(\Vg{k})$,
where $x_0\in \affh$ acts on $C(\Vg{k})$ diagonally.
Because  $H_{\new}$ commutes with $Q_{(0)}$,
the vertex algebra $\Wg{k}$
is $\frac{1}{2}\Z$-graded
by the Hamiltonian $H_{\new}$.
In fact 
$\Wg{k}$ is 
 positively graded
by the Hamiltonian $H_{\new}$
(\cite{FreBen04,KacWak04}).
We denote by $\W_k(\fing,f)$ the unique simple graded quotient of 
$\Wg{k}$.

Similarly,
for $M\in \KL_k$,
the $\Wg{k}$-module
$\BRS{\bullet}{M}$
is graded by $H_{\new}$.
Thus
\begin{align*}
\BRS{\bullet}{M}=\bigoplus_{d\in \C}\BRS{\bullet}{M}_d,
\end{align*}
and the subspace
$ \BRS{\bullet}{M}_d$
is the cohomology of the subcomplex
$
C(M)_d
$.
Here and below,
$C(M)$ is considered as a graded 
$C(\Vg{k})$-module by the Hamiltonian
$H_{\new}$:
$C(M)_d=\{m\in C(M); H_{\new}m =d m\}$.

\subsection{The right exactness of the functor
$\BRS{0}{?}$}
We shall consider the functor
\begin{align}
\BRS{0}{?}: \KL_k\ra \Wg{k}\Mod,\quad
M\mapsto \BRS{0}{M}.
\label{eq:brst-functor}
\end{align}
\begin{Th}[{\cite[15.2]{FreBen04}, \cite[Theorem 6.3]{KacWak04}}]\label{Th:Kac-Wakimoto-vanising}
  Let $M\in \KL_k^{\Delta}$.
Then
$\BRS{i}{M}=0$ 
for all $i\ne 0$,
$\BRS{i}{M}_d$ is finite-dimensional for all $d$,
and $\BRS{0}{M}$ is positively graded.
 \end{Th}

\begin{Th}\label{Th:right-exactness}
 Let $M\in \KL_k$.
Then
\begin{align*}
\BRS{i}{M}=0\quad \text{for all }i> 0.
\end{align*}
Therefore 
the functor $\BRS{0}{?}:\KL_k\ra \Wg{k}\Mod$
is right exact.
\end{Th}
\begin{proof}
By Theorem \ref{Th:Kac-Wakimoto-vanising},
the assertion can be proven in the similar manner as 
\cite[Theorem 3.5 (3)]{AraMal}.
\end{proof}
\begin{Co}\label{Co:finite-dimensionality}
Let $M$ be a finitely generated object of $\KL_k$.
Then $\BRS{0}{M}_d$ is finite-dimensional for all $d\in \C$
and $\BRS{0}{M}$ is positively graded.
\end{Co}

 \subsection{Cohomology vanishing
of associated graded BRST complexes}
The following two assertions are easily seen from the definition.
\begin{Lem}
\begin{enumerate}
 \item  We have
 $\Ring{\Fneu}\cong \C[\fing_{1/2}^*]$
as Poisson algebras,
where
the Poisson structure of $\C[\fing_{1/2}^*]$
is given by the symplectic form
$ \fing_{1/2} \times \fing_{1/2} \ra \C,
(x,y)\mapsto (f,[x,y])$.
\item 
We have
 $\gr^F \Fneu\cong \C[(\fing_{1/2}^*)_{\infty}]$ as vertex Poisson algebras.
\end{enumerate}
\end{Lem}
\begin{Lem}
\begin{enumerate}
 \item We have
$\Ring{\Lamsemi{\bullet}}\cong
\bigwedge_{op}\nolimits^{\bullet}(\finn)
\* \bigwedge\nolimits^{\bullet}(\finn^*)$
as Poisson superalgebras,
where the Poisson superalgebra  structure
on the right-hand-side
is given by
\begin{align*}
\{\psi_f,\psi_x\}=f(x),\quad 
 \{\psi_x,\psi_{x'}\}=\{\psi_f,\psi_{f'}\}=0
\end{align*}
for $x,x'\in \finn$,
$f,f'\in \finn^*$.
Here $\psi_x$ (respectively $\psi_f$)
denotes the element of $\bigwedge\nolimits^{-1}_{op}(\finn)$
(respectively $\bigwedge\nolimits^{1}(\finn^*)$)
corresponding to $x\in \finn$
(respectively to $f\in \finn^*$).
\item
The vertex Poisson algebra 
structure
of 
$
 \gr^F \Lamsemi{\bullet}
$
is the level $0$ vertex Poisson superalgebra
induced from the Poisson superalgebra structure of
$R_{\Lamsemi{\bullet}}=\bigwedge\nolimits_{op}^{\bullet}(\finn)
\* \bigwedge\nolimits^{\bullet}(\finn^*)$.
\end{enumerate} 
\end{Lem}

Let $M\in \KL_k$.
The Li filtration  of $C(M)$ is given by 
\begin{align*}
 F^p C(M)=\sum_{r+s+t=p}F^r M\* F^s \Fneu\* F^t \Lamsemi{\bullet}.
\end{align*}
More generally,
for a  compatible filtration $\Gamma^\bullet M$ of $M$,
\begin{align}
 \Gamma^p C(M):=\sum_{r+s+t=p}\Gamma^r M\* F^s \Fneu\* F^t
 \Lamsemi{\bullet}
\label{eq:filtration-of-complexes}
\end{align}
defines a compatible filtration of 
a $C(\Vg{k})$-module $C(M)$.
We have
\begin{align*}
 \gr^\Gamma C(M)=\gr^\Gamma M\* 
\C[(\fing_{1/2}^*)_{\infty}]\*
\Lamsemi{\bullet}.
\end{align*}
as vertex Poisson $\gr^F C(\Vg{k})$-modules.
Here we have identified 
$\gr^F\Lamsemi{\bullet}$
with $\Lamsemi{\bullet}$.

We have  $Q_{(0)}\Gamma^p C(M)\subset \Gamma^p C(M)$,
and thus,
$(\gr^\Gamma C(M),Q_{(0)})$
 is a cochain complex.
The cohomology
$H^{\bullet}(\gr^F C(\Vg{k}))$ inherits  the graded vertex Poisson (super) algebra
structure from $\gr^F C(\Vg{k})$,
and
the vertex Poisson $\gr^F C(\Vg{k})$-module structure 
on $\gr^\Gamma C(M)$
induces the
vertex Poisson $H^{\bullet}(\gr^F C(\Vg{k}))$-module structure on
$H^{\bullet}(\gr^F C(M))$.

\begin{Th}\label{Th:BRST-vanishing-for-assicaited-graded}
\begin{enumerate}
 \item 
We have
$H^i(\gr^F C(\Vg{k}))=0$ for all $i\ne 0$,
and
\begin{align*}
	H^0(\gr^F C(\Vg{k}))
\isomap \C[\Sl_{\infty}]
     \end{align*}
as vertex Poisson algebras.
\item  Let   $M\in \KL_k$,
$\Gamma^\bullet M$ a  compatible filtration of $M$.
Then
\begin{align*}
	H^i(\gr^\Gamma C(M))\cong 
\begin{cases}
	(\gr^\Gamma M/I_{\infty}\gr^\Gamma M)^{\finn[t]}
&\text{for }i=0,\\0&\text{for }i\ne 0
       \end{cases}
       \end{align*}
as modules over $H^0(\gr^F C(\Vg{k}))=\C[\Sl_{\infty}]$.
\end{enumerate}
\end{Th}
\begin{proof}
Let $M\in \KL_k$.
Set 
\begin{align*}
\bar C=\gr^\Gamma  C(M)
=\bigoplus_{i\leq 0,j\geq 0}\bar C_{i,j},
\end{align*}
where 
\begin{align*}
\bar{C}_{i,j}=\gr^\Gamma  M\* 
\C[(\fing_{1/2}^*)_{\infty}]\*
\bigwedge\nolimits^{-i}(\finn[t\inv]t\inv)
\* \bigwedge\nolimits^{j}(\finn^*[t\inv]).
\end{align*}
We see that
 the operator $Q_{(0)}$ on $\bar C$
decomposes as 
 \begin{align*}
 Q_{(0)}=\bar{Q}^- +\bar{Q}^+,
\end{align*}
where
$\bar {Q}^- \bar C_{i.j}\subset \bar C_{i+1,j}$
and 
$\bar {Q}^+ \bar C_{i.j}\subset \bar C_{i,j+1}$.
Because $Q_{(0)}^2=0$,
we get that
\begin{align*}
(\bar Q^-)^2=(\bar Q^+)^2=\{\bar Q^-,\bar Q^+\}=0.
\end{align*}

Consider 
the
 spectral sequence 
\begin{align}
E_r\Rightarrow H^{\bullet}(\bar C)
\label{eq:spectral-sequence1}
\end{align}
whose zeroth differential is $\bar Q^-$
and first differential is $\bar Q^+$.
This is a converging spectral sequence because
the complex $\bar C$ is a direct sum of subcomplexes
$\Gamma ^p C(M)/\Gamma^{p+1}C(M)$ 
with $p\in \Z$, 
and 
$(\Gamma^p C(M)/\Gamma^{p+1}C(M))\cap \bar C^{i,j}=0$
for $j\gg 0$ by (\ref{eq:filtration-of-complexes}).

By definition the $E_1$-term is the cohomology of the  complex 
$(\bar C,\bar Q^-)$,
which 
 is identical to  the Koszul complex
$\gr^{\Gamma}M\* \C[(\fing_{1/2})_{\infty}^*]$
associated with the
the sequence
$t_1,t_2,\dots, s_1,s_2,\dots$,
where $t_i$ is defined in
\S \ref{subsection
 Homology vanishing}
and $\{s_1,s_2,\dots\}$ is a basis of
 $\fing_{1/2}[t\inv]t\inv$.
To compute this we consider the
Hochschild-Serre spectral sequence
associated with the subsequence 
$t_1,t_2,\dots $.
Let $\dot{E}_r^{p,q}$ denote this spectral sequence.
Proposition \ref{Pro:vanishing-of-homology}
gives that
\begin{align}
 \dot{E}_1^{p,q}=\begin{cases}
		       (\gr^\Gamma M/I_{\infty}\gr^\Gamma M)\*
\C[(\fing_{1/2}^*)_{\infty}]\* \Lam^{\bullet}(\finn^*[t\inv])\*
\Lam^{-p}(\fing_{1/2}[t\inv]t\inv)
&
\text{if }q=0\\
0&\text{if }q\ne 0.
		      \end{cases}
\label{eq:E1-Poisson1}
\end{align}
 Because
$\Lam^{-p}((\fing_{1/2}[t\inv]t\inv)$
is a free $\fing_{1/2}[t\inv]t\inv$-module
we obtain
\begin{align}
 \dot{E}_2^{p,q}\cong \begin{cases}
		       (\gr^\Gamma M/I_{\infty}\gr^\Gamma M)
\* \Lam^{\bullet}(\finn^*[t\inv])
&
\text{if }p=q=0\\
0&\text{otherwise}.
		      \end{cases}
\label{eq:E1-Poisson2}
\end{align}
Therefore
 the spectral sequence
$\dot{E}_r$ collapses at $\dot{E}_2=\dot{E}_{\infty}$
and we get
\begin{align}
 {E}_1^{p,q}\cong \begin{cases}
		       (\gr^\Gamma M/I_{\infty}\gr^\Gamma M)
\* \Lam^{p}(\finn^*[t\inv])
&
\text{if }q=0\\
0&\text{if }q\ne 0.
		      \end{cases}
\label{eq:E1-Poisson}
\end{align}

Next let us
 compute the $E_2$-term
of the spectral sequence \eqref{eq:spectral-sequence1}.
We find 
 from the explicit form
of $\bar{Q}^-$  and (\ref{eq:E1-Poisson})
that
the complex 
$(H^0(\bar{C},\bar{Q}_-),\bar{Q}_+)$
is identical to the the Chevalley complex 
for computing Lie algebra cohomology
$H^{\bullet}_{\on{Lie}}(\finn[t],\gr^\Gamma M/I_{\infty}
\gr^\Gamma M)
$ considered in \S \ref{subsection;A cohomology vanishing}.
Therefore
Proposition  \ref{Pro:cohomology-vanishing-gr}
gives that
\begin{align*}
 H^i(H^j(\bar C,\bar{Q}^-),\bar Q^+)\cong \begin{cases}
				      (\gr^\Gamma /I_{\infty}\gr^\Gamma
				      M)^{\finn[t]}&\text{for
				      $i=j=0$},\\
0&\text{otherwise.}
				     \end{cases}
\end{align*}
Hence the spectral sequence collapses at $E_2=E_{\infty}$.
Therefore
\begin{align}
 H^0(\gr^{\Gamma}C(M))
\cong \begin{cases}
				      (\gr^\Gamma /I_{\infty}\gr^\Gamma
				      M)^{\finn[t]}&\text{for
				      $i=j=0$},\\
0&\text{otherwise.}
				     \end{cases}
\label{eq:conseq-of-double-complex}
\end{align}
Note that
 $E_{\infty}^{p,q}$ lies entirely in $E_2^{0,0}$,
and therefore
the corresponding filtration of $H^0(\gr^{\Gamma} C(M))$
is trivial.

Applying \eqref{eq:conseq-of-double-complex} to $M=\Vg{k}$
and $\Gamma=F$,
we obtain
the isomorphism
\begin{align}
\Phi: 
H^0(\gr^F C(\Vg{k}))
\isomap \C[\Sl_{\infty}].
\label{eq:Phi}
\end{align}
The triviality 
of the filtration associated with the spectral sequence considered above
implies that $\Phi$ is an isomorphism of 
differential algebras.

To see that $\Phi$ is an isomorphism of vertex Poisson
 algebras,
consider the subcomplex 
$\Ring{C(\Vg{k})}=F^0 C(\Vg{k})/F^1 C(\Vg{k})
\subset \gr^F C(\Vg{k})$.
We have
\begin{align*}
\Ring{C(\Vg{k})}
= \Ring{\Vg{k}}\* \Ring{\Fneu}\*\Ring{\Lamsemi{\bullet}}
\end{align*}
as Poisson (super)algebras,
and
the vertex Poisson algebra structure of
$H^{0}(\gr^\Gamma C(M))$ restricts to the Poisson algebra
structure of
$H^0(\Ring{C(\Vg{k})})$.
Now
observe that the assignment
 $\Ring{C(\Vg{k})}\ra H^0(\Ring{\Vg{k}})$
is exactly the BRST realization 
\cite{KosSte87} of 
Hamiltonian reduction 
described in  \cite[\S 3.2]{GanGin02}.
Hence
the restriction of 
$\Phi$ 
gives the isomorphism
\begin{align*}
\Phi|_{H^0(\Ring{C(\Vg{k})})}: 
H^0(R_{C(\Vg{k})})\isomap \C[\Sl]=(\C[\fing^*]/I_0)^{\finn}
\end{align*}
of Poisson algebras.
Since $\C[\Sl_{\infty}]$ is generated by its subring  $\C[\Sl]$
as a differential algebra,
the ring $H^0(\gr^F C(\Vg{k}))$
is also  generated by 
$H^0(\Ring{C(\Vg{k})})$ as a differential algebra.
By the property of the Li filtration
we have
$a_{(n)}b=0$ for $n>0$, $a,b\in H^0(\Ring{C(\Vg{k})})$.
This means that
the vertex Poisson algebra structure of
$H^0(\gr^F C(\Vg{k}))$ coincides with that of
$\C[\Sl_{\infty}]$ on the
generating subspaces
$H^0(\Ring{C(\Vg{k})})\cong \C[\Sl]$.
But
we have shown in 
\cite[Proposition 2.3.1]{Ara12} that
this uniquely determines 
the whole vertex Poison algebra structure.
This completes the proof.
\end{proof}

It is clear from the proof that
the isomorphism of Theorem \ref{Th:BRST-vanishing-for-assicaited-graded} (i)
restricts to the isomorphism
\begin{align*}
H^0(R_{C(\Vg{k})})\iso \C[\mc{S}]
\end{align*}
of Poisson algebras.
Also,
 in the case that $\Gamma=F$,
the isomorphism of Theorem \ref{Th:BRST-vanishing-for-assicaited-graded}
(ii)
restricts to the
isomorphism
\begin{align}
H^0(C(M)/F^1 C(M))\cong (\bar M/I\bar M)^{\finn}
\end{align}
as Poisson $\C[\Sl]$-modules.
Hence Proposition \ref{Pro:generating-subspace-gr} (ii)
gives the following.

 \begin{Pro}\label{Pro:generating-subspace}
Let $M$ be a finitely generated object in $\KL_k$.
Then
$H^0(\gr^F C(M))$ is generated by
the subspace
$H^0( C(M)/F^1 C(M))\cong (\bar M/I\bar M)^{\finn}
$ over  $\C[\mathcal{S}_{\infty}]$.
 \end{Pro}

Since
$\Hnew$ naturally acts on
$H^0(\gr^{\Gamma}C(M))$,
we have 
 the decomposition
\begin{align*}
H^0(\gr^{\Gamma}C(M))
=\bigoplus_{d\in \C}H^0(\gr^{\Gamma}C(M))_d.
\end{align*}
Here and below,
\begin{align*}
 M_d=\{m\in M; H_{\new} m=d m\}
\end{align*}
for any $H_{\new}$-module $M$.
\begin{Pro}\label{Pro:finite-dimensionality-of-associated-graded}
\begin{enumerate}
\item
For an object  $M$ of $\KL_k$
and $d\in \C$,
the dimension of
$H^0(\gr^{\Gamma}C(M))_d$
is independent of  a compatible
     filtration
$\Gamma^{\bullet}M$ of $M$.
\item Let $M$ be a finitely generated object of $\KL_k$,
$\Gamma^\bullet M$ a compatible filtration.
Then $H^0(\gr^{\Gamma }C(M))_d$ is finite-dimensional for all $d\in \C$.
\end{enumerate}
\end{Pro}
\begin{proof}
(i) 
follows from 
 the Euler-Poincar\'{e} principle
and 
Theorem  \ref{Th:BRST-vanishing-for-assicaited-graded}.
(ii)
By (i),
we may assume that $\Gamma=F$, the Li filtration.
But then the assertion follows from
Proposition \ref{Pro:finite-dimensionality-of-associated-graded}.
\end{proof}

\subsection{Strong   vanishing of BRST cohomology}
\label{subsection:strong-vanishing}
Let $M$ be a finitely generated object of 
$\KL_k$,
$\Gamma^\bullet M$ a compatible filtration.
 Let $M\in \KL_k$,
$\Gamma^\bullet M$ a  compatible filtration.
We have
$C(M)=\bigoplus_{d\in \C} C(M)_d$,
and  
$\Gamma^iC(M)_d=\Gamma^i C(M)\cap C(M)_d$
 gives a filtration of the subcomplex
$C(M)_d$.
We wish to show that,
for each $d\in \C$,
$\Gamma^\bullet C(M)_d$ 
is a regular filtration 
 in the sense 
of \cite[p.324]{CarEil56},
that is,
$H^{\bullet}(\Gamma^p C(M)_d)=0$ for a sufficiently large $p$.

Let 
 $\Gamma^\bullet \BRS{\bullet}{M}
=\bigoplus_d \Gamma^\bullet \BRS{\bullet}{M}_d$
be the filtration 
 of $\BRS{\bullet}{M}$
induced by
$\Gamma^\bullet C(M)_d$,
i.e., 
\begin{align*}
\Gamma^p \BRS{\bullet}{M}_d
=\im (H^{\bullet}(\Gamma^p C(M))_d
\ra H^{\bullet}(C(M))_d).
\end{align*}

The filtration $\Gamma^\bullet C(M)$ induces the filtration
$\Gamma^{\bullet}(C(M)/\Gamma^p C(M))$
of 
$C(M)/\Gamma^p C(M)$, which 
is certainly regular.
Let $\Gamma^p H^{\bullet}(C(M)/\Gamma^p C(M))$
 be the induced filtration
of $ H^{\bullet}(C(M)/\Gamma^p C(M))$.
The natural map
$C(M)\ra C(M)/\Gamma^p C(M)$
induces the surjection
\begin{align}
\gr^\Gamma C(M)\ra \gr^\Gamma (C(M)/\Gamma^p C(M)).
\label{eq:gr-gamma-quotient}
\end{align}

\begin{Pro}\label{Pro:vanishing-intermideate-1}
 Let $M\in \KL_k$,
$\Gamma^\bullet M$ a  compatible filtration.

\begin{enumerate}
 \item We have $H^i(C(M)/\Gamma^p C(M))=0$ for all $i\ne 0$
and $p\geq 1$,
and we have
$\gr^\Gamma H^0(C(M)/\Gamma^p C(M))\cong H^0(\gr^\Gamma (C(M)/\Gamma^p C(M)))$.
\item 
Let $d\in \C$
and
suppose that $H^{0}(\gr^{\Gamma}C(M))_d$ is finite-dimensional.
Then \eqref{eq:gr-gamma-quotient}
induced the isomorphism
\begin{align*}
H^0(\gr^\Gamma C(M))_d
\isomap
\gr^{\Gamma}H^0(C(M)/\Gamma^p C(M))_d
\end{align*}
for a sufficiently large $p$.
In particular
\begin{align*}
 \dim H^0 C(M)/\Gamma^p C(M))_d
=\dim H^0(\gr^\Gamma C(M))_d
\end{align*}
for a sufficiently large $p$.
\end{enumerate}
\end{Pro}
\begin{proof}
Set $C=C(M)$.
Because the complex
$\gr^\Gamma (C/\Gamma^p C)$ is a direct summand of
$\gr^\Gamma C$,
Theorem \ref{Th:BRST-vanishing-for-assicaited-graded}
gives  that
$H^i(\gr^\Gamma(C/\Gamma^p C))=0$
for $i\ne 0$.
It follows that
the  spectral sequence
associated with the filtration 
$\Gamma^{\bullet}(C/\Gamma^p C)$
collapses at $E_1=E_{\infty}$,
and we get that
\begin{align}
 &H^i(C/\Gamma^p C)=0\quad \text{for $i\ne 0$},
\label{eq:vanishing-gr-intermidiate1}
\\
& \gr^\Gamma H^0(C/\Gamma^p C)=H^0(\gr^\Gamma (C/\Gamma^p C))
=\bigoplus_{i=0}^{p-1}H^0(\Gamma^iC/\Gamma^{i+1}C).
\label{eq:vanishing-gr-intermidiate2}
\end{align}
This completes the proof.
\end{proof}

\begin{Pro}\label{Pro:regular?}
 Let $M\in \KL_k$,
$\Gamma^p M$ a  compatible filtration.
Then 
$\BRS{i}{M}=\Gamma^p \BRS{i}{M}$
for all $i\ne 0$ and $p\geq 1$.
\end{Pro}
\begin{proof}
We have the injection
\begin{align}
 \BRS{i}{M}/\Gamma^p \BRS{i}{M}
\hookrightarrow  H^{i}(C(M)/\Gamma^p C(M)).
\label{eq:basis-injection}
\end{align}But 
$H^{i}(C(M)/\Gamma^p C(M))=0$ for all $i\ne 0$
by Proposition  \ref{Pro:vanishing-intermideate-1}.
\end{proof}

\begin{Pro}\label{Pro:strainge-way-of-argument}
 Let $M$ be a finitely generated object of $\KL_k$.
Then
the following conditions are equivalent.
\begin{enumerate}
 \item $\BRS{i}{M}=0$ for all $i\ne 0$.
\item For any compatible filtration
$\Gamma^\bullet M$,
we have
$\Gamma^p \BRS{i}{M}=0$ for $i\ne 0$,
$p\geq 0$ and 
$\Gamma^p \BRS{0}{M}_d=0$ for $p\gg 0$, $d\in \C$.
\item
For any compatible filtration
$\Gamma^\bullet M$,
we have
$H^i(\Gamma^p C(M))=0$ for $i\ne 0$,
$p\geq 0$ and 
$H^0(\Gamma^p C(M))_d=0$ for $p\gg 0$, $d\in \C$.
\item There exists a  compatible filtration
$\Gamma^\bullet M$ such that
$H^{{\bullet}}(\Gamma^p C(M))_d=0$ for  $p\gg 0$,
 $d\in \C$.
\end{enumerate}
\end{Pro}
\begin{proof}
Put $C=C(M)$.
The direction (iii) $\Rightarrow$ (iv) is obvious.
The condition (iv) means that 
the filtration $\Gamma^\bullet C(M)_d$ is regular.
Therefore,
the associated spectral sequence converges to $\BRS{\bullet}{M}$,
and
this collapses at $E_1=E_{\infty}$
by virtue of  Theorem \ref{Th:BRST-vanishing-for-assicaited-graded}.
Hence (iv) implies (i).
Next let us  show that (i) implies (ii).
By Proposition \ref{Pro:regular?},
(i) gives $\Gamma^p\BRS{i}{M}=0$ for all $i\ne 0$,
$p\geq 0$.
It remains to show that
$\Gamma^p\BRS{0}{M}_d=0$ for $p\gg 0$.
Recall that
$H^i(\gr^{\Gamma}C(M))_d=0$ for $i\ne 0$
and 
 $H^0(\gr^{\Gamma}C(M))_d$ 
is finite-dimensional, see
Theorem \ref{Th:BRST-vanishing-for-assicaited-graded}
and Proposition \ref{Pro:finite-dimensionality-of-associated-graded}.
The cohomology vanishing  assumption (i) 
and the Euler-Poincar\'{e} principle
implies 
 that
\begin{align}
\dim \BRS{0}{M}_d=\dim H^0(\gr^{\Gamma}C)_d.
\label{eq:dim-is-equal}
\end{align}
Consider the composition
\begin{align}
 \BRS{0}{M}_d\twoheadrightarrow \BRS{0}{M}_d/\Gamma^p\BRS{0}{M}_d
\hookrightarrow H^0(C/\Gamma^p C)_d.
\label{eq:iso-iso?}
\end{align}
As
$H^0(C/\Gamma^p C)_d=H^0(\gr^{\Gamma} C)_d$
for $p\gg 0$
by
Proposition \ref{Pro:vanishing-intermideate-1} (ii),
\eqref{eq:dim-is-equal}
implies that
 the two maps in \eqref{eq:iso-iso?}
 must be isomorphisms
for $p\gg 0$,
proving
 (ii).
Finally, let us show (ii) implies (iii).
By the assumption and the injectivity
of
\eqref{eq:basis-injection},
$\BRS{i}{M}_d$ is embedded into 
$H^i(C/\Gamma^p C)_d$ for $p\gg 0$.
Since the latter vanishes
for $i\ne 0$
by Proposition \ref{Pro:vanishing-intermideate-1}
we get that
$\BRS{i}{M}=0$ for $i\ne 0$.
This together  with
Proposition \ref{Pro:vanishing-intermideate-1} (i)
implies that
the long exact sequence associated 
sequence $0\ra \Gamma^p C\ra C\ra C/\Gamma^p C\ra 0$
gives that
\begin{align*}
&H^i(\Gamma^p C)=0\quad\text{for }i\ne 0,1,\\
&0\ra H^0(\Gamma^p C)\ra \BRS{0}{M}\ra H^0(C/\Gamma^p C)\ra H^1(\Gamma^p C)\ra 0
\quad \text{(exact).}
\end{align*}
On the other hand 
in the proof of (ii)$\Rightarrow$(iii) 
we have
 proved that  the cohomology vanishing 
of
$\BRS{\bullet}{M}$
implies that
the maps in \eqref{eq:iso-iso?} are isomorphisms for $p\gg 0$.
Thus the middle map of the above exact sequence is an isomorphism.
This show that  $H^{0}(\Gamma^p C)_d=H^1(\Gamma^p C)_d=0$ for 
$p\gg 0$,
completing the proof.
\end{proof}

  \begin{Pro}\label{Pro:surjectivity-revised}
Let $M\in \KL_k$, $\Gamma^{\bullet}M$ a compatible filtration.
The natural map
$C(M)\ra C(M)/\Gamma^p C( M)$ induces the surjection
$\BRS{0}{M}\twoheadrightarrow  H^0(C(M)/\Gamma^p C(M))$ for all $p\geq 1$.
  \end{Pro}
 \begin{proof}
The proof divided into 4 steps.
(i) Consider the case
$M\in \KL_k^{\Delta}$.
By Theorems \ref{Th:Kac-Wakimoto-vanising}
and Proposition \ref{Pro:strainge-way-of-argument},
we have
\begin{align*}
 \gr^\Gamma \BRS{0}{M}\cong H^0(\gr^\Gamma C(M)).
\end{align*}
Hence the assertion follows  from 
Proposition  \ref{Pro:vanishing-intermideate-1}.
(ii) Consider the case $M$ is finitely generated and  $\Gamma=F$.
There  is  an exact sequence 
\begin{align}
 0\ra N\ra P{\ra} M\ra 0
\label{eq:short-exact}
\end{align}
in ${\KL}_k$
with
$P\in {\KL}_k^{\Delta}$.
Let $\Gamma^\bullet N$ be the compatible filtration of $N$
induced by $F^\bullet M$,
so that we have
 the exact sequence
\begin{align*}
 0\ra C(N)/\Gamma^p C(N)\ra C(P)/F^p C(P)\ra C(M)/ F^p C( M)\ra 0.
\end{align*}
By Proposition   \ref{Pro:vanishing-intermideate-1},
this induces an exact sequence
\begin{align*}
 0\ra H^0(C(N)/\Gamma^p C(N))\ra H^0(C(P)/F^p C(P))\ra H^0(C(M)/ F^p C( M))\ra 0.
\end{align*}
Consider the commutative diagram
\begin{align*}
\minCDarrowwidth1pc
 \begin{CD}
\BRS{0}{P}
@>>> H^0(C(P)/F^p C(P))
\\
 @VVV @VVV
\\
\BRS{0}{M}
@>>>H^0(C(M)/F^p C(M)).
 \end{CD}&
\end{align*}
Since vertical arrows and the upper horizontal arrow are  surjective,
the lower horizontal arrow is surjective as well.
(iii)
Let $\Gamma^{\bullet}M$ be an arbitrary compatible filtration 
of a finitely generated object $M$ of $\KL_k$.
Set $C=C(M)$,
and 
let $F^{\bullet}(C/\Gamma^p C)$
 be the filtration
of 
$C/\Gamma^p C$ induced by  the Li filtration $F^{\bullet}C$.
Since $F^{p}C\subset \Gamma^p C$,
this filtration is regular,
and
there is a converging spectral sequence
\begin{align}
E_r\Rightarrow H^{\bullet}(C/\Gamma^p C)
\label{eq;s.s.08-24}
\end{align}
whose $E_1$-term is $H^{\bullet}(\gr^F (C/\Gamma^p C))$.
Observe that 
\begin{align*}
 \gr^F (C/\Gamma^p C)=(\gr^F C)/\Gamma^p (\gr^F C),
\end{align*}
where $\Gamma^{\bullet}(\gr^F C)$ is the filtration of
$\gr^F C$ induced by  $\Gamma^{\bullet}C$.
The associated graded space
$\gr^{\Gamma}(\gr^F C)$ is naturally a vertex Poisson $\C[\fing^*_{\infty}]$-module
which belongs to $\Cat$.
Hence we have
\begin{align*}
 H^i(\gr^\Gamma(\gr^F C))=0\text{ for }i\ne 0
\end{align*}
by  Theorem \ref{Th:BRST-vanishing-for-assicaited-graded}.
As in the proof of   Proposition   \ref{Pro:vanishing-intermideate-1}
one sees that 
\begin{align}
 H^i((\gr^F C)/\Gamma^p (\gr^F C))=0\text{ for }i\ne 0,
\label{eq:vanishing-gr-gr}
\end{align}
and 
that
 the exact sequence
\begin{align*}
 0\ra \Gamma^p (\gr^F C)\ra 
\gr^F C(M)\ra (\gr^F C)/(\Gamma^p \gr^F C)\ra 0
\end{align*}
induces the exact sequence
\begin{align}
 0\ra H^0(\Gamma^p (\gr^F C))\ra 
H^0(\gr^F C)\ra H^0((\gr^F C)/(\Gamma^p \gr^F C))\ra 0.
\label{eqexact-gr-gr}
\end{align}
By \eqref{eq:vanishing-gr-gr}
the  spectral
  sequence
\eqref{eq;s.s.08-24}
  collapses at $E_1=E_{\infty}$ and we have
\begin{align*}
 \gr^F H^0(C/\Gamma^p C)\cong H^0((\gr^F C)/(\Gamma^p \gr^F C)). 
\end{align*}
Therefore \eqref{eqexact-gr-gr}
implies  that 
the natural map
$\gr^F C\ra \gr^F (C/\Gamma^p C)
$ induces the surjection
$H^0(\gr^F C)\ra H^0(\gr^F (C/\Gamma^pC))$.
This shows the surjectivity of the map 
\begin{align}
\gr^F H^0(C/F^p C)
\ra \gr^F H^0(C/\Gamma^p C)
\label{eq:surjective-08-25}
\end{align}
since 
$\gr^F H^0(C/F^p C)\cong H^0(\gr^F (C/F^p C))$,
$\gr^F H^0(C/\Gamma^p C)
\cong H^0(\gr^F (C/\Gamma^p C))$,
and we have  the commutative diagram
\begin{align*}
\minCDarrowwidth1pc
 \begin{CD}
H^0(\gr^F C)
@= H^0(\gr^F C)
\\
 @VVV @VVV
\\
H^0(\gr^F (C/F^p C))
@>>>H^0(\gr^F (C/\Gamma^p C)).
 \end{CD}&
\end{align*}
Since \eqref{eq:surjective-08-25}
is surjective,
we find that
the natural map
$C/F^p C\ra C/\Gamma^p C$ induces the surjection
\begin{align*}
 H^0(C/F^p C)\twoheadrightarrow  H^0(C/\Gamma^p C).
\end{align*}
As we have
the commutative diagram
\begin{align*}
\minCDarrowwidth1pc
 \begin{CD}
\BRS{0}{M}
@= \BRS{0}{M}
\\
 @VVV @VVV
\\
H^0(C/F^p C)
@>>>H^0( C/\Gamma^p C)
 \end{CD}&
\end{align*}
and
 the left horizontal arrow is surjective as proved in (ii),
the right horizontal arrow must be surjective   As well.
(iv) Let $M$, $\Gamma^{\bullet}M$ be arbitrary.
Take a sequence 
 $M_1\subset M_2\subset \dots $ of
of finitely generated
  submodules of $M$ such that $M=\bigcup_i M_i$.
Let
$\Gamma^\bullet M$ be  the induced filtration of $M_i$.
Then $\Gamma^{\bullet}M_i$ is compatible for all $i$, 
 $\Gamma^p M_1\subset \Gamma^p M_2\subset \dots$,
$\Gamma^p M=\bigcup_i \Gamma^p M_i$,
and $M/\Gamma^p M=\bigcup_i (M_i/\Gamma^p M_i)$.
Since the cohomology functor commutes with the injective limits
we have
$\BRS{0}{M}=\lim\limits_{\longrightarrow \atop i}\BRS{0}{M_i}$,
$H^0(C(M)/\Gamma^p C(M))=\lim\limits_{\longrightarrow \atop
  i}H^0(C(M_i)/\Gamma^p C(M_i))$.
Therefore  the assertion follows from (iii).
   \end{proof}

\begin{Pro}\label{Pro;i=0,1}
 Let $M$ be an object of $\KL_k$,
$\Gamma^{\bullet}M$ a compatible filtration.
We have $H^i(\Gamma^p C(M))=0$ for $i\geq 1$,
 $p\geq 1$.
\end{Pro}
\begin{proof}
The assertion follows from
Propositions \ref{Pro:vanishing-intermideate-1} 
and \ref{Pro:surjectivity-revised}
by
considering
the
long exact sequence associated with the short exact sequence
\begin{align}
 0\ra \Gamma^pC(M)\ra C(M)\ra C(M)/\Gamma^p C(M)\ra 0.
\label{eq:short-exact-C-Gamma}
\end{align}
\end{proof}

 \begin{Pro}\label{Pro:important-surjectivity}
 Let $M$ be an object of $\KL_k$,
$\Gamma^{\bullet}M$ a compatible filtration,
$d\in \C$.
Assume that
$H^0(\gr^\Gamma C(M))_d$ is finite-dimensional.
Then
the natural map
$C(M)\ra C(M)/\Gamma^p C(M)$ induces the isomorphism
\begin{align*}
\BRS{0}{M}_d\isomap H^0(C(M)/\Gamma^p C(M))_d
\end{align*}
for a sufficiently large $p$.
 \end{Pro}
  \begin{proof}
In view of   Propositions \ref{Pro:vanishing-intermideate-1} and
 \ref{Pro:surjectivity-revised},
the statement is equivalent to that 
\begin{align*}
\dim \BRS{0}{M}_d=\dim H^0(\gr^\Gamma C(M))_d.
\end{align*}
Since 
$\dim H^0(\gr^\Gamma C(M))_d$ is independent of the choice of a
   compatible filtration $\Gamma$
by Proposition
   \ref{Pro:finite-dimensionality-of-associated-graded}, 
we may assume that $\Gamma=F$.
Also, since the cohomology functor commutes with the injective limits
we may assume that $M$ is finitely generated,
so that 
there is  an exact sequence 
\begin{align}
 0\ra N\ra P{\ra} M\ra 0
\label{eq:short-exact}
\end{align}
in ${\KL}_k$
with
$P\in {\KL}_k^{\Delta}$.
We have  the exact sequence
\begin{align*}
 0\ra \Gamma^p C(N)\ra F^p C(P)\ra F^p C(M)\ra 0.
\end{align*}
where $\Gamma^\bullet C(M)$ is the compatible filtration of $C(N)$
induced by $F^\bullet C(M)$.
By Proposition \ref{Pro;i=0,1},
   this induces the exact sequence
\begin{align*}
  H^0(F^p C(P))_d\ra H^0(F^p C(M))_d\ra 0
\end{align*}
As $H^0(F^p C(P))_d=0$ for $p\gg 0$ by 
Theorem 
\ref{Th:Kac-Wakimoto-vanising} and Proposition
   \ref{Pro:strainge-way-of-argument},
we get that 
$H^0(F^p C(M))_d=0$.
Therefore 
Proposition  \ref{Pro:surjectivity-revised}
and
the long exact sequence
associated with the exact sequence
$0\ra F^p C(M)\ra C(M)\ra C(M)/F^p C(M)\ra 0$
gives the isomorphism
$\BRS{0}{M}_d\isomap H^0(C(M)/F^p C(M))_d$ for $p\gg 0$
as required.
  \end{proof}

\begin{Th}
\label{Th:vanishing-new}
 We have $\BRS{i}{M}=0$ for $i\ne 0$
with any object $M$ in $\KL_k$.
In particular   the functor 
$\KL_k\ra \Wg{k}\Mod$,
$M\mapsto \BRS{0}{M}$, is exact.

\end{Th}

\begin{proof}
By Theorem \ref{Th:right-exactness}
it remains to show that
$\BRS{i}{M}=0$ for all $i\leq  -1$.
We proceed  by induction on $|i|$.
Since the cohomology functor commutes with the injective limits
we may assume that $M$ is finitely generated.
Thus, there are  an object $P\in \KL_k^{\Delta}$
and an exact sequence
\begin{align}
 0\ra N\ra P\ra M\ra 0
\label{eq:short-exact2}
\end{align}
in $\KL_k$.
Let $\Gamma^\bullet P$ be a  compatible filtration of $P$,
and let $\Gamma^\bullet M$ and
$\Gamma^\bullet N$ be induced filtration
as in the proof of Proposition \ref{Pro;i=0,1}.
We have the short exact sequence
$0\ra \gr^{\Gamma} N\ra \gr^{\Gamma}P\ra \gr^{\Gamma}M\ra 0$,
that induces an exact sequence
$0\ra H^0(\gr^{\Gamma}C(N))
\ra H^0(\gr^{\Gamma}C(P))
\ra H^{0}(\gr^{\Gamma}C(M))\ra 0$ by Theorem \ref{Th:BRST-vanishing-for-assicaited-graded}.
Note that  
$H^0(\gr^{\Gamma}C(N))_d$,
$H^0(\gr^{\Gamma}C(P))_d$
and $H^0(\gr^{\Gamma}C(M))_d$
 are  finite-dimensional for all $d\in \C$
by Proposition
 \ref{Pro:finite-dimensionality-of-associated-graded}.

We have the commutative diagram
\begin{align*}
\minCDarrowwidth1pc
 \begin{CD}
  0@>>> H_f^{\semiinf-1}(M)_d@>>> \BRS{0}{N}_d@>>>
\BRS{0}{P}_d@>>> \BRS{0}{M}_d@>>> 0\\
@. @. @VVV @VVV@VVV
\\
@. 0@>>>H^0(C_1/\Gamma^p C_1)_d
@>>>H^0(C_2/\Gamma^p C_2)_d
@>>>H^0(C_3/\Gamma^p C_3)_d
@>>> 0,
 \end{CD}&
\end{align*}
where 
$C_1=C(N)$,
$C_2=C(P)$,
$C_3=C(M)$.
But the vertical arrows are all isomorphisms by 
 Proposition \ref{Pro:important-surjectivity}
for a sufficiently large $p$.
Therefore we get that
$H_f^{\semiinf-1}(M)=0$.

Next suppose that we have shown that
$H_f^{\semiinf-i}(M)=0$
for all
objects $M\in \KL_k$.
Since $H_f^{\semiinf-i-1}(P)=0$ by 
Theorem \ref{Th:Kac-Wakimoto-vanising},
the long exact sequence associated with 
\eqref{eq:short-exact}
gives that
$H_f^{\semiinf-i-1}(M)=0$ as desired.
\end{proof}
\begin{Rem}
 In the case that
$f$ is a principal nilpotent element 
Theorem 
\ref{Th:BRST-vanishing-for-assicaited-graded}
has been obtained by Frenkel and Gaitsgory \cite{FreGai07}
from the vanishing result \cite{Ara07} of the BRST cohomology
associated with the ``$-$''-reduction functor,
which is a certain modified version \cite{FKW92} of 
the cohomology functor $\BRS{\bullet}{?}$.
\end{Rem}

 Let $M$ be an object of $\KL_k$,
$\Gamma^\bullet M$ be a  compatible  filtration of $M$.
By Proposition \ref{Pro:strainge-way-of-argument}
and Theorem \ref{Th:vanishing-new},
the filtration $F^p C(M)_d$ is regular, and hence
there is a converging spectral sequence
$E_r\Rightarrow \BRS{\bullet}{M}$ whose $E_1$-term is 
$H^{\bullet}(\gr^\Gamma C(M))$.
By Theorem \ref{Th:BRST-vanishing-for-assicaited-graded}
this  spectral sequence
collapses at $E_1=E_{\infty}$ 
and
we have
the isomorphism
\begin{align}
 \gr^\Gamma \BRS{0}{M}
\cong  H^0(\gr^\Gamma C(M))\cong 
		       (\gr ^\Gamma M/I_{\infty}\gr^\Gamma
 M)^{\finn[t]}.
\label{eq:associated-graded-as-vector-spaces}
\end{align}

\begin{Th}
\label{Th:strong-vanishing}
$ $

\begin{enumerate}
 \item The Li filtration 
of $\Vg{k}$ induces the Li filtration 
of the vertex algebra $\Wg{k}$,
that is,
$F^p \Wg{k}=\im (H^0(F^p C(\Vg{k}))\ra \Wg{k})$.
Moreover,
 we have  
\begin{align*}
\gr^F \Wg{k}\cong \C[\Sl_{\infty}]
\end{align*}
as 
vertex Poisson algebras.
 \item   
More generally,
for 
a finitely generated  object $M$ of $\KL_k$,
the Li filtration of $M$ induces the Li filtration of $\BRS{0}{M}$.
Hence
\begin{align*}
 \gr^F \BRS{p}{M}
&\cong  H^p(\gr^F C(M))\\
&\cong 
\begin{cases}
 		       (\gr ^F M/I_{\infty}\gr^F M)^{\finn[t]}&\text{for
 }p=0,\\
0&\text{for }p\ne 0,
\end{cases}
\end{align*}
as modules over the vertex Poisson algebras $\C[\Sl_{\infty}]$.
\end{enumerate}
\end{Th}
\begin{proof}
Let $M$ be an object of $\KL_k$.
We  temporary denote
 by $\Gamma^\bullet \BRS{0}{M}$
the compatible
filtration of $\BRS{0}{M}$ induced by the Li filtration of $M$.
We have
\begin{align*}
F^p \BRS{0}{M}\subset \Gamma^p \BRS{0}{M}
\end{align*}
since the Li filtration is the finest compatible filtration
of $\BRS{0}{M}$.
It follows that
we have the natural map
\begin{align}
 \gr^F \BRS{0}{M}\ra \gr^\Gamma \BRS{0}{M}
\cong  (\gr ^\Gamma M/I_{\infty}\gr^\Gamma
 M)^{\finn[t]}
\label{eq:map-to-be-iso}
\end{align}
by  \eqref{eq:associated-graded-as-vector-spaces}.

(i)
Let $M=\Vg{k}$. 
Observe that the map (\ref{eq:map-to-be-iso}) is an homomorphism 
\begin{align}
 \gr^F \Wg{k}\ra \gr^\Gamma \Wg{k}\cong \C[\mathcal{S}_{\infty}]
\label{eq:F-to-K}
\end{align}
of vertex Poisson algebras.
In particular,
this restricts to the
surjective homomorphism
\begin{align*}
 \Wg{k}/F^1\Wg{k}\twoheadrightarrow \Wg{k}/\Gamma^1\Wg{k}\cong \C[\mc{S}]
\end{align*}
of Poisson algebras.
Since
$\C[\Sl_{\infty}]$ is generated by 
$\C[\Sl]$ as differential algebras,
the map  (\ref{eq:F-to-K})
is  surjective.
As
$\dim \gr^F\Wg{k}_d=\dim \gr^\Gamma \Wg{k}_d
=\dim \Wg{k}_d$ 
for each $d$,
the map (\ref{eq:F-to-K}) must be injective as well.
(ii)
The map \eqref{eq:F-to-K}
restricts to the surjection
\begin{align*}
\BRS{0}{M}/F^1 \BRS{0}{M}\twoheadrightarrow
 \BRS{0}{M}/\Gamma^1\BRS{0}{M}\cong H^0(C(M)/F^1 C(M)),
\end{align*}
which is 
a homomorphism
of $\C[\mc{S}]$-modules  by (i).
Since 
$H^0(\gr^ F C(M)))$ is generated by 
$ H^0(C(M)/F^1 C(M))$
by Proposition \ref{Pro:generating-subspace},
\eqref{eq:map-to-be-iso} must be surjective,
and hence is an isomorphism.
\end{proof}

\begin{Co}
\label{Co:Zhu's-Poisson-algebra-of-W}
\begin{enumerate}
 \item (\cite{De-Kac06})
We have
\begin{align*}
\Ring{\Wg{k}}\cong \C[\Sl]
\end{align*}
 as Poisson algebras.
\item
For an finitely generated object $M$ of $\KL_k$
we have
\begin{align*}
\BRS{0}{M}/F^1 \BRS{0}{M}\cong (\bar M/ I\bar M)^{\finn}
\end{align*}
as Poisson modules over $\C[\Sl]$.
In particular
$\gr^F \BRS{0}{M}$ is
 finitely generated
over 
$\C[\Sl_{\infty}]$.
\end{enumerate}
\end{Co}

 \begin{Rem}
 The fact that
$\gr^F \BRS{0}{M}\cong H^0(\gr^F C(M))$
is true for any
object $M$ of $\KL_k$.
To see this,
  take an increasing sequence
$M_1\subset M_2\subset \dots $  of finitely generated submodule
  of $M$ such that
$M=\bigcup_{i=1}^{\infty}M_i$.
Then
$F^nM_1\subset F^n M_2\subset \dots $
and $F^n M_i=\bigcup_{i}F^n M_i$.
Hence
\begin{align*}
 H^{0}(F^n C(M)/F^{n+1}C(M))=
H^{0}(\lim\limits_{\longrightarrow \atop i}
F^nC(M_i)/F^{n+1}C(M_i))
\\=\lim\limits_{\longrightarrow\atop i} 
H^{i}(
F^nC(M_i)/F^{n+1}C(M_i))
=\lim\limits_{\longrightarrow
\atop i}  F^n \BRS{0}{M_i}/F^{n+1}\BRS{0}{M_i}\\
=F^n\BRS{0}{M}/F^{n+1}\BRS{0}{M}
\end{align*}
since cohomology functor commutes with the injective limits.

 \end{Rem}
Recall that
$V_k(\fing)$ (resp.\ $\W_k(\fing,f)$)
denotes the simple quotient of $\Vg{k}$ (resp.\ $\W^k(\fing,f)$).
By Theorem \ref{Th:vanishing-new},
the projection
$\Vg{k}\ra V_k(\fing)$ induces the
surjective 
homomorphism
\begin{align*}
\Wg{k}=\BRS{0}{\Vg{k}}\twoheadrightarrow \BRS{0}{V_k(\fing)}
\end{align*}
of vertex algebras.
Hence we have the following assertion.
\begin{Pro}
Suppose that
$\BRS{0}{V_k(\fing)}$
 is nonzero.
Then it is a 
quotient vertex algebra 
of
$\Wg{k}$,
and hence, 
 has $\W_k(\fing,f)$ as its unique graded simple quotient.
The vertex algebra $\BRS{0}{V_k(\fing)}$ is
$C_2$-cofinite
if and only if it is $C_2$-cofinite as a module over $\Wg{k}$.
If this is the case  the simple $W$-algebra  
$\W_k(\fing,f)$ is  $C_2$-cofinite as well.
\end{Pro}

\subsection{BRST Reduction of associated varieties}
We shall identify 
$\Ring{\Wg{k}}$ 
with
and $\C[\Sl]$
through  
 Corollary \ref{Co:Zhu's-Poisson-algebra-of-W} (i).
Thus,
for a finitely generated object $M$ of $\KL_k$,
 $X_{\BRS{0}{M}}$ is
a  $\C^*$-invariant Poisson subvariety  of 
$\Sl$.
The following assertion follows immediately from Corollary \ref{Co:Zhu's-Poisson-algebra-of-W} (ii).
\begin{Th}\label{Th:BRST-reduction-of-varieties}
  Let $M$ be a finitely generated object of $\KL_k$.
 The associated variety 
$\Ass{\BRS{0}{M}}$ is isomorphic to the 
scheme theoretic intersection $\Ass{M}\cap \Sl$.
\end{Th}

\begin{Pro}\label{Pro:condition-to-be-nonzero}
Let $M$ be a finitely generated object of $\KL_k$.
Then
$\BRS{0}{M}\ne 0$
if and only if 
$\Ass{M}$ contains the closure $\overline{\Ad G\cdot \chi}$
of the coadjoint orbit $\Ad G\cdot \chi$.
\end{Pro}
\begin{proof}
Because   $\Ass{\BRS{0}{M}}$
is stable under the $\C^*$-action,
it follows that
$\BRS{0}{M}\ne 0$ if and only if
$\Ass{M}$ contains the unique $\C^*$-fixed point $\chi$
of $\Sl$.
The assertion follows because
$\Ass{M}$ is $G$-invariant and closed.
\end{proof}
The following assertion follows immediately
from \eqref{eq:C_2-cofintie-criteri},
Theorem \ref{Th:BRST-reduction-of-varieties}
and the transversality of $\mc{S}$ to $G$-orbits.
 \begin{Th}\label{Th:C2-cofinitness-if-variety-is-the-orbit}
Let $M$ be a finitely generated object of $\KL_k$.
Suppose that
$X_{M}\cong \overline{\Ad G.f}$
as subvarieties  of $\fing^*$.
Then $\BRS{0}{M}$ is (non-zero and) $C_2$-cofinite.
 \end{Th}

  \subsection{The BRST cohomology functor kills integrable representations}
\label{subsection:minimal-nilpotent}
Let $f_{\min}$ be a minimal nilpotent element of $\fing$,
$\chi_{\min}=\nu(f_{\min})$,
$\mb{O}_{min}= \Ad G.\chi_{\min}$.
Then $\mb{O}_{min}\subset \Nil$ is the unique nonzero nilpotent orbit
of minimal dimension (see \cite{ColMcG93}).

The following assertion  is a special case of
  \cite[Main Theorem]{Ara05}.
\begin{Th}
\label{Th:duke}
Let
 $\lam\in \affP^+_k$.
Then
 $H_{f_{\min}}^{\semiinf+0}(\Irr{\lam})$ is non-zero  
if and only
$\Irr{\lam}$ is not an integrable representation of $\affg$.
\end{Th}

\begin{Pro}\label{Pro:when-contains-minimal-orbit}
Let $\lam\in \affP^+_k$.
If $\lam$ is integral dominant then
$X_{\Irr{\lam}}=\{0\}$.
Otherwise 
 $\overline{\mb{O}}_{\min}\subset \Ass{\Irr{\lam}}$.
In particular
$\Irr{\lam}$ is an integrable representation of
$\affg$ if and only if 
$X_{\Irr{\lam}}=\{0\}$.
\end{Pro}
\begin{proof}
It is well-known that
$X_{\Irr{\lam}}=\{0\}$
for an integrable representation $\Irr{\lam}$ 
 (see  e.g.\ \cite[Proposition 3.5.1]{Ara12} for a proof).
If $\Irr{\lam}$  is not integrable
 $\BRS{0}{\Irr{\lam}}\ne 0$
by Theorem \ref{Th:duke},
and hence,
 $\overline{\mathbb{O}}_{\min}\subset X_{\Irr{\lam}}$
by
Proposition  \ref{Pro:condition-to-be-nonzero}.
\end{proof}

The following assertion follows 
immediately from 
Proposition  \ref{Pro:condition-to-be-nonzero}
and Proposition
\ref{Pro:when-contains-minimal-orbit}.
\begin{Co}
\label{Co:kills-integrable}
 Let $f$ be a non-zero nilpotent element of $\fing$,
$\Irr{\lam}$  an integrable representation of $\affg$.
Then  $\BRS{\bullet}{\Irr{\lam}}=0$.
\end{Co}
 \begin{Rem}
Corollary \ref{Co:kills-integrable}
also follows from Theorem 
\ref{Th:vanishing-new} and 
\cite[Theorem 3.2]{KacRoaWak03}.
 \end{Rem}
\section{Associated varieties of Kac-Wakimoto
admissible representations
and $C_2$-cofiniteness of  $W$-algebras
}
In this section 
we prove the Feigin-Frenkel conjecture
on the singular support of $G$-integrable
{\em admissible} representations 
and
determine their  associated varieties.
As a consequence 
we 
 prove the $C_2$-cofiniteness for a large number of
$W$-algebras including all  exceptional
$W$-algebras
(Theorem \ref{Th:C_2-cofiniteness-of-modules-over-W-algebras}
and Theorem \ref{Th:exceptionals-are-C2}).
In this section $\fing^*$ is often identified with $\fing$
and $X_M$ is identified with its underlying topological space.

\subsection{Kac-Wakimoto admissible representations}
\label{section:admissible-representations}
A subset 
$\Delta'$ of $\wh \Delta^{\on{re}}$
is called a {\em subroot system}
if $s_{\alpha}(\beta)\in \Delta'$
for any $\alpha,\beta\in \Delta'$.
For a subroot system $\Delta'$,
$\Delta_+'=\Delta'\cap \Delta^{\on{re}}_+$
is a set of positive root and
 $\Pi'=\{\alpha\in \Delta'_+;
s_{\alpha}(\Delta'_+\backslash\{\alpha\})
\subset \Delta'_+\}$
is the set of simple roots
(\cite{MooPia95,KasTan98}).

For $\lam\in \dual{\affh}$,
let 
$\wh \Delta(\lam)$ 
be the associated  {\em integral root system} defined by
\begin{align*}
\wh{\Delta}(\lam)=\{\alpha\in \wh{\Delta}^{\on{re}};
\bra \lam+\wh\rho,\alpha\che\ket \in \Z\}.
\end{align*}
Then $\wh \Delta(\lam)$ 
is a subroot system of $\wh \Delta^{\on{re}}$.
Let $\wh \Delta(\lam)_+$
and $\wh \Pi(\lam)$
be the sets of positive roots and simple roots
of $\wh\Delta(\lam)$ respectively.
The  subgroup $\affW(\lam)
=\bra s_{\alpha};
\alpha\in \wh\Delta(\lam)\ket$ 
of $\affW$
is  called the {\em integral Weyl group} of $\lam$.

A weight $\lam\in \dual{\affh}$
is called {\em admissible}
\cite{KacWak89}
if 
\begin{enumerate}
 \item $\lam$ is regular dominant,
that is, $\bra  \lam+\wh\rho,\alpha\che\ket \not\in \{0,-1,-2,\dots,\}$
for any $\alpha\in \Delta^{\on{re}}_+$,
\item $\Q\wh{  \Delta}(\lam)=\Q \Delta^{\on{re}}$.
\end{enumerate}
Note that an integral dominant weight of $\affg$ is admissible.

For $\lam\in \dual{\affh}$,
the irreducible highest weight representation
$\Irr{\lam}$ 
is called an {\em admissible
representation}
if 
$\lam$ is admissible.
The condition (i) implies \cite{KacWak88}  that
the formal character $\ch \Irr{\lam}$
 of $\Irr{\lam}$ is given by
\begin{align}
 \ch \Irr{\lam}=\sum_{w\in \affW(\lam)}(-1)^{\ell_{\lam}(w)}
e^{w\circ \lam}
\prod_{\alpha\in \wh \Delta_+}
(1-e^{-\alpha})^{-\dim \affg_{\alpha}},
\label{eq:ch-of-admissible-modules}
\end{align}
where 
$\ell_{\lam}:\affW(\lam)\ra \Z_{\geq 0}$
is the length function
and $\affg_{\alpha}$ is the root space of root $\alpha$.


\subsection{$G$-integrable admissible representations}
\label{subsection:$G$-integrable admissible representations}
Set
\begin{align*}
 \Prp^k=\{\lam\in \affP_+^k; \lam\text{ is admissible} \},
\quad \Prp=\bigcup_{k\in \C}\Prp^k,
\end{align*}
where $\affP_+^k$ is defined in \eqref{eq:p+}.

A complex number $k$
is called {\em admissible} for $\affg$
if the weight $k\Lam_0$ is  admissible.
The following assertion is easy to see (cf.\ Proposition \ref{Pro:description-of-admissble-weights}).
\begin{Lem}\label{Lem:admissibe-module-in-the-category-KL}
The set  $\Prp^k$ is non-empty if and only if 
 $k$ is admissible.
\end{Lem}

The admissible numbers 
are classified in \cite{KacWak89,KacWak08}.
Set
\begin{align*}
&\A= \A_{\fing}= \bigsqcup_{q\in \N\atop
 (q,r\che)=1}\A[q],\quad
{}^L\A={}^L\A_{\fing}= 
\bigsqcup_{q\in \N}
{}^L\A[q],
\end{align*}
where
\begin{align*}
& \A[q]\teigi \{-h\che+\frac{p}{q};p,q\in \N,\ (p,q)=1,\
p\geq h\che\},\\
&{}^L\A[q]\teigi \{-h\che+\frac{p}{r\che q};p,q\in \N,\
 (p,q)=1,\
(p,r\che)=1,\
p\geq h_{\fing}\}.
\end{align*}
Here
$r\che$ is the lacety of $\fing$
and
$h_{\fing}$
 is the Coxeter number of $\fing$.
\begin{table}
\begin{center}
\begin{tabular}{|l|l|l|l|l|}
\hline 
$\fing$  & $h_{\fing}$ & $h^{\vee}$&${}^L h\che$& $r\che$ 
\\ \hline
$A_{l}$ & $l+1$ & $l+1$ & $l+1$ & $1$
\\
$B_{l}$ & $2l$ & $2l-1$ &$l+1$&$2$
 \\
$C_{l}$ &  $2l$ & $l+1$ &$2l-1$ &$2$
 \\
$D_{l}$ & $2l-2$ & $2l-2$ & $2l-2$ &$1$
\\
$E_{6}$ & 12 & 12 & $12$ &$1$ 
 \\
$E_{7}$  & 18 & 18 & 18 &1 
\\
$E_{8}$  & 30 & 30 &30&1
 \\
$F_{4}$ & 12 & 9 &9&2 
\\
$G_{2}$  & 6 & 4&4&3 
\\
\hline
\end{tabular}
\end{center}
\caption{}
\label{Coxeter}
\end{table}

\begin{Pro}[\cite{KacWak89,KacWak08}]
\label{Pro:description-of-admissble-weights}
 \begin{enumerate}
  \item 
A complex number $k$ is admissible 
if and only if
$k\in \A \cup {}^L\A $.
\item 
Let $k\in \A [q]$ with $q\in \N$,
$(q,r\che)=1$.
Then,
for  $\lam \in \Prp^k$,
$
\wh\Delta(\lam)=\wh \Delta(k\Lam_0)=\{\alpha+nq\delta;
\alpha\in \Delta, n\in \Z\}
$,
$\wh \Pi(\lam)=\{-\theta+q\delta,
\alpha_1,\dots,\alpha_{l}\}$. 

\item 
 Let $k\in {}^L\A[q]$ with $q\in \N$.
Then,
for  $\lam \in \Prp^k$,
we have
$
\wh \Delta(\lam)
=\wh \Delta(k\Lam_0)
=
 \{\alpha+r\che nq\delta;
\alpha\in \Delta_{long}, n\in \Z\}
\sqcup \{\alpha+nq\delta;\alpha\in \Delta_{short},n\in\Z\}$,
$\wh \Pi(\lam)=\{-\theta_s+q\delta,\alpha_1,\dots,\alpha_{l }\}$.
Here  $\Delta_{long}$ and $\Delta_{short}$
are the sets of long roots and  short roots of $\fing$,
respectively.

 \end{enumerate}
\end{Pro}

\begin{Rem}
\begin{enumerate}
\item We have $\A[1]=\Z_{\geq 0}$
and $Adm^k$ for $k\in \Z_{\geq 0}$ is 
the set of integrable dominant weights of level $k$.
 \item
Let $k\in \A$, $\lam\in\Prp^k$.
Then $\wh\Delta(\lam)\cong \wh \Delta^{\on{re}}$
as root systems,
that is, 
 $\lam$ is a {\em principal admissible weight} \cite{KacWak89}.
\item
Let $k\in {}^L\A$, $\lam\in\Prp^k$.
Then $\wh \Delta(\lam)\che$
is isomorphic to the root system
of $\widehat{({}^L\fing)}$.
\item
The assignment
\begin{align*}
k\mapsto {}^L k:=\frac{1}{r\che(k+h\che )}-
{}^L h\che
\end{align*}gives
 bijections
$\A_\fing\isomap {}^L\A_{{^L}\fing}$
and
${}^L\A_\fing\isomap \A_{{^L}\fing}$,
where ${}^L \fing$ is the Langlands dual Lie algebra
of $\fing$ and
   ${}^Lh\che$ is the dual Coxeter number of 
${}^L \fing$.
\end{enumerate}
\end{Rem}

\subsection{The Feigin-Frenkel conjecture}
The following assertion was conjectured by 
Feigin and Frenkel (unpublished).
\begin{Th}
\label{Th:Conj:Feigin-Frenkel}
Let $\lam\in \Prp$.
Then $SS(\Irr{\lam})\subset \Nil_{\infty}$,
or equivalently,
$X_{\Irr{\lam}}\subset \Nil$.
 \end{Th}

\begin{Rem}
If  $\Irr{\lam}$  is an integrable representation of $\affg$
then
 $X_{\Irr{\lam}}
=\{0\}$ (see Proposition \ref{Pro:when-contains-minimal-orbit}),
and
we have obviously that
$X_{\Irr{\lam}}\subset \Nil$.
\end{Rem}
For 
$\fing=\mf{sl}_2$,
Theorem \ref{Th:Conj:Feigin-Frenkel}
has been proved by Feigin and Malikov \cite{FeiMal97}.
Below we reproduce their proof for completeness in the case that $\lam=k\Lam_0$.

\begin{Th}[Feigin and Malikov \cite{FeiMal97}]
\label{th:Feing-Malikov}
Let $\fing=\mf{sl}_2$,
and let $k$ be an admissible number for
 $\affg=\widehat{\mathfrak{sl}}_2$.
Then $X_{\Irr{\lam}}\subset \Nil$.
\end{Th}
\begin{proof}
Since
$V_k(\fing)$ is a quotient of $\Vg{k}$,
$R_{V_k(\fing)}$ is a quotient of
$R_{\Vg{k}}=\C[\fing^*]$ by 
 the image $I$ of the maximal ideal
of $\Vg{k}$
in $R_{\Vg{k}}$.
Let $k\in \A[q]$
and
write $k+2=p/q$ with $p\in \N$.
It is known \cite{KacWak88} that
the maximal ideal 
of $\Vg{k}$ is generated by a
singular vector,
 say $v$,
of weight $s_{-\theta+q\delta}\circ k\Lam_0$.
The projection formula \cite{MalFeuFuk86}
of
$v$ implies that
its image $\bar v$ in $I\subset R_{\Vg{k}}=\C[\fing^*]$ is non-zero.
Since $\bar v$
 is a singular vector 
of the $\ad \fing$-module $\C[\fing^*]$
of weight $2(p-1)\alpha$  and degree $(p-1)q$,
Kostant's Separation Theorem implies that
$\bar v$ must  
be
 equals to
$\Omega^{\frac{(p-1)(q-1)}{2}} e^{p-1}$ up to
 multiplication by a nonzero constant,
where $\Omega=\frac{1}{2}h^2+2 ef$.
It follows that  $I$ contains some power of $\Omega$,
and hence,
$X_{V_k(\fing)}\subset \Nil$.
\end{proof}

 \begin{Rem}
 The condition that
$V_k(\fing)\subset \Nil$ is equivalent to that
$V_k(\fing)$ is $C_2$-cofinite as a $\Q$-graded vertex algebra
in the sense of \cite{DonLiMas97}.
\end{Rem}

 \subsection{Proof of Theorem \ref{Th:Conj:Feigin-Frenkel}}

\begin{Lem}\label{Lem:enught-to-consider-the-vacuum}
Suppose that
the Feigin-Frenkel conjecture holds for $V_k(\fing)$,
that is,
$X_{V_k(\fing)}\subset \Nil$.
Then  
for a finitely strongly generated
$V_k(\fing)$-module $M$ 
we have
 $X_M\subset \Nil$.
\end{Lem}\begin{proof}
	  Clear since
$\bar M=M/F^1 M$  is a finitely generated
$R_{V_k(\fing)}$-module.
	 \end{proof}
\begin{Th}[\cite{MalFre99}, see also \cite{FreMal97}]\label{Th:Frenkel-Malikov}
 Let $k$ be an admissible number,
$\lam\in Adm_+^k$.
Then $L(\lam)$ is a $V_k(\fing)$-module.
\end{Th}

\begin{Co}\label{Co:vacuume-case-is-enough}
 Let $k$ be an admissible number.
If $X_{V_k(\fing)}\subset \Nil$
then 
$X_{\Irr{\lam}}\subset \Nil$
for all $\lam\in Adm_+^k$.
\end{Co}

Let $\{e_i,h_i,f_i; i=1,\dots,l \}$
be Chevalley generators of $\fing$.
Denote by $\mf{sl}_2^{(i)}$
the copy of
$\mf{sl}_2$
in  $\fing$ spanned
 by $e_i$, $h_i$
and $f_i$.
Let $\finp^{(i)}$ be the parabolic subalgebra 
of $\fing$ spanned by $\finb$ and $f_i$,
 $\fink^{(i)}$ its Levi subalgebra,
and $\finr^{(i)}$ its nilradical.
Thus,
$\fink^{(i)}$ is equal to the direct sum of $\mf{sl}_2^{(i)}$
and the orthogonal complement $\finh_i^{\bot}$ of $\C h_i$
in $\finh$.

In order to reduce
the proof of Theorem \ref{Th:Conj:Feigin-Frenkel}
to the $\mf{sl}_2$-case 
we 
 shall consider
the {\em semi-infinite restriction}
of $V_k(\fing)$,
which is by definition 
 the semi-infinite
cohomology space
$H^{\semiinf+\bullet}(\finr^{(i)}[t,t\inv], V_k(\fing))$.

For $k\ne -h\che$,
define $k_i\in \C$ by
\begin{align*}
 k_i+2=
\frac{2}{(\alpha_i|\alpha_i)}
(k+h\che).
\end{align*}
Note that
if $k$ is an admissible number for $\affg$,
then $k_i$ is an admissible number for $\widehat{\mf{sl}}_2$.
The space  $H^{\semiinf+\bullet}(\finr^{(i)}[t,t\inv], V_k(\fing))$
 is naturally a vertex algebra,
and 
there is a
natural  
vertex algebra homomorphism
\begin{align}
 V^{k_i}(\mf{sl}_2^{(i)})
\ra H^{\semiinf+0}(\finr^{(i)}[t,t\inv],
V_k(\fing)
),
\label{eq:va-homo-sl2}
\end{align}
which sends the vacuum vector
of $ V^{k_i}(\mf{sl}_2^{(i)})$ to the image of the vacuum vector
of $V_k(\fing)
$ in $H^{\semiinf+0}(\finr^{(i)}[t,t\inv],
V_k(\fing)
)$, see e.g. \cite{HosTsu91}.

\begin{Th}[{\cite[Theorem 7.5]{A-BGG}}]\label{Th:reduction-to-sl_2}
Let $k$ be an admissible number,
$i=1,\dots,l$.
The map 
\eqref{eq:va-homo-sl2}
 factors through the embedding
\begin{align*}
V_{k_i}({\mf{sl}}_2^{(i)})
\hookrightarrow
H^{\semiinf+0}(\finr^{(i)}[t,t\inv],
V_k(\fing)
),
\end{align*}where 
$V_{k_i}({\mf{sl}}_2^{(i)})
$ denotes the unique simple quotient of
$V^{k_i}({sl}_2^{(i)})$
(which is isomorphic to an admissible representation of 
$\wh{\mf{sl}_2}^{(i)}$).
\end{Th}

 \begin{proof}[Proof of Theorem \ref{Th:Conj:Feigin-Frenkel}]
By Corollary
\ref{Co:vacuume-case-is-enough}
it is sufficient to prove the assertion  for $\lam=k\Lam_0$.

Let $(C^{\bullet},d)$ denote 
 Feigin's standard complex for calculating
the semi-infinite cohomology
$H^{\semiinf+\bullet}(\finr^{(i)}[t,t\inv], V_k(\fing))$.
The space $C^{\bullet}$ is naturally a vertex algebra.
We have
\begin{align*}
 R_{C^{\bullet}}
=R_{V_k(\fing)}\* \bigwedge\nolimits^{\bullet}_{op} \finr^{(i)}\* 
\bigwedge\nolimits^{\bullet} (\finr^{(i)})^*,
\end{align*}
and
$d$ induces the
differential $\bar d=d_{(0)}$ on 
$R_{C^{\bullet}}$.
We have the natural Poisson algebra homomorphism
\begin{align*}
 R_{H(C^{\bullet})}\ra H(R_{C^{\bullet}})=H(R_{C^{\bullet}},\bar d).
\end{align*}
On the other hand
the embedding
$V_{k_i}(\mf{sl}_2^{(i)})
\hookrightarrow
 H(C^{\bullet})$ in Theorem \ref{Th:reduction-to-sl_2}
induces the
homomorphism 
\begin{align*}
R_{V_{k_i}(\mf{sl}_2^{(i)})}
\ra R_{H(C^{\bullet})}
\end{align*}
 of 
Poisson algebras.
By composing the above two maps we obtain the Poisson algebra homomorphism
\begin{align*}
\Phi: R_{V_{k_i}(\mf{sl}_2^{(i)})}
\ra H(R_{C^{\bullet}}).
\end{align*}

By Theorem \ref{th:Feing-Malikov}
we have
$X_{V_{k_i}(\mf{sl}_2^{(i)})}\subset \Nil$,
or equivalently,
$\Omega_i:=\frac{1}{2}h_i^2+2 e_i f_i$  is nilpotent in
$ R_{V_{k_i}}(\mf{sl}_2^{(i)})$.
Therefore
$\Phi(\Omega_i)$ is nilpotent in $H(R_{C^\bullet})$.

Now by definition we have 
\begin{align*}
&\bar d R_{C^{\bullet}}\subset  (\finr_i\* \C \* \C)R_{C^{\bullet}}
+ (\C \* \finr^{(i)}\* \C)R_{C^{\bullet}}+
 (\C \* \C\*  ( \finr^{(i)})^*)R_{C^{\bullet}},\\
& \Phi(x)\equiv x\* 1\* 1\pmod{(\C \* \finr^{(i)}\* \C)R_{C^{\bullet}}+
 (\C \* \C\* (\finr^{(i)})^*)R_{C^{\bullet}}
+\bar d R_{C^{\bullet}}}\quad\text{for }x\in \mf{sl}_2^{(i)},
      \end{align*}
see
 \cite[2.15]{HosTsu91}.
Hence the
nilpotency 
of $\Phi(\Omega_i)$ implies that
\begin{align*}
 h_i^{N}\equiv 0\pmod{\bigoplus_{\alpha\in \Delta}\fing_{\alpha}R_{V_k(\fing)}}
\quad \text{ in }R_{V_k(\fing)}
\end{align*}
for a sufficiently large $N$.
This gives that
$\lam(h_i)=0$
for
$\lam\in X_{V_k(\fing)}\cap \finh^*$,
 $i=1,\dots, l$,
and hence,
$X_{V_k(\fing)}
\cap \finh^*=\{0\}$.
As 
$X_{V_k(\fing)}$
 is $G$-invariant
and closed,
we obtain
$X_{V_k(\fing)}\subset \Nil$ as required.
\end{proof}

\subsection{A character formula of $\BRS{0}{\Irr{\lam}}$}
We now wish to
determine the
variety
$X_{L(\lam)}$ for $\lam\in Adm^k_+$.
By Theorem \ref{Th:Conj:Feigin-Frenkel},
$X_{\Irr{\lam}}$ is a finite union of nilpotent orbits
of $\fing$.
Hence,
in view of Proposition \ref{Pro:condition-to-be-nonzero},
we have only to know for which
nilpotent element $f$  the cohomology
$\BRS{0}{L(\lam)}$ is nonzero.

We assume that the grading
(\ref{eq:good-grading})
is 
 Dynkin.
Set
\begin{align*}
 \ch_q \BRS{0}{\Irr{\lam}}
=\sum_{d\in \C}
q^d \dim \BRS{0}{\Irr{\lam}}_d.
\end{align*}
By Corollary \ref{Co:finite-dimensionality},
$\ch_q \BRS{0}{\Irr{\lam}}$ is well-defined.

For $\lam\in \Prp^k$,
set
\begin{align*}
\wh\Delta_f(\lam) &=\{\alpha\in \wh{\Delta}(\lam);
\alpha(D+\x) =0\}.
\end{align*}
By Proposition \ref{Pro:description-of-admissble-weights}
we have
\begin{align*}
\wh\Delta_f(\lam) =
\begin{cases}
\{\alpha-\alpha(\x)\delta; \alpha\in \Delta,\
\alpha(\x)\equiv 0 \pmod{q}\}&\text{if $k\in \A[q]$},\\
\begin{array}{l}
\{\alpha-\alpha(\x)\delta;\alpha\in \Delta_{long},\
\alpha(x_0)\equiv 0 \pmod{r\che q}\}\\
\quad \sqcup 
\{\alpha-\alpha(\x)\delta;\alpha\in \Delta_{short},\
\alpha(x_0)\equiv 0 \pmod{q}\}
\end{array}
&\text{if $k\in {}^L\A[q]$.}
\end{cases}
\end{align*}
It follows that $\wh\Delta_f(\lam)  $ is a finite subroot system of 
$\wh{\Delta}^{\on{re}}$ containing 
$\Delta_0$ (recall \eqref{eq:dec-of-roots}).
Let $\wh W_f(\lam)=\bra s_{\alpha};
\alpha\in \wh{\Delta}_f(\lam)\ket
\subset \affW$.
\begin{Pro}\label{Pro:criterion-nonzer-character}
 Let $\lam \in \Prp$.
\begin{enumerate}
 \item Suppose that
$\wh\Delta_f(\lam) \supsetneq \Delta_0$.
Then $\BRS{0}{\Irr{\lam}}= 0$. 
\item 
Suppose that
$\wh\Delta_f(\lam) =
\Delta_0$.
Then
\begin{align*}
& \ch_q \BRS{0}{\Irr{\lam}}
=
 \frac{\sum\limits_{y\in W^0(\lam)}(-1)^{\ell_{\lam}(y)}
\prod\limits_{\alpha\in\Delta_{0,+}}\frac{\bra y(\lam+ \rho),\alpha\che\ket}{\bra  \rho,\alpha\che\ket}
q^{-\bra y\circ \lam,D+\x\ket}}{
\prod\limits_{j\geq 1}(1-q^j)^{l +\#\Delta_{0}}
\prod\limits_{
j\geq 0}(1-q^{\frac{1}{2}+j})^{\# \Delta_{\frac{1}{2}}}},
\end{align*}
where $W^0(\lam)$ is a representative of the coset
$W_0\backslash \affW(\lam)$.
In particular,
$\BRS{0}{\Irr{\lam}}$ is nonzero and
the dimension of the top weight space
of $\BRS{0}{\Irr{\lam}}$
 is given by
\begin{align*}
 \dim  \BRS{0}{\Irr{\lam}}_{-\bra \lam,D+x_0\ket}=\prod_{\alpha\in\Delta_{0,+}}
\frac{\bra \lam+\rho,\alpha\che\ket}{\bra \rho\che,\alpha\che\ket}.
\end{align*}
\end{enumerate}\end{Pro}
\begin{proof}
Let $\ch C^i(\Irr{\lam})$ be the formal character of
the $\affh$-module $C^i(\Irr{\lam})$:
$\ch C^i(\Irr{\lam})=\sum_{\mu}e^{\mu} \dim C^i(\Irr{\lam})^{\mu}$,
where 
$ C^i(\Irr{\lam})^{\mu}$ is the weight space of weight $\mu$.
We have
\begin{align*}
\sum_{i\in \Z} (-1)^iz^i\ch C^i(\Irr{\lam})=\ch \Irr{\lam}\times
\frac{\prod\limits_{\alpha\in \Delta_{> 0}
\atop n\geq 0}(1+z e^{-\alpha+n\delta})\prod\limits_{\alpha\in \Delta_{> 0}
\atop n\geq 1}(1+z\inv e^{\alpha+n\delta})}{
\prod\limits_{\alpha\in \Delta_{1/2}\atop n\geq 1}
(1-e^{-\alpha+n\delta}),}\nonumber\\
= \frac{\sum\limits_{w\in\affW(\lam)}(-1)^{\ell_{\lam}(w)}e^{w\circ \lam}\prod\limits_{\alpha\in \Delta_{> 0}
\atop n\geq 0}(1+z e^{-\alpha+n\delta})\prod\limits_{\alpha\in \Delta_{> 0}
\atop n\geq 1}(1+z\inv e^{\alpha+n\delta})
}{\prod_{i=1}^\infty(1-q^j)^l\prod\limits_{\alpha\in
 \wh{\Delta}_+^{re}}(1-e^{-\alpha})
\prod\limits_{\alpha\in \Delta_{1/2}\atop n\geq 1}
(1-e^{-\alpha+n\delta})}.
\end{align*}
where $\Delta_{>0}=\bigsqcup_{i>0}\Delta_i$,
see  \cite{KacRoaWak03,Ara05}.

Let $\rho_1\che=\frac{1}{2}\sum_{\alpha\in \wh\Delta_f(\lam) }\alpha\che$.
Since
$\BRS{i}{\Irr{\lam}}=0$ for $i\ne 0$
by Theorem \ref{Th:vanishing-new},
 the Euler-Poincar\'{e} principle
gives that
\begin{align*}
& \ch_q \BRS{0}{\Irr{\lam}}\\
&=\lim_{t\rightarrow 0}
\left(\sum_{w\in \affW(\lam)}(-1)^{\ell_\lam(w)}q^{-
\bra w\circ\lam,D+\x+t \rho_1\che\ket}
\prod_{\alpha\in \Delta_{0,+}
}(1-q^{t
\bra \alpha,\rho_1\che \ket })\inv 
\right)\\
& \quad \times 
\prod_{j\geq 1}(1-q^j)^{-l-\#\Delta_{0}}
\prod_{
j\geq 0}(1-q^{\frac{1}{2}+j})^{-\# \Delta_{\frac{1}{2}}}.
\end{align*}

Choose a  representative $\wh W^f(\lam)$ of 
the coset $\wh W_f(\lam) \backslash \affW(\lam)$.
We have
\begin{align*}
 &\sum_{w\in \affW(\lam)}(-1)^{\ell_{\lam}(w)}
q^{-\bra w\circ \lam,D+\x+t \rho_1\che\ket}\\
&= \sum_{u\in \affW_f(\lam)}\sum_{y\in \affW^f(\lam)}(-1)^{\ell_{\lam}(uy)}
q^{-\bra uy\circ \lam,D+\x\ket}q^{-t\bra uy\circ \lam,\rho_1\che\ket}
\\&=
\sum_{y\in \wh W^f(\lam)}(-1)^{\ell_{\lam}(y)}
q^{-\bra y\circ \lam,D+\x\ket}
\sum_{u\in \wh W_f(\lam)}(-1)^{\ell_{\lam}(u)}q^{t\bra uy\circ \lam,\rho\che_1\ket}.
\end{align*}
Let
$\ell_1: \wh W_f(\lam)\ra \Z_{\geq 0}$ be the length function
of the Coxeter group $\affW_f(\lam)$.
Since $(-1)^{\ell_1(s_{\alpha})}=(-1)^{\ell_{\lam}(s_{\alpha})}=-1$ for
 a reflection $s_{\alpha}\in \affW_f(\lam)$,
 $(-1)^{\ell_{\lam}(u)}=(-1)^{\ell_1(u)}$
for $u\in \widehat W_f(\lam)$.
Hence
\begin{align*}
&
\sum_{u\in \wh W_f(\lam)}(-1)^{\ell_{\lam}(u)}q^{t\bra uy\circ
 \lam,\rho\che_1\ket}=
\sum_{u\in \wh W_f(\lam)}(-1)^{\ell_{1}(u)}q^{t\bra uy\circ
 \lam,\rho\che_1\ket}\\
&=
q^{\bra y\circ\lam,t\rho_1\che\ket}
\prod_{\alpha\in\wh \Delta_f(\lam)\cap \wh \Delta_+}(1-q^{t \bra y( \lam+
\rho),\alpha\che\ket}).
\end{align*}
It follows that
\begin{align*}
 &\lim_{t\rightarrow 0}
\frac{\sum\limits_{w\in \affW(\lam)}(-1)^{\ell_\lam(w)}q^{-
\bra w\circ\lam,D+\x+t \rho_1\che\ket}}
{\prod\limits_{\alpha\in \Delta_{0,+}
}(1-q^{t
\bra \alpha,\rho_1\che \ket })}
=\lim_{t\rightarrow 0}\frac{q^{\bra y\circ\lam,t\rho_1\che\ket}
\prod\limits_{\alpha\in\wh \Delta_f(\lam)\cap \wh \Delta_+}(1-q^{t \bra y( \lam+
\rho),\alpha\che\ket})}{\prod\limits_{\alpha\in \Delta_{0,+}
}(1-q^{t
\bra \alpha,\rho_1\che \ket })}\\
&=\begin{cases}
   \prod\limits_{\alpha\in\Delta_{0,+}}\frac{\bra
   y(\lam+\rho),\alpha\che\ket}{\bra \alpha,\rho_1\che\ket}&\text{if }
   \affW_f(\lam)=\Delta_{0},\\
0&\text{otherwise.}
  \end{cases}
\end{align*}
This gives 
 the desired results.
\end{proof}

\subsection{Height and coheight of nilpotent elements}
In this subsection we  continue to assume that
the grading (\ref{eq:good-grading})
is
 Dynkin,
so that $\x=h/2$.
The largest integer $j$ such that $\fing_{j/2}\ne 0$
is called 
 the {\em height}  \cite{Pan99} of the nilpotent element $f$
and denoted by
$\on{ht}(f)$.
Because $\fing_{\theta}\subset \fing_{\on{ht}(f)/2}$,
it follows that
\begin{align*}
 \on{ht}(f)=2 \theta(\x).
\end{align*}
Define 
the {\em coheight} 
 $\on{ht}\che(f)$ of $f$
by
\begin{align}
 \on{ht}\che(f)
=2 \theta_s(\x).
\label{eq:coheight}
\end{align}

If  $\fing$ be a classical Lie algebra
then the nilpotent orbits are parameterized by certain 
partitions of $N$,
where
 $N=l +1$
(respectively $2l  +1$, $2l $, $2l $)
when $\fing$
is of type $A_{l }$
(respectively $B_{l }$, $C_{l }$,
$D_{l }$).
If $X=A$
(respectively $B$, $C$, $D$),
let $\mc{P}_X(N)$ be the set of partitions
of $N$ parameterizing the set
of nilpotent orbits for 
 $A_{l }$
(respectively $B_{l }$, $C_{l }$, $D_{l }$).
A precise description of $\mc{P}_X(N)$
\cite[Theorems 5.1.2-5.1.4]{ColMcG93}
is given as follows.
\begin{itemize}
 \item $\mc{P}_A(N)$: all partition of $N$.
\item $\mc{P}_B(N)$: partitions of $N$ such that even parts occur with
      even 
multiplicity. 
\item $\mc{P}_C(N)$: partitions of $N$ such that odd parts occur with
      even 
multiplicity. 
\item $\mc{P}_D(N)$: partitions of $N$ such that even parts occur with
      even 
multiplicity. 
\end{itemize}
Let $\mb{O}_d$ be the nilpotent orbit corresponding to 
a partition $d$
(when $\fing$ is  of type $D$ and $d$ is a very even partition there are
two orbits
$\mb{O}^I_d$  and $\mb{O}^{II}_d$ corresponding to $d$).
The following assertion can be seen from the formula for
the weighted  Dynkin diagram (see \cite[5.3]{ColMcG93}).
\begin{Pro}\label{Pro:classical}
 Let
$f$ be a nilpotent element corresponding to a partition
$(d_1,d_2,\dots,d_n)$ 
in a classical Lie algebra 
$\fing$.
\begin{enumerate}
 \item (\cite{Pan99})
\begin{enumerate}
 \item If $\fing=\mf{sl}_n$ or $\mf{sp}_{2n}$ then
$\on{ht}(f)=2(d_1-1)$.
\item If  $\fing=\mf{so}_n$ then $\on{ht}(f)=
\begin{cases}
 (d_1+d_2)-2&\text{if $d_2\geq d_1-1$},\\
2(d_1-2)&\text{if $d_2\leq d_1-2$.}
\end{cases}$
\end{enumerate}
\item 
\begin{enumerate}
 \item 
If $\fing=\mf{sp}_{2n}$ then
$\on{ht}\che(f)=
\begin{cases}
(d_1+d_2)-2&\text{if $d_2\geq d_1-1$},\\
2(d_1-2)&\text{if $d_2\leq d_1-2$.}
\end{cases}$
\item If $\fing=\mf{so}_{2n+1}$ then
$\on{ht}\che(f)=d_1-1$.
\end{enumerate}
\end{enumerate}
\end{Pro}

\subsection{Associated varieties of $G$-integrable  admissible representations}

For $q\in \N$,
set
\begin{align*}
& \Nil_q\teigi \{f\in \Nil; \on{ht}(f)<2q\}\cup \{0\}\subset \Nil,\\
& {}^L\Nil_q\teigi \{f\in \Nil; \on{ht}\che(f)<2q\}\cup \{0\}\subset \Nil.
\end{align*}
Note that
\begin{align*}
&\Nil_q=\{x\in \fing;
(\ad x)^{2q}=0\}\quad
\text{and}\quad
{}^L\Nil_q=\{x\in \fing; \pi_{\theta_s}(x)^{2q}=0\},
\end{align*}
where $\pi_{\theta_s}$
is the irreducible representation of $\fing$
with highest weight $\theta_s$.
Because 
\begin{align}
 \on{ht}(f_{\prin})=2(h_{\fing}-1),
\quad \on{ht}\che(f_{\prin})=2({}^L h\che-1),
\label{eq;height-of-principal-nil}
\end{align}
we have
\begin{align*}
 \Nil_q=\Nil\iff    q\geq h_{\fing},
\quad \text{and}
\quad 
 {}^L\Nil_q=\Nil\iff   q\geq {}^L h\che.
\end{align*}

\begin{Th}\label{Th:main-admissible}
Let $k$ be an admissible number,
 $\lam\in \Prp^k$.
Then 
 \begin{align*}
 X_{\Irr{\lam}}\cong \begin{cases}
	     \Nil_q &\text{if }k\in \A[q]
\text{ with }(q,r\che)=1,
\\ {}^L\Nil_q &\text{if }k\in {}^L\A[q].
	    \end{cases}
\end{align*}
\end{Th}

The rest of this subsection is devoted to the proof of
Theorem \ref{Th:main-admissible}.

\begin{Lem}\label{Lem:case-by-case}
Let $q\in \N$.

 \begin{enumerate}
  \item Suppose that
$\on{ht}(f)\geq 2q$.
Then  there exists a root $\alpha$ such that
$\alpha(\x)=q$.
\item Suppose that 
$\on{ht}\che(f)\geq 2q$.
Then either of the following holds:
\begin{enumerate}
 \item  There exists a short root $\alpha$ such that
$\alpha(\x)=q$.
\item There exists a long root $\alpha$ such that
$\alpha(\x)=r\che q$.
\end{enumerate}
 \end{enumerate}
\end{Lem}
\begin{proof}
 (i)
We need to show that $\fing_q\ne 0$.
If
$\on{ht}(f)$ is even,
that is,
if  $\theta(\x)\in \Z$,
then  the assertion follows from the $\mf{sl}_2$-representation theory.
If $\on{ht}(f)$ is odd,
then
it
 follows from the fact 
 \cite[Proposition 2.4]{Pan99} that
the weighted Dynkin diagram of  $f$
contains no $ 2$'s.

(ii)
We may assume that
 $\fing$ is not simply laced.

Let $\fing=\mf{so}_{2l +1}$.
We assume that
the simple roots of $\fing$ are
labeled as in \cite{ColMcG93},
so that
$\theta_s=\alpha_1+\dots +\alpha_{l }$
and
the positive short roots are of the 
form  $\alpha_i+\alpha_{i+1}\dots +\alpha_l $
with $1\leq i\leq  l $.
We have $\alpha_i(x_0)\in \{0,1/2,1\}$.
Suppose that
there is no short positive root $\alpha$ such that
$\alpha(\x)=q$.
Then
the condition $\theta_s(\x)\geq q$
implies that
there exists 
$i$ such that
$(\alpha_i+\dots +\alpha_{l })(\x)=q+1/2 $
and $\alpha_i(\x)=1$. 
Then  $\beta=\alpha_i+2(\alpha_{i+1}+\dots \alpha_{l })$
is a long root satisfying 
$\beta(\x)=2 q$.

Let $\fing=\mf{sp}_{2l }$.
We show that the condition (a) holds.
Short positive roots are
$e_i\pm e_j$ with $i<j$ in the notation of \cite{ColMcG93}.
Let $(d_1,\dots,d_n)
\in \mc{P}_C(N)$ be the 
partition  corresponding to $f$.
Let $\{h_1,\dots,h_{l }\}$ be 
numbers associated with
$(d_1,\dots,d_n)$
as in \cite[5.3]{ColMcG93},
so that
$(h_i-h_j)/2$ and
$(h_i+h_j)/2$ with $i<j$
are  the values of some positive short roots
at $\x$.
We have $h_1=d_1-1$.

First, suppose  that $d_2\leq d_1-2$.
(Then $d_1$ must be even.)
By Proposition \ref{Pro:classical},
$\theta_s(\x)=d_1-2$.
Because $h_2=h_1-2$ in this case
$(h_1+h_2)/2=d_1-2$.
It follows that 
any non-negative integer which is equals to or less than
$d_1-2$ appears as
the value of some positive short root
at $\x$.
In particular (a) holds.

Next, suppose that $d_2\geq d_1-1$.
Then,  by Proposition \ref{Pro:classical},
 $\theta_s(\x)=(d_1+d_2)/2-1$.
If $d_2=d_1-1$ 
 then
$\theta_s(\x)=d_1-3/2$.
Hence
$q\leq \theta_s(\x)$
implies that
 $q\leq d_1-2$. 
In this case $(h_1+h_3)/2=d_1-2$
and
  the assertion follows 
similarly as above.
If $d_1=d_2$ 
then 
$\theta_s(\x)=d_1-1$.
In this case $(h_1+h_2)/2=d_1-1$
and
again the assertion follows similarly.

For types $G_2$ and $F_4$
one can consult Dynkin's tables of the 
weighted Dandier diagrams
(see \cite[ch.\ 8]{ColMcG93}).
\end{proof}
\begin{Th}\label{Th:H(L(lam))-is-non-zero-if-and-only-if}
\begin{enumerate}
 \item 
Let $k\in \A[q]$
with $(q,r\che)=1$,
$\lam\in \Prp^k$.
We have
$\BRS{0}{\Irr{\lam}}\ne 0$ 
if and only if $\on{ht}(f)<2q$,
or equivalently,
$f\in \Nil_q$.
 \item 
Let $k\in {}^L\A[q]$,
$\lam\in \Prp^k$.
We have
$\BRS{0}{\Irr{\lam}}\ne 0$ 
if and only if $\on{ht}\che(f)<2 q$,
or equivalently,
$f\in {}^L\Nil_q$.

\end{enumerate}
 \end{Th}
\begin{proof}
(i)
Suppose that 
$\on{ht}(f)\geq 2q$.
By Lemma \ref{Lem:case-by-case},
there exists  $\alpha\in \Delta_+$ such that $\alpha(\x)=q$.
It follows that  $-\alpha+q\delta\in \wh\Delta_f(\lam) \backslash \Delta_0$.
Hence
$\BRS{0}{\Irr{\lam}}=0$
by 
Proposition \ref{Pro:criterion-nonzer-character}.
Conversely, suppose that 
$\on{ht}(f)< 2q$.
Then
 $\bra \alpha,D+\x\ket\geq 0$
for any $\alpha\in \wh \Delta(\lam)\cap \wh \Delta_+$,
and the equality holds
 if and only if 
$\alpha\in \Delta_{0,+}$.
Hence
$\wh{\Delta}_f(\lam)=\Delta_0$
and the assertion follows from
 Proposition \ref{Pro:criterion-nonzer-character}.
The proof of (ii) is similar.
\end{proof}

\begin{proof}[Proof of Theorem \ref{Th:main-admissible}]
Since
$X_{\Irr{\lam}}\subset \Nil$ by 
Theorem \ref{Th:Conj:Feigin-Frenkel},
the assertion follows immediately from
  Proposition \ref{Pro:condition-to-be-nonzero}
and
Theorem
\ref{Th:H(L(lam))-is-non-zero-if-and-only-if}.
\end{proof}

\subsection{Irreducibility of $\Nil_q$ and ${}^L\Nil_q$}
\begin{table}
 \begin{tabular}{|c|c|l|l|}
\hline $\fing$& $q$ & $\mb{O}_q$&exceptional  \\
\hline $\mf{sl}_{l +1}$&any &$(q,\dots,q,s) $, $0\leq s\leq q-2$
& yes\\
\hline $\mf{sp}_{2l }$&odd  &$(\underbrace{q,\dots,q}_{\text{even}},
 s) $, $0\leq s\leq q-1$,
$s$ even &yes\\
 & &$(\underbrace{q,\dots,q}_{\text{even}},q-1,s)$,
$1\leq s\leq q-1$,
$s$ even&yes if $s=q-1$\\
 &(even) &$({q,\dots,q},s)$,
$0\leq s\leq q-1$,
$s$ even&yes\\
\hline 
$\mf{so}_{2l  +1}$&odd&
$(\underbrace{q,\dots,q}_{\text{even}},s)$, $0\leq s\leq q$,
$s$ odd &yes
\\
 & &  $(\underbrace{q,\dots,q}_{\text{odd}},s,1)$, $0\leq s\leq q-1$,
$s$ odd &yes if $s=1$\\
 & (even)&  $(q+1,\underbrace{q,\dots,q}_{\text{even}})$ &yes\\
 & &  $(q+1,\underbrace{q,\dots,q}_{\text{even}},s,1)$,
$1\leq s\leq q-1$, $s$ odd&yes if $s=1$\\
 & &  $(q+1,\underbrace{q,\dots,q}_{\text{even}},q-1,s)$,
$1\leq s\leq q-1$, $s$ odd&yes if $s=q-1$\\
\hline
$\mf{so}_{2l }$&odd &
$(\underbrace{q,\dots,q}_{\text{odd}},s)$, $0\leq s\leq q$,
$s$ odd&yes if $s=q$ \\
& &
$(\underbrace{q,\dots,q}_{\text{even}},s,1)$, $0\leq s\leq q-1$,
$s$ odd& yes \\
 & even &
$(q+1,\underbrace{q,\dots,q}_{\text{even}},s)$, $0\leq s\leq q-1$,
$s$ odd& yes if $s=1$\\
& &$(q+1,\underbrace{q,\dots,q}_{\text{even}},q-1,s,1)$, $0\leq s\leq
	  q-1$,
$s$ odd& yes if $s=q-1$\\
\hline
\end{tabular}
\caption{Orbits
 $\mb{O}_q$ 
in classical Lie algebras}
\label{table:classical-principal}
\end{table}
\begin{table}
 \begin{tabular}{|c|c|l|}
\hline $\fing$& $q$ & ${}^L\mb{O}_q$ \\
\hline $\mf{sp}_{2l }$&even  &$(q,\dots,q,
 s) $, $0\leq s\leq q-1$,
$s$ even\\
 &odd &$(q+1,\underbrace{q,\dots,q}_{\text{even}},s)$,
$0\leq s\leq q-1$, $s$ even\\
 & &$(q+1,\underbrace{q,\dots,q}_{\text{even}},q-1,s)$,
$2\leq s\leq q-1$, $s$ even\\
\hline 
$\mf{so}_{2l  +1}$&any&
$(\underbrace{2q,\dots,2q}_{\text{even}},s)$, $0\leq s\leq 2q-1$,
$s$ odd
\\
 & &  $(\underbrace{2q,\dots,2q}_{\text{even}},2q-1,s,1)$, $0\leq s\leq
	  2q-1$,
$s$ odd\\
\hline
\end{tabular}
\caption{Orbits
${}^L\mb{O}_q$ in non-simply laced classical Lie
 algebras
}
\label{table:classical-coprincipal}
\end{table}
\begin{table}
\begin{center}
 \begin{tabular}{|c|c|c|c|c|}
 \hline
$q$ & $\mb{O}_q$&exceptional  &$c(\frac{p}{q})$\\
\hline 
$\geq 6$& $G_2$ &yes   &$-\frac{2 (12 p-7 q) (7 p-4 q)}{p q}$\\
\hline
$(3),4,5$&$G_2(a_1)$& no  &$-\frac{4 (6 p-7 q) (p-2 q)}{p q}$\\
\hline 
$2$ &$\widetilde{A}_1$& yes  & $63-9 p-\frac{112}{p}$
\\
\hline
$1$&0&yes &$\frac{14 (p-4)}{p}$\\
\hline 
\end{tabular}
\end{center}
\caption{Orbits $\mb{O}_q$ in type  $G_2$}
\label{Oq-for-G_2}
\end{table}
\begin{table}
\begin{center}
 \begin{tabular}{|c|c|c|c|}
 \hline
$q$ & ${}^L\mb{O}_q$ &$c(\frac{p}{3q})$\\
\hline 
$\geq 4$& $G_2$  &$-\frac{2 (7 p-12 q) (4 p-7 q)}{p q}$\\
\hline
$2,3$&$G_2(a_1)$ & $-\frac{4 (2 p-7 q) (p-6 q)}{p q}$\\
\hline 
$1$ &${A}_1$&$-\frac{2 (p-12) (p-7)}{p}$\\
\hline
\end{tabular}
\end{center}
\caption{Orbits ${}^L\mb{O}_q$ in type   $G_2$}
\label{Oqche-for-G2}
\end{table}
\begin{table}
\begin{center}
 \begin{tabular}{|c|c|c|c|c|}
 \hline 
$q$&$\mb{O}_q$&exceptional &$c(\frac{p}{q})$\\
\hline 
$\geq 12$ & $F_4$&yes  &$-\frac{4 (18 p-13 q) (13 p-9 q)}{p q}$\\
\hline
$(8),9,(10),11$&$F_4(a_1)$&no &$1062-\frac{600 p}{q}-\frac{468 q}{p}$\\
\hline
$(6),7$&$F_4(a_2)$&no  &$632-\frac{216 p}{7}-\frac{3276}{p}$\\
\hline
$(4),5$&$F_4(a_3)$&no &$-\frac{12 (p-15) (6 p-65)}{5 p}$\\
\hline
$3$&$\widetilde{A}_2+A_1$&yes  &$316-18 p-\frac{1404}{p}$
\\\hline
$(2)$&$A_1+\widetilde{A}_1$&yes  &
\\\hline
$1$&$0$&yes &$\frac{52 (p-9)}{p}$\\
\hline
\end{tabular}
\end{center}
\caption{Orbits $\mb{O}_q$ in type $F_4$}
\label{Oq-for-F4}
\end{table}
\begin{table}
\begin{center}
 \begin{tabular}{|c|c|c|c|}
 \hline 
$q$& ${}^L\mb{O}_q$ &$c(\frac{p}{2q})$\\
\hline 
$\geq 9$ & $F_4$ &$-\frac{4 (13 p-18 q) (9 p-13 q)}{p q}$\\
\hline
$6,7,8$&$F_4(a_1)$ &$1062-\frac{300 p}{q}-\frac{936 q}{p}$\\
\hline
$5$&$F_4(a_2)$ &$632-\frac{108 p}{5}-\frac{4680}{p}$\\
\hline
$4$&$B_3$&$562-21 p-\frac{3744}{p}$\\
\hline
$3$&$F_4(a_3)$ &$-\frac{12 (p-18) (p-13)}{p}$
\\\hline
$2$&$A_2+\widetilde{A_1}$ &$-\frac{9 (p-16) (p-13)}{p}$\\
\hline 
$1$&$A_1$ &$-\frac{3 (p-24) (p-13)}{p}$\\
\hline
\end{tabular}
\end{center}
\caption{Orbits ${}^L\mb{O}_q$ in type  $F_4$}
\label{Oqche-for-F4}
 \end{table}

\begin{table}
\begin{center}
  \begin{tabular}{|c|c|c|c|c|c|}
 \hline 
$q$&$\mb{O}_q$&exceptional  &$c(\frac{p}{q})$\\
\hline 
$\geq 12$ & $E_6$&yes &$-\frac{6 (12 p-13 q) (13 p-12 q)}{p q}$
\\
\hline
$9,10,11$&$E_6(a_1)$&no   &$-\frac{8 (9 p-13 q) (7 p-9 q)}{p q}$
\\
\hline
$8$&$D_5$&yes  &$-45 p+1162-\frac{7488}{p}$\\
\hline
$6,7$&$E_6(a_3)$&no  &$-\frac{36 (6 p-13 q) (p-2 q)}{p q}$\\
\hline
$5$&$A_4+A_1$&yes  &$-\frac{2 (7 p-90) (9 p-130)}{5 p}$
\\\hline
$4$&$D_4(a_1)$&no &$-18 p+524-\frac{3744}{p}$
\\\hline
$3$&$2A_2+A_1$&yes  &$-\frac{18 (p-13) (p-12)}{p}$
\\\hline
$2$&$3 A_1$&yes &$-9 p+263-\frac{1872}{p}$
\\\hline
$1$&$0$&yes  &$\frac{78 (p-12)}{p}$
\\\hline
\end{tabular}
\end{center}
\caption{Orbits $\mb{O}_q$ in type  $E_6$}
\label{Oq-for-E6}
\end{table}
\begin{table}
\begin{center}
  \begin{tabular}{|c|c|c|c|c|c|}
 \hline 
$q$& $\mb{O}_q$&exceptional  &$c(\frac{p}{q})$
 \\
\hline 
$\geq 18$ & $E_7$&yes &$-\frac{7 (18 p-19 q) (19 p-18 q)}{p q}$
\\
\hline
$14,15,16,17$&$E_7(a_1)$&no  &$-\frac{9 (14 p-19 q) (11 p-14 q)}{p q}$
\\
\hline
$12,13$&$E_7(a_2)$&no  &$-\frac{(106 p-171 q) (9 p-14 q)}{p q}$
\\
\hline
$10,11$&$E_7(a_3)$&no  &$-\frac{666 p}{q}+2521-\frac{2394 q}{p}$
\\
\hline
$9$&$E_6(a_1)$&no  &$56 p+2199-\frac{21546}{p}$
\\\hline
$8$&$E_7(a_4)$&no  &$-\frac{189 p}{4}+1901-\frac{19152}{p}$
\\\hline
$7$&$A_6$&yes  &$-\frac{(p-18) (48 p-931)}{p}$
\\\hline
$6$&$E_7(a_5)$&no  &$-\frac{3 (p-19) (13 p-252)}{p}$ 
\\\hline
$5$&$A_4+A_2$&yes  &$\frac{9}{5} \left(-16 p+655-\frac{6650}{p}\right)$
\\\hline
$4$&$A_3+A_2+A_1$&yes  &$-\frac{3 (3 p-56) (5 p-114)}{2 p}$
\\\hline
$3$&$2A_2+A_1$&yes  &$-18 p+721-\frac{7182}{p}$
\\\hline
$2$&$4A_1$&yes  &$-\frac{12 (p-21) (p-19)}{p}$
\\\hline
$1$&$0$&yes&$\frac{133 (p-18)}{p}$
\\\hline
\end{tabular}
\end{center}
\caption{Orbits $\mb{O}_q$ in type $E_7$}
\label{Oq-for-E7}
\end{table}
\begin{table}
  \begin{tabular}{|c|c|c|c|c|c|}
 \hline 
$q$&$\mb{O}_q$&exceptional &$c(\frac{p}{q})$ \\
\hline 
$\geq 30$ & $E_8$&yes &$-\frac{8 (30 p-31 q) (31 p-30 q)}{p q}$\\
\hline
$24,25,26,27,28,29$&$E_8(a_1)$&no  &$-\frac{10 (24 p-31 q) (19 p-24 q)}{p q}$\\
\hline
$20,21,22,23$&$E_8(a_2)$&no &$-\frac{12 (20 p-31 q) (13 p-20 q)}{p q}$\\
\hline
$18,19$&$E_8(a_3)$&no &$-\frac{2400 p}{q}+8438-\frac{7440 q}{p}$\\
\hline
$15,16,17$&$E_8(a_4)$&no  &$-\frac{16 (15 p-31 q) (7 p-15 q)}{p q}$
\\\hline
$14$&$E_8(b_4)$&no &$-\frac{696 p}{7}+6426-\frac{104160}{p}$
\\\hline
$12,13$&$E_8(a_5)$&no  &$-\frac{4 (23 p-60 q) (12 p-31 q)}{p q}$
\\\hline
$10,11$&$E_8(a_6)$&no  &$-\frac{24 (10 p-31 q) (3 p-10 q)}{p q}$
\\\hline
$9$&$E_8(b_6)$&no &$-\frac{4 (4 p-135) (11 p-372)}{3 p}$
\\\hline
$8$&$A_7$&yes &$-\frac{3 (p-31) (21 p-640)}{p}$ 
\\\hline
$7$&$A_6+A_1$&yes  &$-\frac{342 p}{7}+3192-\frac{52080}{p}$
\\\hline
$6$&$E_8(a_7)$&no  &$-\frac{40 (p-36) (p-31)}{p}$
\\\hline
$5$&$A_4+A_3$&yes  &$-\frac{12 (p-31) (3 p-100)}{p}$
\\\hline
$4$&$2 A_3$&yes  &$-\frac{30 (p-32) (p-31)}{p}$
\\\hline
$3$&$2 A_2+2A_1$&yes  &$-\frac{20 (p-36) (p-31)}{p}$
\\\hline
$2$&$4 A_1$&yes  &$-\frac{12 (p-40) (p-31)}{p}$
\\\hline
$1$&$0$&yes &$\frac{248 (p-30)}{p}$
\\\hline
\end{tabular}
\caption{Orbits $\mb{O}_q$ in type  $E_8$}
\label{table:E8}
\end{table}

\begin{Th}\label{Th:Main-orbit1}
Let $q\in \N$.
 \begin{enumerate}
  \item (\cite{Geo04}\footnote{Unfortunately there are some typos in the 
$E_7$ table in \cite{Geo04}.},
not necessarily $(q,r\che)=1$)
There exists a unique nilpotent orbit
$\mb{O}_q$ such that 
$\Nil_q=\overline{\mb{O}_q}$.
If $q\geq h_{\fing}$ 
then $\mb{O}_q=\mb{O}_{\prin}$,
the principal nilpotent orbit.
The orbits $\mb{O}_q$ for 
$q< h_{\fing}$ are listed in Tables 
 \ref{table:classical-principal}, 
\ref{Oq-for-G_2},
\ref{Oq-for-F4},
\ref{Oq-for-E6},
\ref{Oq-for-E7} and  \ref{table:E8}.
  \item 
(\cite{Geo04} for type $G_2$ and $F_4$)
There exists a unique nilpotent orbit
${}^L\mb{O}_q$ such that 
${}^L\Nil_q=\overline{{}^L \mb{O}_q}$.
	If  $q\geq {}^L h\che$
then ${}^L \mb{O}_q=\mb{O}_{\prin}$.
The orbits ${}^L \mb{O}_q$ for $q<{}^L h\che$
for non-simply laced Lie algebras are listed in
Tables \ref{table:classical-coprincipal},
\ref{Oqche-for-G2} and
\ref{Oqche-for-F4}.
 \end{enumerate}
\end{Th}
\begin{proof}
(ii) Let $X=B_{l }$ or $C_{l }$,
and let $\fing$ be a Lie algebra of type $X$,
 $d\in \mc{P}_X(N)$.
 From Proposition \ref{Pro:classical}
it follows that
$\mb{O}_d\subset {}^L\Nil_q$ 
if and only if
$d$ is dominated by
$\tilde{d}_q$,
where
\begin{align*}
 \tilde{d}_q=\begin{cases}
	      (q+1,q,q,\dots,q,s),\ 0\leq s\leq q-1&
\text{for $\mf{sp}_{2n}$},\\
(2q,2q,\dots,2q,s), 0\leq s\leq 2q-1
&\text{for $\fing=\mf{so}_{2n+1}$}.
	     \end{cases}
\end{align*}
It is known that  there exists a unique  maximal partition $d_q$
 in $\mc{P}_X(N)$
dominated by $\tilde{d}_q$, see \cite[Lemma 6.3.3.]{ColMcG93}.
The partition $d_q$ is 
called {\em the $X$-collapse} of $\tilde{d}_q$.
It follows that ${}^L\mb{O}_q=\mb{O}_{d_q}$
gives that
${}^L\Nil_q={}^L\overline{\mb{O}}_{q}$,

For type $G_2$ and $G_4$
the representation 
$\pi_{\theta_s}$ is
the minimal representation of $\fing$.
Hence the assertion has been proved in  \cite[4.1]{Geo04}.
\end{proof}
%

\subsection{The $C_2$-cofiniteness of $W$-algebras}
For an admissible number $k$,
set
\begin{align}
 \mb{O}[k]\teigi \begin{cases}
       \mb{O}_q&\text{if }k\in \A[q] \text{ with $(q,r\che)=1$},\\
{}^L\mb{O}_q &\text{if }k\in {}^L\A[q].
      \end{cases}
\label{eq:O[k]}
\end{align}
Then
  Theorem \ref{Th:main-admissible}
and Theorem
\ref{Th:Main-orbit1} read as
\begin{align}
X_{\Irr{\lam}}=\overline{\mb{O}[k]}\label{eq:varity-is-the-clousre}
\end{align}
for  $\lam\in \Prp^k$.
Hence the following assertion follows
immediately from Theorem \ref{Th:C2-cofinitness-if-variety-is-the-orbit}.
\begin{Th}\label{Th:C_2-cofiniteness-of-modules-over-W-algebras}
Let $k$ be  an admissible number.
\begin{enumerate}
 \item 
 For
 $\lam\in \Prp^k$,
the $\Wg{k}$-module
$\BRS{0}{\Irr{\lam}}$ is (nonzero and) $C_2$-cofinite
if and only if
$f\in \mb{O}[k]$.
\item
The vertex algebra $\BRS{0}{V_k(\fing)}$ 
is (nonzero and) $C_2$-cofinite
 if and only if
$f\in \mb{O}[k]$.
In particular,
the simple $W$-algebra
$\W_k(\fing,f)$ is $C_2$-cofinite
if $f\in \mb{O}[k]$.
\end{enumerate}
\end{Th}

In  Tables \ref{Oq-for-G_2}--\ref{table:E8}
for exceptional type Lie algebras 
we give 
the explicit formulas 
of the central charge
\eqref{eq:cc}
of
$\W_{k}(\fing,f)$
for  an admissible number $k$ and $f\in \mb{O}[k]$
(The central charge of $\W^k(\fing,f)$
with $k+h\che=p/q$ is denoted by
$c(p/q)$.)

\begin{Rem}
Suppose that $l (=\rank \fing)=2$.
Then
 any nilpotent element of $\fing$ belongs to $\mb{O}[k]$
for some admissible number $k$.
Hence this case 
 the $W$-algebra
$\W^k(\fing,f)$ associated with any nilpotent element
$f$ is $C_2$-cofinite for some value of $k$.
\end{Rem}

\subsection{The $C_2$-cofiniteness of exceptional $W$-algebras}
Recall \cite{KacWak08,ElaKacVin08}
that
a pair
$(q,f)$ of a positive integer $q$
and a nilpotent element $f$ of $\fing$
is called {\em exceptional}
if the following conditions are satisfied.
\begin{itemize}
 \item 
$q$ is equal to or greater than
the maximum of the Coxeter numbers of the simple factors
of the minimal Levi subalgebra containing $f$,
\item 
$\dim \fing^f=\dim \fing^{\sigma_q}$,
where $\sigma_q$ is the automorphism of $\fing$
such that $\sigma_q(x_{\alpha})=e^{\frac{2\pi \sqrt{-1}}{q}\on{ht}(\alpha)}x_{\alpha}$
for a root vector $x_{\alpha}$.
\end{itemize}
Exceptional pairs are classified in \cite{KacWak08} for
$\fing=\mf{sl}_n$
and in \cite{ElaKacVin08} for a general $\fing$.

The simple $W$-algebra
$\W_k(\fing,f)$  are called {\em exceptional}
if $(q,f)$ is an exceptional pair
and  $k\in \A[q]$ with $(q,r\che)=1$.
(The latter condition is a technical one.)
In the case that
$f$ is a principal nilpotent element
exceptional $W$-algebras are  the minimal series
$W$-algebras associated with principal nilpotent elements
discovered by Frenkel, Kac and Wakimoto \cite{FKW92}.

In \cite{KacWak08} it was conjectured that
the conformal field theories associated with
the exceptional $W$-algebras
are rational.
In the language of vertex operator algebras
this amounts to showing 
the exceptional $W$-algebras 
are rational and $C_2$-cofinite (cf.\ \cite{Zhu96}).

\begin{Th}\label{Th:Kac-Wakimoto}
The following are equivalent:
\begin{enumerate}
 \item $(q,f)$ is an exceptional pair.
\item $f\in \mb{O}_q$ and $f$ is of principal type,
that is,
$f$ is a principal nilpotent element in a Levi subalgebra of $\fing$.
\end{enumerate}
\end{Th}
\begin{proof}
 The assertion follows from the classification \cite{ElaKacVin08}
of exceptional pairs and 
that of $\mb{O}_q$.
\end{proof}
In Tables \ref{table:classical-principal},
\ref{Oq-for-G_2},
\ref{Oq-for-F4},
\ref{Oq-for-E6},
\ref{Oq-for-E7} and
\ref{table:E8},
we indicate 
whether the pair of $q$
and the nilpotent  is exceptional or not.

The following assertion
follows immediately
from 
 Theorem \ref{Th:C_2-cofiniteness-of-modules-over-W-algebras}
and Theorem
\ref{Th:Kac-Wakimoto}.
\begin{Th}\label{Th:exceptionals-are-C2}
All the 
exceptional $W$-algebras are
 $C_2$-cofinite. 
\end{Th}

\end{document}